\documentclass[11pt]{amsart}

\usepackage[top=1in, bottom=1in, left=1in, right=1in]{geometry}

\usepackage{graphicx}
\usepackage[dvipsnames]{xcolor}
\usepackage{amsfonts,amsmath,amssymb,amsthm,bbm,centernot}
\usepackage{mathrsfs}
\usepackage{enumerate}
\usepackage{accsupp}
\usepackage{hyperref}
\hypersetup{colorlinks=true,linkcolor=MidnightBlue,citecolor=OliveGreen,urlcolor=MidnightBlue}
\usepackage[normalem]{ulem}
\newcommand{\urltilde}{\BeginAccSupp{method=hex,ActualText=7E}\ensuremath{\mathord{\sim}}\EndAccSupp{}}
\usepackage{mathtools}
\usepackage[scr=dutchcal]{mathalpha}
\usepackage{multirow}

\newtheorem{theorem}{Theorem}[section]
\newtheorem{proposition}[theorem]{Proposition}
\newtheorem{lemma}[theorem]{Lemma}
\newtheorem{corollary}[theorem]{Corollary}

\newtheorem{asm}{Assumption}

\theoremstyle{remark}
\newtheorem{remark}[theorem]{Remark}

\theoremstyle{definition}
\newtheorem{definition}{Definition}[section]

\newcommand{\N}{\mathbb N}

\newcommand{\R}{\mathbb R}
\newcommand{\Z}{\mathbb Z}

\newcommand{\cal}{\mathcal}
\renewcommand{\epsilon}{\varepsilon}

\renewcommand{\d}{\delta}
\newcommand{\e}{\epsilon}

\newcommand{\B}{\mathfrak{B}}
\newcommand{\D}{\mathfrak{D}}

\newcommand{\fl}[1]{\lfloor #1 \rfloor}

\newcommand{\ind}[1]{ \mathbbm{1}_{\{ #1 \}} }

\renewcommand{\dim}{\mathfrak{d}}

\DeclareMathOperator{\Var}{Var}

\let\emptyset\varnothing

\title[Self-interacting Random Walks]{Functional limit theorems for self-interacting random walks}

\author{Elena Kosygina}
\address{Elena Kosygina\\One Bernard Baruch Way \\ Department of Mathematics, Box B6-230 \\ Baruch College \\ New York, NY 10010 \\ USA}
\email{elena.kosygina@baruch.cuny.edu}
\urladdr{http://www.baruch.cuny.edu/math/elenak/}
\thanks{E.\,Kosygina was partially supported by the Simons Foundation through a Collaboration Grant for Mathematicians \#523625.}

\author{Laure Mar\^ech\'e}
\address{Laure Mar\^ech\'e\\Institut de Recherche Math\'ematique Avanc\'ee,
UMR 7501 Universit\'e de Strasbourg et CNRS,
7 rue Ren\'e Descartes, 67000 Strasbourg, France}
\email{laure.mareche@math.unistra.fr}
\urladdr{https://irma.math.unistra.fr/\urltilde mareche/}

\author{Thomas Mountford}
\address{Thomas Mountford\\\'Ecole Polytechnique F\'ed\'erale de Lausanne\\Department of Mathematics\\EPFL SB MATH PRST\\
MA B1 517 (B\^atiment MA)
Station 8
CH-1015 Lausanne\\
Switzerland}
\email{thomas.mountford@epfl.ch}
\urladdr{http://people.epfl.ch/thomas.mountford}

\author{Jonathon Peterson}
\address{Jonathon Peterson\\Purdue University\\Department of Mathematics\\150 N University Street\\West Lafayette, IN  47907\\USA}
\email{peterson@purdue.edu}
\urladdr{https://www.math.purdue.edu/\urltilde peterson/}
\thanks{J.\,Peterson was partially supported by the Simons Foundation through a Collaboration Grant for Mathematicians \#635064. Part of this work was completed while J.\,Peterson was a David Parkin Visiting Professor at the University of Bath. Support of the University of Bath is gratefully acknowledged.}



\begin{document}

\begin{abstract}
  We consider a wide class of nearest neighbor self-interacting random walks on $\Z$ and show that, under mild conditions, their scaling limits at the process level can be uniquely identified and constructed directly from the augmented joint Ray-Knight theorem for their local times, i.e., the joint convergence of rescaled local time ``curves'' of the walks stopped at certain stopping times together with the merge and absorption times of these ``curves'' to a concatenation of coalescing reflected-absorbed diffusions and their merge and absorption times. We apply our results to polynomially self-repelling walks introduced and studied in the mid-90s by B.\,T\'oth, \cite{tGRK}, and to ``true'' self-avoiding walks with bond repulsion, \cite{tTSAW}. The first application solves a problem that has been open since T\'oth's work: it gives a construction and first properties of a new family of processes, polynomially self-repelling motions, as scaling limits of polynomially self-repelling walks. These processes were simultaneously and independently constructed from flows of correlated powers of Bessel processes by Y.\,Yu, \cite{Y25+}. The second application strengthens an earlier result, \cite{KP25+}, on convergence of rescaled ``true'' self-avoiding walks to the ``true'' self-repelling motion of B.\,T\'oth and W.\,Werner, \cite{TW98}, and streamlines its proof. Our results can be thought of as an ``inversion'' of the augmented joint Ray-Knight theorem to get a functional limit theorem for self-interacting random walks.
\end{abstract}

\maketitle

\section{Introduction}\label{intro}

In the 1980s, motivated by the study of polymer chains with
  excluded-volume interactions, several self-interacting non-Markovian
  random processes were introduced and studied in the physics
  literature (see, for instance, \cite{APP83, PP87}). Mathematically
  tractable models on $\Z$ were proposed and investigated by T\'oth,
  \cite{tTSAWGBR,tTSAW,tGRK,tLTWRRW}. These serve as the main examples
  and motivation for the present work, even though the scope of our
  general results is broader and not restricted to these specific
  models. Our main interest is in establishing functional limit
  theorems for self-interacting nearest-neighbor discrete time
  processes $(X_n)_{n\ge 0}$ on $\Z$ which are recurrent, in the sense
  that with probability one every $x\in\Z$ is visited infinitely many
  times, and have site or edge local times satisfying certain limit
  theorems. Before giving precise assumptions and statements of our
  results we describe several classes of T\'oth's walks as well as known results
  and approaches to constructing scaling limits for these walks.

\subsection{The self-interacting random walks of T\'oth}\label{sec:SIRWfamily}
 The walks which are the motivating examples for our results are nearest-neighbor motions on
$\Z$ defined via a fixed monotone function $w: \Z_+ \to (0,\infty)$
which we will refer to as the weight function.  The walk
$(X_n)_{n\geq 0}$ starts with $X_0=0$ and has
$X_n-X_{n-1} \in \{-1,1\}$ chosen so that the probability the random
walk crosses an edge is proportional to the weight function $w$
evaluated at the number of times that edge has been crossed so
far. That is,
\begin{align}
 &P\left(X_{n+1} = X_n + 1 \mid \mathcal{F}_n \right) 
 = 1-  P\left(X_{n+1} = X_n - 1 \mid \mathcal{F}_n \right)
= \frac{w\left( \ell_n(X_n) \right)}{w\left( \ell_n(X_n) \right) + w\left( \ell_n(X_n-1)\right)}
, \label{sirw}
\end{align}
where $\mathcal{F}_n = \sigma(X_i, \, i\leq n)$ is the natural filtration,
and
$\ell_n(k) = \sum_{j=0}^{n-1} \ind{\{X_j, X_{j+1}\} = \{k,k+1\}}$ is the number of crossings of the undirected edge $\{k,k+1\}$ by time $n$.

Note that the monotonicity of $w$ entails that the random walks $(X_n)_{n\geq 0}$ constructed in this way are either self-attracting or self-repelling depending on whether the weight function is non-decreasing or non-increasing.

T\'oth studied the following families of random walks from this
general class, and, in particular, identified the appropriate
scaling exponents in each case. 
\begin{enumerate}[--]
 \item ``True'' self-avoiding walk (TSAW): $w(j) = e^{-\beta j}$ for some $\beta>0$.  The sequence  $(n^{-2/3}X_n)_{n\ge 1}$ is tight, \cite{tTSAW}.
 \item ``True'' self-avoiding walk with generalized bond repulsion: $w(j) = e^{-\beta j^\kappa}$ for some $\beta>0$ and $\kappa \in (0,1)$. The sequence  $(n^{-(\kappa+1)/(\kappa+2)}X_n)_{n\ge 1}$ is tight, \cite{tTSAWGBR}.
 \item Walks defined by weight functions of the form 
\begin{equation}
  \label{w}
  \frac{1}{w(j)}=j^\alpha\left(1+\frac{2B}{j}+O\left(\frac{1}{j^{2}}\right)\right)\quad\text{as $j\to\infty$ for some $\alpha>-1$ and $B \in \R$},
\end{equation}
 which he classified further into the following three cases, \cite{tGRK}: 
 \begin{enumerate}[(i)]
  \item $\alpha > 0$: polynomially self-repelling walk (PSRW); 
  \item $\alpha=0$: asymptotically free walk; 
  \item $\alpha \in (-1,0)$: weakly reinforced walk.  
 \end{enumerate}
\end{enumerate}
In cases (i) and (ii) the correct scaling was diffusive, while case
(iii) required scaling by $n^{(1+\alpha)/(2+\alpha)}$. We note that
$\alpha=-1$ is a boundary case for recurrence, \cite[Theorem 3.2]{Dav90}. A
well-studied representative of this boundary case is a linearly
reinforced random walk: $w(j)=w(0)+j$, $w(0)>0$, see \cite{Pem07}.

T\'oth analyzed these processes via the directed  edge local times. That is, if
\begin{equation}
  \label{upcr}
    {\mathcal E}^+(n,k)=\sum_{j=0}^{n-1}\ind{X_j=k, X_{j+1}=k+1} ;\quad{\mathcal E}^-(n,k)=\sum_{j=0}^{n-1} \ind{X_j=k, X_{j+1}=k-1},
\end{equation}
are the number of times the walk has stepped to the right or left,
respectively, from $k$ by time $n$, then for certain inverse local
times $\tau$ the family of directed edge local times
$(\mathcal{E}^\pm(\tau,k))_{k\in\Z}$ has a Markovian structure that
can be used to analyze the random walk (see Appendix \ref{sec:Polya}).
The key to T\'oth's method was to prove scaling limits for the
associated Markov chains of directed edge local times. 
These scaling limits may be called
``generalized Ray-Knight theorems'' because they
described the limiting distribution of the (rescaled) local time
profiles of the walk when stopped at certain inverse local times. This
is similar to how the classical Ray-Knight theorems describe the
distribution of the local time profile of a Brownian motion stopped at
inverse local times.  Using these generalized Ray-Knight theorems,
T\'oth obtained a number of asymptotic
results for the position $X_n$ of the walk.  As mentioned above,
T\'oth was able to identify
an explicit scaling exponent $a \in (0,1)$ such that the sequence
  $(n^{-a}X_n)_{n\ge 0}$ is tight. In addition, in most cases he also showed that there
was a limiting law for $n^{-a}X_{\fl{\theta n}}$ where
$\theta \sim \text{Exp}(\lambda)$ and so characterized the limiting
distribution of $n^{-a}X_n$, if it exists.  These results
naturally raise the question whether the suitably renormalized
  processes satisfy functional limit theorems.  They also
suggested the possibility that the limiting
process could be identified via the generalized Ray-Knight theorems
(or, going further, that these theorems with minimal auxiliary
hypotheses could be used to prove functional limit laws for the walk).
In some sense, this is asking if one can ``invert'' the generalized
Ray-Knight theorems.

We now discuss two recent works which are partial responses to these
questions.  The first, \cite{KMP23}, considered both
asymptotically free and polynomially self-repelling processes
mentioned above.  T\'oth had shown that the limits in generalized Ray-Knight theorems \cite[Theorems 1A and 1B]{tGRK}
were the same as in known Ray-Knight
theorems for (multiples of) Brownian motions perturbed at
extrema (BMPEs) \cite{cpyBetaPBM}. Recall that for any $p,q\in (-\infty,1)$, a
$(p,q)$-BMPE is a process $W^{p,q}(\cdot)$ that is the pathwise unique
solution \cite{cdPUPBM,pwPBM} to the functional equation
\begin{equation}\label{BMPE-def}
 W^{p,q}(t) = B(t) + p \sup_{s\leq t} W^{p,q}(s)  + q \inf_{s\leq t} W^{p,q}(s),\quad t\ge 0,
\end{equation}
where $B(\cdot)$ is a standard Brownian motion.

The results of \cite{KMP23} showed that while asymptotically free
  and polynomially self-repelling walks can be renormalized to have the same generalized
  Ray-Knight theorems, the distribution of the rescaled paths of the
  walk must be very different.  In particular, it was argued that asymptotically free walks
approximately satisfy a discrete version of \eqref{BMPE-def}: the walk
is equal to a martingale, plus a linear combination of its running
minimum and running maximum, plus an error term that vanishes in the
scaling limit.  This approach could not be extended to the PSR case
though and, somewhat surprisingly, \cite[Theorem 1.5]{KMP23}
demonstrated that PSRWs cannot converge to a
multiple of any BMPE.  This negative result revealed
that generalized Ray-Knight theorems do not
characterize the scaling limit of the path of a walk, and this
became a starting point for one of the questions addressed here.

The second recent work motivating our results, \cite{KP25+}, showed
that the TSAW converges to the ``true'' self-repelling motion (TSRM)
of T\'oth and Werner, \cite{TW98}. Convergence to (a stationary version of) the TSRM had been
proved by \cite{NR06} for a stationary version of the related discrete
motion introduced in \cite{TW98} but not for the TSAW of \cite{tTSAW}.
The proof in \cite{KP25+} is similar to the results in this paper
  as it uses ``augmented joint Ray-Knight theorems'' to prove a
  functional limit law for the walk.  However, the proof of the
functional limit law in \cite{KP25+} used in several crucial ways that
the candidate limiting process had already been constructed in
\cite{TW98} and that its various properties were known.  In contrast,
the main result in our paper applies even in cases where there is no
known candidate limiting process.

Taking the lead from these papers, we establish a more general
framework for the convergence of processes which includes many of
  the random walks in the above families (see Theorem
\ref{thm:main_X}).
We show that while generalized Ray-Knight theorems do not suffice to
establish convergence at the process level, a form of joint
convergence of the local time paths, which we call an augmented joint
Ray-Knight theorem (see Definition \ref{def:JRK}), does suffice, under
mild conditions, for both tightness and convergence of finite
dimensional distributions. This gives at once the existence and a
construction of the limiting process. Our approach builds on that of
\cite{KP25+} but represents a clear advance in that it applies to a
much wider class; in particular, it covers walks for which no
candidate limit has yet been identified.  Moreover, the current paper
gives stronger results on the local time process. The main result in
\cite{KP25+} gave the joint convergence of the rescaled path and the
local time at the current location of the walk, whereas in this paper
we show joint convergence of the rescaled path of the walk together
with the entire local time profile (see Theorem
\ref{thm:main_XL}). Moreover, in Theorem ~\ref{thm:main_JRK} we prove
that the local times of the limiting process have the appropriate
joint Ray-Knight theorem.

In addition to streamlining and strengthening the results in
\cite{KP25+}, we treat the PSRW directly and prove its convergence.
Our work also fully resolves the counterexample given in \cite{KMP23}.
That article only showed that for PSRWs, under the natural nontrivial
renormalization, the renormalized walks do not converge to the
``natural'' process  suggested by the
  generalized Ray-Knight theorems in \cite{tGRK}; it did not show
that the renormalized walks converge at all, nor that a limit, if it
existed, would have well-behaved local times satisfying the joint
Ray-Knight theorem.  This paper, in showing all three of these
results, gives us a natural example of two distinct (in law) processes
(see Theorems \ref{thm:main_JRK} and \ref{main_psr}) which have the
same Ray-Knight theorems and the same one dimensional distributions;
see Remark \ref{rem:PSRMprop} for a more precise statement.

In addition, the limit process of PSRWs, a polynomially
self-repelling motion (PSRM), is of interest in its own right\footnote{We adopt the name \emph{polynomially self-repelling motion} from \cite{Y25+} where this process is constructed as the analog of the T\'{o}th--Werner TSRM \cite{TW98} in the
polynomial case.}.  As
heuristically discussed in Appendix~\ref{psr_sde} it has a driving
mechanism similar to the one T\'oth and Werner described for the
  TSRM, \cite{TW98}.  This heuristic equation shares common features
with the processes introduced by \cite{LST21} and would be a rich area
of future study.

A further advantage of our approach is that, if one is primarily
interested in the construction of a limit process,
it bypasses many of the technical issues
that arise  when one attempts to first define a
candidate limiting process by mimicking the construction of TSRM in
\cite{TW98}.\footnote{For a modern approach and extensions see
  \cite{SSS17, Mareche2026}.}  The analysis that T\'oth had done in
proving the generalized Ray-Knight theorem for the TSAW suggested a
particular structure for the joint Ray-Knight theorem of the
  limiting process, namely that of coalescing, reflected/absorbed Brownian
motions. T\'oth and Werner first constructed a family of such
coalescing reflected/absorbed Brownian motions starting from every
point in the upper half plane (closely related to what is now called the Brownian web), and then used this
family of curves to construct the TSRM.
    The scaling limit for the TSAW in \cite{KP25+} was completed by
    first proving the appropriate joint generalized Ray-Knight theorem
    for the TSAW and then, via a combination of general arguments and
    specific known properties of the TSRM, using the joint generalized
    Ray-Knight theorem to prove a functional limit law.  This
    suggests the following approach to proving scaling limits for
    other self-interacting random walks such as the polynomially
    self-repelling walks.
\begin{description}
\item[Step 1] Prove joint generalized Ray-Knight theorems for the
  random walk.
\item[Step 2] Construct a family of merging, reflected, absorbed
  curves (generalizing the Brownian web) that are consistent with the
  joint Ray-Knight theorems for the walk. Use this web of curves to
  construct a candidate limiting process.
 \item[Step 3] Adapt the arguments in \cite{KP25+} to prove a functional limit law. 
\end{description}
Our main result shows that the second step, the construction of a
candidate limiting process through a Brownian-web-like family of
curves, can be avoided altogether, and that the third step can be
modified so that it does not require an a priori candidate. This
simplifies the proof of scaling limits not only for the PSRWs but for
any walk with no known candidate limit having the correct joint
Ray-Knight theorems. For the PSRW, steps 1 and 2 have been completed
in recent work \cite{Y25+}.  Our proof of the functional limit
  law for the PSRW, however, uses from \cite{Y25+} only the joint
  Ray-Knight theorem, together with basic properties of the SDE
  \eqref{WNrep} describing the limiting curves.

\subsection{Organization of the paper}

The paper unfolds as follows.  In the next section we give necessary
definitions and notations for the article and introduce the
assumptions needed for our random walks.  In particular we introduce
``augmented joint Ray-Knight theorems", the multi-curve extension of
``generalized Ray-Knight theorems'' of T\'oth needed to characterize
limits.  We then state the main results of our
work. Section~\ref{tightness} demonstrates that under our assumptions,
augmented joint Ray-Knight theorems automatically ensure tightness.
This is exploited in Section~\ref{sec:fdd} to establish our general
results which give functional limit theorems for the walk as a
consequence of the augmented joint Ray-Knight theorems,
Theorems~\ref{thm:main_X}--\ref{thm:main_JRK}.  The final part (before
the Appendices) is devoted to establishing Theorem \ref{main_psr}, the
functional limit theorem for the PSRW jointly with its local time
processes.  This resolves questions raised in \cite{KMP23} by showing
that the limit process for the PSRW has a continuous local time that
satisfies exactly the same generalized Ray-Knight theorems as a
multiple of a BMPE but is nonetheless distinct from it, as previously discussed.

\section{Main results}\label{main}
In this section we give precise statements of our results and a
  more detailed discussion of their novelty. Before presenting our
  main theorems, we define what we mean by augmented Ray-Knight
  theorems, joint Ray-Knight theorems, and augmented joint Ray-Knight
  theorems for a random walk. 

The classical Ray-Knight theorems for Brownian motion describe the
distribution of the local time profile of a Brownian motion when
stopped at certain inverse local times.  While our results are
motivated by self-interacting random walks they apply more generally,
and we will refer to any nearest-neighbor stochastic process
$(X_n)_{n\geq 0}$ on $\Z$ that visits every site infinitely many
times, a.s., as a recurrent random walk.  By a Ray-Knight theorem for
a random walk, we mean a limiting distribution for the local time
profile of a random walk stopped at inverse local times. To be more
precise, let
\begin{equation}
  \label{ltaukm}
   L(n,k) = \sum_{i=0}^n \ind{X_i=k},\quad\tau_{k,m} = \inf\{n\ge 0: L(n,k) > m \}, \quad k\in \Z,\ \  m,n \in \Z_+,
\end{equation}
be the local time of the walk at location $k$ and time $n$ and its
inverse local time at $k$ respectively.  We will be interested
in limiting distributions for the rescaled local time profile at these
inverse local times.  With this in mind, for some increasing to
infinity scaling function $g:\N\to (0,\infty)$ which is model
dependent, we define
  \begin{align}
  \label{taunxh}
    \tau_{x,h}^n &= \tau_{\fl{xn},\fl{hg(n)}}, \quad x \in \R, \, h\geq 0, \, n\in \N;\\ \label{lanxh} \Lambda^n_{x,h}(y) &= (g(n))^{-1}L(\tau_{x,h}^n,\fl{yn}), \qquad y \in \R. 
\end{align}

We will also need to control the endpoints of the support of these curves as well as the locations where these curves first meet going to the right or left from their starting points.
To this end, for any compactly supported functions $f_1,f_2:\R\to\R_+$ and any $x,x' \in \R$ let 
\begin{equation}\label{merge-def}
 \begin{split}
  m_{x,x'}^-(f_1,f_2) &= \sup\{y \leq x\wedge x': f_1(y)=f_2(y) \} \\
 \text{and}\quad 
 m_{x,x'}^+(f_1,f_2) &= \inf\{y \geq x\vee x': f_1(y)=f_2(y)\}, 
 \end{split}
\end{equation}
so that $m_{x,x'}^-(f_1,f_2)$ and $m_{x,x'}^+(f_1,f_2)$ are the first points to the left and right, respectively, of the interval between $x$ and $x'$ where the curves $f_1$ and $f_2$ meet. 
Let $\mathbf{0}$ be the constant zero function. Then, with the above notation,
\[
 \mu_{x,h}^{n,-} := m_{x,0}^-(\Lambda_{x,h}^n,\mathbf{0})
 \quad \text{and}\quad
 \mu_{x,h}^{n,+} := m_{x,0}^+(\Lambda_{x,h}^n,\mathbf{0}),
\]
are the endpoints of the support of the curve $\Lambda_{x,h}^n(\cdot)$, and 
\[
 \mu_{(x,h),(x',h')}^{n,-} := m^-_{x,x'}(\Lambda_{x,h}^n, \Lambda_{x',h'}^n)
 \quad\text{and}\quad 
 \mu_{(x,h),(x',h')}^{n,+} := m^+_{x,x'}(\Lambda_{x,h}^n, \Lambda_{x',h'}^n),
\]
are the locations where the curves $\Lambda_{x,h}^n(\cdot)$ and $\Lambda_{x',h'}^n(\cdot)$ merge to the left and the right of the interval between $x$ and $x'$. 
(Note that it follows easily from the definition of these curves as rescaled local time profiles that the curves $\Lambda_{x,h}^n$ and $\Lambda_{x',h'}^n$ merge to the left/right of the interval between $x$ and $x'$ at these points in the sense that $\Lambda_{x,h}^n(y) = \Lambda_{x',h'}^n(y)$ for all $y \notin [\mu_{(x,h),(x',h')}^{n,-}, \mu_{(x,h),(x',h')}^{n,+}]$.)

\begin{definition}\label{def:RK}
  A recurrent random walk $(X_n)_{n\geq 0}$ satisfies a \textbf{Ray-Knight
    theorem (RKT)} 
     if there is a scaling function $g:\N\to(0,\infty)$,
    $g(n)\to \infty$ as $n\to\infty$ such that for every
    $(x,h) \in \R\times \R_+$ the process $\Lambda_{x,h}^n(\cdot)$
    converges weakly in the Skorokhod $J_1$ topology on $D(\R)$ to
    some process $\Lambda_{x,h}(\cdot)$.  The walk satisfies an
  \textbf{augmented Ray-Knight theorem} if additionally one has joint
  convergence of the curves $\Lambda_{x,h}^n$ together with the right
  and left endpoints of these curves. That is,
  $(\Lambda_{x,h}^n(\cdot), \mu_{x,h}^{n,-}, \mu_{x,h}^{n,+})
  \underset{n\to\infty}{\Longrightarrow} (\Lambda_{x,h}(\cdot),
  \mu_{x,h}^{-}, \mu_{x,h}^{+})$, in the natural topology on
  $D(\R)\times \R^2$, where
  $\mu_{x,h}^{-} = m_{x,0}^-(\Lambda_{x,h},\mathbf{0})$ and
  $\mu_{x,h}^{+} = m_{x,0}^+(\Lambda_{x,h},\mathbf{0})$.
\end{definition}

\begin{definition}\label{def:JRK}
  A recurrent random walk $(X_n)_{n\geq 0}$ satisfies a \textbf{joint Ray-Knight
    theorem} if there is a scaling function $g:\N\to(0,\infty)$,
    $g(n)\to \infty$ as $n\to\infty$ such that for every $N \in \N$
    and every set of points
    $(x_1,h_1),\ldots,(x_N,h_N) \in \R\times \R_+$ the $N$-dimensional
    process $( \Lambda_{x_i,h_i}^n(\cdot))_{1\leq i \leq N}$ converges
    weakly in the product Skorokhod $J_1$ topology on $(D(\R))^N$ to
    some process
    $\left(\Lambda_{x_i,h_i}(\cdot) \right)_{1\leq i \leq N}$.  The
  walk satisfies an \textbf{augmented joint Ray-Knight theorem} if
  additionally one also has joint convergence of the curves together
  with the endpoints of the curves and the points where the curves
  merge. More specifically, let
  $\mu^\pm_{(x_i,h_i),(x_j,h_j)} = m_{x_i,x_j}^\pm(\Lambda_{x_i,h_i},
  \, \Lambda_{x_j,h_j})$ denote the left/right merging points of the
  curves $\Lambda_{x_i,h_i}$ and $\Lambda_{x_j,h_j}$. We say that an
  augmented joint Ray-Knight theorem is satisfied if for every
  $N \in \N$ and every finite set of points
  $(x_1,h_1),\ldots,(x_N,h_N) \in \R\times \R_+$ the $N$-dimensional
  process $(\Lambda_{x_i,h_i}^n(\cdot))_{i\leq N}$ jointly with the
  endpoints $(\mu_{x_i,h_i}^{n,*})_{i\leq N, * \in \{+,-\}}$ and merge
  points
  $(\mu_{(x_i,h_i),(x_j,h_j)}^{n,*})_{1\leq i <j\leq N, * \in
    \{+,-\}}$ converges weakly in the natural product topology to
the process  $(\Lambda_{x_i,h_i}(\cdot))_{i\leq N}$, its endpoints
  $(\mu_{x_i,h_i}^{*})_{i\leq N, * \in \{+,-\}}$, and merge points
  $(\mu_{(x_i,h_i),(x_j,h_j)}^{*})_{1\leq i <j\leq N, * \in \{+,-\}}$.
\end{definition}

The main result of the present paper is that augmented joint
Ray-Knight theorems with the scaling function of the form
  $g(n)=a_0n^\gamma$ for some $a_0,\gamma>0$ can be inverted to give a
  functional limit theorem for a random walk. The limiting process is
  then constructed as a scaling limit of the walks. Many properties of
  the limiting process can be studied via this construction, in
  particular, its local times and RKTs.

  From now on we set
\begin{equation}
  \label{g(n)}
  g(n)=a_0n^\gamma,\quad a_0,\gamma>0.
\end{equation}
We will need a few mild assumptions on the joint distributions of the
limiting curves $(\Lambda_{x_i,h_i}(\cdot))_{i\leq N}$.  Before
stating these assumptions, we first note that if an augmented joint
Ray-Knight theorem holds, then the curves
$(\Lambda_{x_i,h_i}(\cdot))_{i\leq N}$ will necessarily inherit
non-crossing, coalescence, absorption properties from the curves
$(\Lambda_{x_i,h_i}^n(\cdot))_{i\leq N}$. That is, the following
statements hold almost surely for all $N\in \mathbb{N}$ and
$1\le i,j\le N$.
\begin{enumerate}[--]
 \item {\em Non-crossing:}
 \begin{equation}\label{non-cross}
 (\Lambda_{x_i,h_i}(y)-\Lambda_{x_j,h_j}(y))(\Lambda_{x_i,h_i}(y')-\Lambda_{x_j,h_j}(y'))\ge 0,\quad \forall\, y,y' \in \R.
 \end{equation}
 \item {\em Coalescence to the right and left:}
 \[
  \Lambda_{x_i,h_i}(y) = \Lambda_{x_j,h_j}(y), 
  \quad \forall\, y \notin \left[ \mu_{(x_i,h_i),(x_j,h_j)}^-, \mu_{(x_i,h_i),(x_j,h_j)}^+ \right).
  \]
\item {\em Absorption to the right and left of the origin:}
 \[
  \Lambda_{x_i,h_i}(y)=0, \quad \forall\, y \notin \left[ \mu_{x_i,h_i}^-, \mu_{x_i,h_i}^+ \right).
 \] 
\end{enumerate}

Note that by choosing a scaling function too large, one may obtain a trivial augmented joint Ray-Knight theorem where the limiting curves are $\Lambda_{x,h}(y) \equiv h$ for all $x,y\in \R$ and $h\geq 0$.
Thus, it is clear that we need some additional assumptions on the limiting curves in the joint Ray-Knight theorems to obtain non-trivial results on the walk, 
and we will now state the additional properties that are needed for our main results.

Our first assumption concerns the distribution of a single curve
$\Lambda_{x,h}(\cdot)$ in the RKTs for the random walk. Since in a
wide class of random walk models (see Remark \ref{rem:Markov}) the directed edge local times of the
walk at the stopping times $\tau_{k,m}$ are concatenations of Markov
chains (the Markov chains considered typically change when passing the
point $x$ or the origin) then it is natural to expect that the curve
$\Lambda_{x,h}(\cdot)$ is a concatenation of diffusion
processes. Moreover, if a RKT holds with scaling function
$g(n) = a_0 n^\gamma$ then it is natural to expect that the limiting
curves have the scaling relation
\begin{equation}\label{scaling:marginal}
 (\Lambda_{cx,c^\gamma h}(cy) )_{y \in \R} \overset{\text{Law}}{=} ( c^\gamma \Lambda_{x,h}(y) )_{y \in \R}, \qquad \forall c>0. 
\end{equation}
However, if a diffusion process satisfies such a scaling relation then
it should be (away from the boundary) a solution to an SDE  of the form
\[
 d\Lambda(y) = A_1 (\Lambda(y))^{1-\frac{1}{\gamma}} \, dt + A_2(\Lambda(y))^{1-\frac{1}{2\gamma}} \, dB(y), 
\]
for some constants $A_1\in\R,\, A_2>0$ and a standard Brownian motion
$B(\cdot)$. One can easily check that such a process is equal in law
to $C Y(\cdot)^{\gamma}$ where
$C = \left( \frac{A_2}{2\gamma} \right)^{2\gamma}$ and
$Y(\cdot)$ is a squared Bessel process of dimension
$2(1-\gamma)+\frac{4A_1\gamma}{A_2^2}$. This leads us to the following
assumption. Throughout the paper we shall write
$Y\in\rm{BESQ}^\dim$ if $Y$ is a squared Bessel process of
dimension $\dim$ (defined in \eqref{besq}).

\begin{asm}[$\Lambda_{x,h}$ is a concatenation of powers of
    BESQ]\label{asm:Bessel}
  For each $(x,h) \in \R\times \R_+$, $\Lambda_{x,h}(\cdot)$ is a
  continuous curve and the curves to the right
  $(\Lambda_{x,h}(x+s))_{s\geq 0}$ and left
  $(\Lambda_{x,h}(x-s))_{s\geq 0}$ of $x$ are independent Markov
  processes started with initial value $\Lambda_{x,h}(x)=h$. Moreover,
  the curves to the right and left are multiples of {\rm BESQ}
  processes or their concatenations raised to the power $\gamma$,
  where $\gamma$ is the exponent in the scaling function
  $g(n) = a_0 n^\gamma$ in the RKTs.  More
  precisely, there are constants $\d_{r,\pm} > 0$, $\d_{a,\pm} < 2$, and
  $C_{r,\pm}, C_{a,\pm}>0$ such that
  \begin{enumerate}[--]
   \item If $x<0$ then $(\Lambda_{x,h}(x+s))_{0\leq s\leq -x}\overset{\emph{Law}}{=}(C_{r,+}Y(s)^\gamma)_{0\le s\le -x}$, where $Y\in {\rm BESQ}^{\d_{r,+}}$ with reflecting boundary condition at $0$. 
   \item If $x>0$ then $(\Lambda_{x,h}(x-s))_{0\leq s\leq x}\overset{\emph{Law}}{=}(C_{r,-}Y(s)^\gamma)_{0\le s\le x}$, where $Y\in {\rm BESQ}^{\d_{r,-}}$ with reflecting boundary condition at $0$. 
   \item $(\Lambda_{x,h}(x+s))_{s\geq \max \{0,-x\}}\overset{\emph{Law}}{=}(C_{a,+}Y(s)^\gamma)_{s\ge \max\{0,-x\}}$, where $Y\in {\rm BESQ}^{\d_{a,+}}$ with absorbing boundary condition at $0$. 
   \item $(\Lambda_{x,h}(x-s))_{s\geq \max\{0,x\}}\overset{\emph{Law}}{=}(C_{a,-}Y(s)^\gamma)_{s\ge \max\{0,x\}}$, where $Y\in {\rm BESQ}^{\d_{a,-}}$ with absorbing boundary condition at $0$. 
  \end{enumerate}
 \end{asm}

\begin{remark}
  The stated boundary conditions for the processes in
  Assumption~\ref{asm:Bessel} are natural for any Ray-Knight
  theorem, and the dimensions of the
  \rm{BESQ} processes reflect the known properties of these
    processes, i.e., \rm{BESQ} processes of
  dimension $\delta<2$ will eventually hit zero and \rm{BESQ}
  processes of dimension $\delta > 0$ are instantaneously
  reflected at $0$ (unless we specify otherwise).
\end{remark}

\begin{asm}[Joint Markovian structure]\label{asm:JMarkov}
 Let $\{(x_i,h_i)\}_{i\leq N} \subset \R\times \R_+$,
 be such that $x_1\leq x_2\leq \ldots \leq x_N$. Then the paths to the right have the following Markovian structure. 
 For any non-empty subset $I = \{i_1,i_2,\ldots,i_k\} \subset \{1,2,\ldots,N\}$,
 \[
  \mathcal{Z}_I(s) = \left( \Lambda_{x_{i_1},h_{i_1}}(x_{N}+s), \Lambda_{x_{i_2},h_{i_2}}(x_{N}+s), \ldots, \Lambda_{x_{i_k},h_{i_k}}(x_{N}+s) \right),\qquad s\geq 0
 \]
is a (non-homogeneous) Markov process on $\R^{k}$. 
Moreover, this Markov process depends on the curves $\Lambda_{x_i,h_i}(\cdot)$ to the left of $x_N$ only through the values $(\Lambda_{x_i,h_i}(x_N))_{i\in I}$ in the sense that 
\[
 P\left( (\mathcal{Z}_I(s))_{s\geq 0} \in \cdot \mid (\Lambda_{x_i,h_i}(x_i+t))_{0\leq t \leq x_N-x_i, \, i \leq N} \right) = 
 P\left( (\mathcal{Z}_I(s))_{s\geq 0} \in \cdot \mid (\Lambda_{x_i,h_i}(x_N))_{i \in I} \right).
\]
Similarly, the paths to the left have an analogous Markovian structure. 
\end{asm}

\begin{remark}\label{rem:Markov}
  Assumption \ref{asm:JMarkov} is a natural assumption when
  considering scaling limits of self-interacting random walks where
  the transition probabilities of the walk depend only on the past
  behavior of the walk at the present location; e.g., the
  self-interacting random walks defined by \eqref{sirw}, excited random
  walks (see \cite{kzERWsurvey,kpERWMCS,Tra18} and references therein), and locally Markov walks such as \cite{HLSH18}. 
  For such random walks, the directed edge local times have a
  Markovian structure analogous to the one in Assumption
  \ref{asm:JMarkov}. Since directed edge local times are roughly half
  of local times of sites, then it is natural that this Markovian
  structure is preserved in the limit of the joint RKT.
\end{remark}

Our final assumption is a non-degeneracy condition on the joint distribution of the curves. 

\begin{asm}[Non-degeneracy condition]\label{asm:ND}
The probability that the curves $\Lambda_{k,\ell}(\cdot)$ and $\Lambda_{k,\ell+1}(\cdot)$ merge to the right/left within distance $1/2$ of $k$ is strictly positive, uniformly in $(k,\ell) \in \Z\times \Z_+$. Equivalently, 
there exists a constant $c_0>0$ such that
\begin{equation}\label{nomergeasm}
 \sup_{k\in \Z, \ell \in \Z_+} \max\left\{ P\left( \mu_{(k,\ell),(k,\ell+1)}^+ \geq k + \frac{1}{2} \right), 
 P\left( \mu_{(k,\ell),(k,\ell+1)}^- \leq k - \frac{1}{2} \right) \right\} \leq 1-c_0.
\end{equation}
\end{asm}

\begin{remark}\label{rem:notight}
  Assumption \ref{asm:ND} is not satisfied for the RKTs
  of weakly reinforced random walks
  \cite{tLTWRRW} or self-repelling random walks with directed edges
  \cite{TV08} where the limiting curves of the RKTs are
  \emph{deterministic} curves of the form 
 \[
  \Lambda_{x,h}(y) = \left\{ h^\beta + B(|x|-|y|) \right\}_+^{1/\beta}, 
 \]
 for some $B>0$ and $\beta \in (0,1]$ (the RKTs in \cite{tLTWRRW,TV08} were actually proven for stopping times other than $\tau_{x,h}^n$, but they are expected to hold also for $\Lambda_{x,h}^n$). It was shown in
   \cite{Mareche_et_al2024} that under the natural scaling the paths
   of self-repelling random walks with directed edges are not tight;
   instead, a different kind of functional limit theorem holds ``in
   the bulk''. The proof of tightness in the current paper depends
   crucially on Assumption \ref{asm:ND}.
\end{remark}

Now we are ready to state our main results.  Following T\'oth
\cite{tTSAWGBR,tTSAW,tGRK}, one sees that a RKT with scaling function
$g(n)=a_0 n^\gamma $ suggests a limiting distribution for the walk
when time is scaled by $n$ and space by $n^{1/(1+\gamma)}$.  For our
proof, however, it will be more convenient to
scale time by $n g(n)=a_0 n^{1+\gamma}$ and space by $n$ so that the rescaled path of the walk is given by
\begin{equation}
  \label{xnt}
  \mathcal{X}_n(t) = n^{-1}X_{\fl{t n g(n)}}, \quad t\geq 0. 
\end{equation}
Throughout the paper, weak convergence (in the appropriate space) will
be denoted by $\Longrightarrow$.
It is not hard to see that Theorem
\ref{thm:main_X} implies convergence stated in the traditional form,
that is,
$\left(\frac{X_{\fl{tn}}}{c n^{1/(1+\gamma)}} \right)_{t\geq 0}
\Longrightarrow (\mathfrak{X}(t))_{t\geq 0}$, where
$c = a_0^{-1/(1+\gamma)}$ (see the explanation after Theorem 2.1 in
\cite{KP25+} for more details).

\begin{theorem}\label{thm:main_X}
 Suppose that a recurrent random walk $(X_n)_{n\geq 0}$ satisfies an augmented joint RKT, and that the limiting curves of the joint Ray-Knight theorem satisfy Assumptions \ref{asm:Bessel}--\ref{asm:ND}. 
  Then, there is a continuous process $(\mathfrak{X}(t))_{t\geq 0}$ such that 
  $\mathcal{X}_n  \underset{n\to\infty}{\Longrightarrow} \mathfrak{X}$ with respect to the Skorokhod $J_1$ topology on $D([0,\infty))$. 
\end{theorem}

While the statement of Theorem \ref{thm:main_X} does not characterize the limiting process $\mathfrak{X}$, the proof of Theorem \ref{thm:main_X} will show that the finite dimensional distributions of $\mathfrak{X}$ are uniquely determined by the augmented joint RKT.

One expects that the limiting process $\mathfrak{X}$ should also have the same joint RKT as the random walk. The next two theorems show that this is indeed the case. 
First we will show the joint convergence of the path of the walk and the rescaled local time profile of the walk
\begin{equation}
  \label{lntx}
  \mathcal{L}_n(t,x) = (g(n))^{-1}L(\fl{t n g(n)}, \fl{xn}), \quad t \geq 0, \, x \in \R. 
\end{equation}
 
Note that for a fixed $t\geq 0$, $\mathcal{L}_n(t,\cdot)$ is an element of $D(\R)$ and so we will view 
$\mathcal{L}_n$ as an element of the Skorokhod space $D([0,\infty); D(\R))$ (see Section \ref{sec:notation} for a precise description of the topologies used for this and other spaces).

\begin{theorem}\label{thm:main_XL}
 Under the same assumptions as in Theorem \ref{thm:main_X}, as $n\to \infty$, the sequence of processes
$(\mathcal{X}_n(t), {\cal L}_n(t,\cdot))_{t\ge 0}$ converges weakly in
$D([0,\infty)) \times D([0,\infty);D(\R))$ to a process
$(\mathfrak{X}(t), \mathfrak{L}(t,\cdot))_{t\ge 0}$ where, almost surely,
$(\mathfrak{X}, \mathfrak{L})\in C([0,\infty))\times
C([0,\infty);C_0(\R))$ and $\mathfrak{L}$ is the local
time\footnote{We recall that in general the fact that
  $\mathcal{X}_n \Rightarrow \mathfrak{X} $ where each $\mathcal{X}_n$
  has a local time does not imply that the limit process has a local
  time nor that the local times of the $\mathcal{X}_n$ converge.} for
$\mathfrak{X}$; i.e., almost surely, 
\begin{equation}
  \label{otf}
  \int_a^b \mathfrak{L}(t,x) \, dx = \int_0^t \mathbbm{1}_{[a,b]} (\mathfrak{X}(s))ds,
  \qquad \forall\, a<b, \, t\geq 0.
\end{equation}
\end{theorem}

\begin{theorem}\label{thm:main_JRK}
 Let the same assumptions as in Theorem \ref{thm:main_X} hold and $(\mathfrak{X}, \mathfrak{L})$ be the limiting process from Theorem \ref{thm:main_XL}. Then the process $\mathfrak{X}$ has the same joint Ray-Knight theorems as the random walk. More precisely, 
 if
 \[
  \mathfrak{t}_{x,h} = \inf\{t\geq 0: \mathfrak{L}(t,x) > h \}, \quad x\in \R,\, h\geq 0,
 \]
 are the inverse local times for the process $\mathfrak{X}$, then for any finite set of points $(x_1,h_1),\ldots,(x_N,h_N) \in \R \times \R_+$, 
the joint distribution of the local time profiles $
 \left( \mathfrak{L}(\mathfrak{t}_{x_i,h_i},\cdot) \right)_{i\leq N}$ 
 is the same as that of the curves $\left( \Lambda_{x_i,h_i}(\cdot) \right)_{i\leq N}$ from the joint Ray-Knight theorems for the random walk. 
 Moreover, $\mathfrak{X}$ is the unique continuous process with these joint Ray-Knight theorems in the sense that if $\tilde{\mathfrak{X}}$ is a continuous process with  local time $(\tilde{\mathfrak{L}}(t,x))_{t\geq 0, x\in \R}$ that is non-decreasing in $t$ and such that
 $(\tilde{\mathfrak{X}},\tilde{\mathfrak{L}})$ has the same joint Ray-Knight theorems as $\mathfrak{X}$, then
 $\tilde{\mathfrak{X}}$ and $\mathfrak{X}$ have the same distribution.
 \end{theorem}

\begin{remark}
  We highlight here that the uniqueness mentioned in Theorem
  \ref{thm:main_JRK} contrasts with the fact that the RKTs for a
    single curve ($N=1$) do not uniquely characterize the process (as
  shown for polynomially self-repelling walks in \cite{KMP23}).
\end{remark}
 
As a first application of the above results, we note that Theorems
\ref{thm:main_X} and \ref{thm:main_XL} combined with the augmented
joint RKT for TSAW \cite[Theorem 1.1]{KP25+} and the recurrence of the
TSAW that can be deduced from \cite{tTSAW} give not only a new proof
of the convergence of the path of a TSAW that is independent of the a
priori construction of the TSRM, but also joint convergence of the
path of the walk and its local time profile.  While this proof
strategy is independent of the construction of TSRM in \cite{TW98},
Theorem \ref{thm:main_JRK} together with the joint Ray-Knight theorems
for the TSRM \cite[Theorem 4.3]{TW98} can then be used to identify the
scaling limit as the TSRM after the fact.  We summarize these results
in the following corollary.

\begin{corollary}\label{cor:TSAW}
 Let $(X_n)_{n\geq 0}$ be a ``true'' self-avoiding walk corresponding to a weight function $w(n) = e^{-\beta n}$ for some $\beta>0$, and let 
 $\sigma_\beta^2 = \frac{\sum_{x\in\Z} x^2 e^{-\beta x^2}}{\sum_{x\in\Z} e^{-\beta x^2} }$. 
If $\mathcal{X}_n$ and $\mathcal{L}_n$ are defined as in \eqref{xnt} and \eqref{lntx} with scaling function $g(n) = 2\sigma_\beta \sqrt{n}$, then
$(\mathcal{X}_n, \mathcal{L}_n) \underset{n\to\infty}{\Longrightarrow} (\mathfrak{X}, \mathfrak{L})$ in $D([0,\infty)) \times D([0,\infty); D(\R))$ where $\mathfrak{X}$ is a ``true'' self-repelling motion (TSRM)  and $\mathfrak{L}$ is its local time process. 
\end{corollary}
As noted above, Corollary \ref{cor:TSAW} is a stronger result than
what was proved for the TSAW in \cite{KP25+}.  The main result in
\cite{KP25+} was the joint convergence of the (rescaled) position and
local time at the current location for the TSAW. The following
corollary shows that this follows from the stronger statement in
Theorem \ref{thm:main_XL}.

\begin{corollary}\label{xh}
  Let the same assumptions as in Theorem \ref{thm:main_X} hold
  and let $\mathcal{H}_n(t):={\cal L}_n(t,{\cal X}_n(t))$, $t\ge 0$,
    be the rescaled local time of the walk at its current
    location. Then
  $(\mathcal{X}_n, \mathcal{H}_n)\underset{n\to\infty}{\Longrightarrow} (\mathfrak{X}, \mathfrak{H})$  in $D([0,\infty))\times D([0,\infty))$, where
  $\mathfrak{H}(t) = \mathfrak{L}(t,\mathfrak{X}(t))$.
\end{corollary}

Our main application of Theorems \ref{thm:main_X}--\ref{thm:main_JRK} in the current paper is a functional limit theorem for polynomially self-repelling random walks. 

\begin{theorem}\label{main_psr}
  Let $(X_n)_{n\geq 0}$ be a polynomially self-repelling random walk
  corresponding to a weight function $w(n)$ with asymptotics of the
  form \eqref{w} with $\alpha>0$. Then, the walk satisfies an
  augmented joint RKT with scaling function $g(n) = n$
 and with limiting curves for which Assumptions
    \ref{asm:Bessel}--\ref{asm:ND} hold. Hence, there is a continuous
  process $(\mathfrak{X}_\alpha(t))_{t\geq 0}$ with continuous local
  times $(\mathfrak{L}_\alpha(t,x))_{t\geq 0, \, x \in \R}$ such that
  $(\mathcal{X}_n, \mathcal{L}_n)
  \underset{n\to\infty}{\Longrightarrow} (\mathfrak{X}_\alpha, \mathfrak{L}_\alpha)$
  in $D([0,\infty)) \times D([0,\infty);D(\R))$. Moreover, the
  process $\mathfrak{X}_\alpha$ is the unique continuous process with joint
  Ray-Knight theorems consistent with the joint Ray-Knight theorems
  for the polynomially self-repelling walk. 
\end{theorem}
We prove Theorem~\ref{main_psr} in Section~\ref{PSRcase}. The joint
RKT for PSRW (a version with edge local times instead of site local times, see Section \ref{sec:PSRaug}) has been shown in
\cite[Theorem~1.2]{Y25+}.
We prove
Theorem \ref{main_psr} by extending the joint RKT to an augmented
joint RKT and check the assumptions stated in Theorem
\ref{thm:main_X}.

\begin{remark}\label{rem:PSRMprop}
  The limiting processes $\mathfrak{X}_\alpha$, which we call
  polynomially self-repelling motions (PSRMs), form a new family of
  continuous self-interacting processes. Since $\mathfrak{X}_\alpha$
  is the unique continuous process whose joint Ray-Knight theorems are
  consistent with those for the $\alpha$-PSRWs, it coincides with the
  process constructed recently in \cite{Y25+} from a flow of
  correlated squares of Brownian motions (``two-sided PSR flow'',
  introduced in \cite{Y26flows}), much as the TSRM emerged from a
  Brownian web in \cite{TW98}. A characterization of the PSRM by
    its joint Ray-Knight theorems is also given in
    \cite[Proposition~4.7]{Y25+}. While the constructions here and in
  \cite{Y25+} do not lend much insight into the dynamics of the PSRM,
  heuristic calculations in Appendix \ref{psr_sde} suggest that
\begin{equation}\label{PSRM-heur}
\text{``}\mathfrak{X}_\alpha(t)=-\frac{\alpha}{2}\int_0^t \left[\frac{\partial}{\partial x} \log \left( \mathfrak{L}_\alpha(s,x)\right) \right]_{x=\mathfrak{X}_\alpha(s)}\,ds+B(t) + (\text{boundary effects})\text{''}, 
\end{equation}
where $B(\cdot)$ is a standard Brownian motion. The ``boundary
effects'' in the above equation refer to the fact that
\eqref{PSRM-heur} should describe the dynamics of
$\mathfrak{X}_\alpha$ only when it is in the interior of its range.
We emphasize that \eqref{PSRM-heur} is only a heuristic equation,
  since the local time $\mathfrak{L}_\alpha$ is not differentiable in
  either of its variables, but that it does reflect the self-repelling
  nature of the process, in that the drift points toward regions where
  the local time is smaller. A similar heuristic equation was given
  for the TSRM in \cite{TW98}. In the special case $\alpha=1$,
  essentially the same driving mechanism as in \eqref{PSRM-heur}
  appears in \cite{LST21} as a description of the scaling limit of a
  different self-interacting random walk, though the relation between
  the PSRM $\mathfrak{X}_1$ and the limiting process there is not
  immediately clear. 
\end{remark}

In conclusion, we would like to point out that
  Theorem~\ref{main_psr} completely resolves the counterexample and the
  open question posed in \cite[Remark 1.7]{KMP23}. Recall that the
results in \cite{tGRK} and \cite{cpyBetaPBM} identified, for each
fixed $t$, the limiting distribution of the rescaled PSRW
$n^{-1/2}X_{\fl{nt}}$, should it exist, as that of
$\sqrt{2\alpha+1}\,W^{1/2,1/2}(t)$.  Therefore, it follows from
Theorem~\ref{main_psr} that $\mathfrak{X}_\alpha$ and
$ \sqrt{2\alpha+1}\,W^{1/2,1/2}$ are continuous processes with the
same marginal distributions and the same RKTs, but different joint
RKT. In \cite{tGRK}, T\'oth noted that it was surprising that the
parameter $\alpha>0$ for PSRW appeared in the RKT only as a scaling
factor and that the limiting curves were always BESQ$^{1}$ for all
values of $\alpha$.  The joint RKTs for PSRW and Theorem \ref{main_psr} provide a
  fuller picture in which this fact fits naturally: the limiting
processes $\mathfrak{X}_\alpha$ are genuinely distinct as $\alpha$
varies, even though their marginal distributions differ only by a
scaling factor depending on $\alpha$.

\subsection{Notation}\label{sec:notation}
This subsection records some notation, definitions of function spaces
and topologies that will be used throughout the paper.
\begin{enumerate}[--]
\item $\N=\{1,2,\ldots\}$, $\Z_+=\{0,1,2,\ldots\}$, $\R_+=[0,\infty)$.

\item Weak convergence (in the appropriate space) will be denoted by $\Longrightarrow$ throughout the paper.

\item  We shall write $Y\in\rm{BESQ}^{\dim}$ to denote a squared Bessel process $Y(\cdot)$ of dimension $\dim \in \R$, i.e., a unique strong solution of the SDE
\begin{equation}
  \label{besq}
   dY(t) = \dim \, dt + 2\sqrt{|Y(t)|} \, dB(t), \quad t\ge 0,\ Y(0)=y_0\ge 0,
\end{equation}
where $B=(B(t))_{t\ge 0}$ is a standard Brownian motion. Note that
$Y$ almost surely never hits zero if $\dim \geq 2$ and $Y(0)>0$ and
hits $0$ in finite time if $\dim<2$. For $\dim=0$ the point $0$ is
absorbing and for $0<\dim<2$ the point $0$ is instantaneously
reflecting, \cite{ryCMBM,gySGBP}. If $\dim<0$ we shall always
assume that the process is absorbed at $0$ at the first hitting time
of $0$.

 \item $C([0,\infty))$ and $C(\R)$ are the spaces
of real-valued continuous functions on $[0,\infty)$ and $\R$
respectively with the metric
\begin{equation}\label{dinf-metric}
d_\infty(f,g)=\sum_{i=1}^\infty 2^{-i}\min\left\{\max_{s\in
      K_i}|f(s)-g(s)|,1\right\}, 
\end{equation}
 where $K_i=[0,i]$ and $K_i=[-i,i]$,
$i\in\mathbb{N}$, respectively. This corresponds to the topology of
the locally uniform convergence. Endow the space
$C([0,\infty);C(\R))$ of functions
 from $[0,\infty)$ to $C(\R)$ with the metric
\[\text{dist}_\infty(f,g)=\sum_{i=1}^\infty
  2^{-i}\min\left\{\max_{s\in[0,i]}d_\infty
    (f(s,\cdot),g(s,\cdot)),1\right\}.\] The spaces $C([0,\infty))$,
$C(\R)$, and $C([0,\infty);C(\R))$ with the above metrics are Polish.
\item We denote by $C_0(\R)$ the space of
    real-valued continuous functions on $\R$ with compact
    support and by $C_B(\R)$ the space of real-valued bounded
    continuous functions on $\R$. On these spaces, we use the
    same metric $d_\infty$ (see \eqref{dinf-metric}) as on
    $C(\R)$.
\item For a closed interval $I\subset \R$, the set $D(I)$ is the space of
  real-valued c\`adl\`ag functions  on $I$.  We will use
  either the Skorokhod $J_1$ topology or the topology of uniform
  convergence on compact sets on these spaces, and we will indicate
  which topology we are using at different times.  In the case where
  $I=\R$ we work with the Skorokhod metric defined in \cite[page 191]{bCOPM}. With this metric
  $D(\R)$ is a Polish space.
\item For a closed interval $I\subset [0,\infty)$ and a Polish space $\mathcal{S}$, $D(I;\mathcal{S})$ will denote the set of c\`adl\`ag functions $f:I\to \mathcal{S}$. When $\mathcal{S}$ is a Polish space other than $\R$ we will always equip $D(I;\mathcal{S})$ with the Skorokhod $J_1$ topology. 

\end{enumerate}

\section{The augmented joint RKT implies tightness}\label{tightness}

In this section we will prove that the augmented joint RKT implies
tightness for the rescaled random walks ${\cal X}_n$ and for the
  local time processes ${\cal L}_n$. We shall need the notation introduced in 
 \eqref{taunxh}, \eqref{lanxh}, \eqref{g(n)}, \eqref{xnt}, and \eqref{lntx}.
Before proceeding with the main results of this section, we note three
properties of the families of curves
$(\Lambda_{x_i,h_i}(\cdot))_{i\leq k}$ which follow from the joint RKTs and Assumptions \ref{asm:Bessel} and
\ref{asm:JMarkov}.
\begin{enumerate}[--]
 \item {\em A shift invariance property:}
 for any $h\geq 0$,   
 \begin{align}
  (\Lambda_{x,h}(x+s))_{s\geq 0} &\overset{\text{Law}}{=}
  (\Lambda_{0,h}(s) )_{s\geq 0}, \quad\text{if } x\geq 0, \label{Lxh-shift} \\
  (\Lambda_{x,h}(x+s))_{0\leq s \leq |x|\wedge|x'|} &\overset{\text{Law}}{=}
  (\Lambda_{x',h}(x'+s) )_{0\leq s\leq |x|\wedge|x'|}, \quad\text{if } x,x' \leq 0 \label{Lxh-shift2}, 
 \end{align}
 and a symmetric property for the ``backward'' curves $\Lambda_{x,h}(x-s)$, $s\geq 0$. 
 
 \item {\em Continuity in law:} for any $k,m \in \N$, the map
 \[
  \left( (x_1,h_1), \ldots, (x_k,h_k); y_1, \ldots, y_m \right) \mapsto \text{Law} \left( \left(\Lambda_{x_i,h_i}(y_j) \right)_{i\leq k, j \leq m} \right)
 \]
 is continuous from $(\R \times \R_+)^k \times \R^m$ to the space of Borel probability measures on $\R^{km}$ with the topology of weak convergence (see Proposition \ref{prop:fdd} in Appendix \ref{app:contprop}).

\item {\em A scaling property:} for any $c>0$ and any choice of $(x_1,h_1),\ldots,(x_k,h_k) \in \R\times \R_+$, 
 \begin{equation}\label{Lxh-rescale}
  \left( (\Lambda_{c x_i, c^\gamma h_i}(c y )_{y\in \R})\right)_{i\leq k} 
  \overset{\text{Law}}{=} \left( (c^\gamma \Lambda_{x_i,h_i}(y))_{y\in \R} \right)_{i\leq k}.
 \end{equation}
Indeed, it is easy to see that \eqref{Lxh-rescale} follows immediately from the joint RKTs if $c>0$ is rational. If $c$ is irrational, we obtain \eqref{Lxh-rescale} using a sequence of rationals tending to $c$ and the continuity in law stated above.
\end{enumerate}

Next we state three lemmas that follow from the augmented joint
RKTs and which will be crucial for both the tightness
argument below and also the identification of the finite dimensional
distributions.  The proofs are similar to analogous statements in
\cite{KP25+}, but we include them here both for completeness and to
account for minor differences in the current more general setting.

\begin{lemma}\label{lem:txhn-conv}
 For any $k\in \N$ and any $(x_1,h_1),(x_2,h_2),\ldots,(x_k,h_k) \in \R \times \R_+$ we have that
 \begin{equation}\label{tauxhn-lim}
 \left( \frac{\tau^n_{x_i,h_i}}{n g(n)} \right)_{i\leq k} \Longrightarrow 
 \left( \int_\R \Lambda_{x_i,h_i}(y) \, dy \right)_{i\leq k}.
\end{equation}
\end{lemma}

\begin{proof}
 Since $\sum_{i\in\Z} L(n,i) = n+1$, it follows for any $x\in \R$, $h\geq 0$ and $n\geq 1$ that
 \begin{align*}
  \frac{\tau_{x,h}^n}{ng(n)} &= \frac{1}{ng(n)} \sum_{i \in \Z} L(\tau_{x,h}^n,i) - \frac{1}{n g(n)} \\
  &= \sum_{i \in \Z} \Lambda_{x,h}^n(i/n) \frac{1}{n} - \frac{1}{n g(n)} 
  =\int_{\mu_{x,h}^{n,-}}^{\mu_{x,h}^{n,+}}  \Lambda_{x,h}^n(y)\, dy - \frac{1}{n g(n)}. 
 \end{align*}
   The curves $\Lambda_{x_i,h_i}(\cdot)$ are continuous with compact support,
   so the conclusion of the lemma follows from the augmented joint
   Ray-Knight theorem.
\end{proof}

\begin{lemma}\label{lem:RKn}
 For any  $n\geq 1$ and  $K,T\in(0,\infty)$, let $R_{K,T,n}$ be the event 
 \[
  R_{K,T,n} = \left\{ \max_{k \leq T n g(n)} |X_k| \leq K n \text{ and } \max_{k \in \Z} L(\fl{T n g(n)}, k) \leq K^{1+\gamma} g(n) \right\}. 
 \]
Then for every $\eta>0$ and $T\in(0,\infty)$ there exists
$K\geq 1$ such that $P(R_{K,T,n}) > 1-\eta/3$ for all
 $n$ sufficiently large.
\end{lemma}
 
\begin{proof}
 First of all, note that if $K\in\mathbb{N}$ then $R_{K,T,n} \supset R_{K,T,n}^+ \cap R_{K,T,n}^-$ where
 \begin{align*}
  R_{K,T,n}^\pm &=  \left\{ \tau_{\pm Kn,0} > T n g(n), \text{ and } \max_{k\in \Z} L(\tau_{\pm Kn,0}, k) < K^{1+\gamma} g(n) \right\} \\
  &= \left\{  \int_\R \Lambda_{\pm K,0}^n(y) \, dy > T + \frac{1}{n g(n)}, \text{ and }  \sup_{y\in\R} \Lambda_{\pm K,0}^n(y) < K^{1+\gamma} \right\}.
 \end{align*}
Then the conclusion of the Lemma follows from the augmented joint Ray-Knight theorem and the observation that the scaling property \eqref{Lxh-rescale} implies that
\begin{align*}
 &P\left( \int_\R \Lambda_{\pm K,0}(y) \, dy > T , \text{ and }  \sup_{y\in\R} \Lambda_{\pm K,0}(y) < K^{1+\gamma} \right) \\
 &\qquad = P\left( \int_\R \Lambda_{\pm 1,0}(y) \, dy > \frac{T}{K^{1+\gamma}} , \text{ and }  \sup_{y\in\R} \Lambda_{\pm 1 ,0}(y) < K \right), 
\end{align*}
since the last probability can be made arbitrarily close to 1 by choosing $K$ sufficiently large. 
\end{proof}

For the next Lemma we need some additional notation.
\begin{enumerate}
 \item For a finite collection of reals $\mathcal{T}= \{ t_i \}_{i\in\mathcal{I}}$, $2\le\text{card}(\mathcal{I})<\infty$,  let
\[
 \text{gap}(\mathcal{T}) = \min\{|t_i-t_j|: \, i \neq j,\ i,j\in\mathcal{I}\}.
\]
 \item For $\d>0$ and $n\geq 1$ let 
 \[
  \mathcal{T}_{\d,n} = 
  \left\{ \frac{\tau_{x,h}^n}{n g(n)} \right\}_{(x,h) \in Q_\delta} ,
  \quad \text{where}\quad 
  Q_\d = (\d \Z) \times (\d^\gamma \Z_+),
 \]
 and similarly for $\d>0$, $n\geq 1$ and $K\geq 1$ let 
 \begin{equation}\label{TdnK-def}
  \mathcal{T}_{\d,n}^{(K)} = 
  \left\{ \frac{\tau_{x,h}^n}{n g(n)}  \right\}_{(x,h) \in  Q_\d^{(K)}} , 
  \quad \text{where}\quad  
 Q_\d^{(K)} =  Q_\d \cap ([-K,K] \times [0,K^{1+\gamma}]).
 \end{equation}

\end{enumerate}

\begin{lemma}\label{lem:gap}
 For $K\geq 1$ and $\eta,\d>0$ fixed, there exists a $\d' = \d'(K,\eta,\d) > 0$ such that for $n$ sufficiently large 
 \[
  P\left( \mathrm{gap}(\mathcal{T}^{(K)}_{\d,n}) > \d' \right) > 1 - \frac{\eta}{3}.
 \]
\end{lemma}

\begin{proof}
 It follows from Lemma \ref{lem:txhn-conv} that 
 $\text{gap}(\mathcal{T}_{\d,n}^{(K)})  \Longrightarrow \text{gap}(\mathfrak{T}_\d^{(K)})$, where 
 \[
  \mathfrak{T}_\d^{(K)} 
  = \left\{  \int_\R \Lambda_{x,h}(y) \, dy \right\}_{(x,h) \in Q_\d^{(K)} } .
 \]
 By Lemma \ref{lem_int_diff}, 
 \[
  P\left( \int_\R \Lambda_{x,h}(y) \, dy \neq \int_\R \Lambda_{x',h'}(y) \, dy \right) = 1, \qquad \forall\, (x,h) \neq (x',h'),
\]
so $\text{gap}(\mathfrak{T}_\d^{(K)}) > 0$, almost
 surely. Therefore, we may choose $\d'>0$ sufficiently small so that
 $\liminf_{n\to\infty} P\left( \text{gap}(\mathcal{T}^{(K)}_{\d,n}) >
   \d' \right) \geq P\left( \text{gap}(\mathfrak{T}_\d^{(K)}) > \d'
 \right) > 1-\frac{\eta}{3}$.
\end{proof}

\begin{remark}\label{rem:RKn+gap}
 Note that for $\d>0$ and $K\geq 1$ fixed and $n$ sufficiently large
 \[
  R_{K,1,n} \cap \left\{ \text{gap}(\mathcal{T}_{\d,n}^{(K)}) > \d' \right\} \subset  \left\{ \text{gap}\left( \mathcal{T}_{\d,n} \cap [0,1] \right) > \d' \right\}. 
 \]
 We will use this fact in the proofs of tightness below. 
\end{remark}

\subsection{Tightness of the walks.}

The main result of this subsection is the tightness of the
  $D([0,\infty))$-valued random variables
  $(\mathcal{X}_n(\cdot))_{n\geq 1}$.  The proof below is a
  generalization of the one for the TSAW, see
  \cite[Proposition 4.5]{KP25+}.
\begin{proposition}\label{prop:Xtight}
 The sequence of $D([0,\infty))$-valued random variables $(\mathcal{X}_n)_{n\geq 1}$ is tight with respect to the Skorokhod $J_1$ topology in $D([0,\infty))$. Moreover, any subsequential limit is concentrated on paths that are continuous. 
\end{proposition}

\begin{remark}
  The proof of Proposition \ref{prop:Xtight} will follow closely that of
    \cite[Proposition 4.5]{KP25+} for the TSAW.  We will present the
  main ideas here and direct the reader to \cite{KP25+} for certain
  details.
\end{remark}

\begin{proof}
 It is enough to show that for any $\e,\eta,T>0$ there is a $\d'>0$ such that
 \begin{equation}\label{Xtight-cond}
  \limsup_{n\to\infty} 
  P\left( 
  \max_{\substack{k,\ell \leq T n g(n) \\ |k-\ell|\leq \d' n g(n) }} 
  |X_k - X_\ell| > \e n \right) < \eta. 
 \end{equation}
 We will prove this only for $T=1$ as the argument is essentially unchanged for general $T>0$.
 Remark \ref{rem:RKn+gap} allows us on the event $R_{K,1,n} \cap \left\{ \text{gap}(\mathcal{T}_{\d,n}^{(K)}) > \d' \right\}$ to reduce the problem to considering only the distance traveled by the walk between consecutive times in the set $\mathcal{T}_{\d,n}$. Indeed, following the same argument as in the proof of \cite[Proposition 4.5]{KP25+} (with the time scaling $n^{3/2}$ replaced here by $n g(n)$ and the bound on local times $2\sigma K \sqrt{n}$ replaced here by $K^{1+\gamma} g(n)$) we have that for all large enough $n$
 \begin{equation}\label{Xtightdec}
  \left\{ \max_{\substack{k,\ell \leq n g(n) \\ |k-\ell|\leq \d' n g(n) }} 
  |X_k - X_\ell| > \e n \right\}
  \cap R_{K,1,n} \cap \left\{ \text{gap}(\mathcal{T}_{\d,n}^{(K)} ) > \d' \right\} 
  \subset 
  \bigcup_{(x,h) \in Q_\d^{(K)} }
  A_{x,h,\d,\e,n}^{(K)}, 
 \end{equation}
 where
\begin{align*}
 A_{x,h,\d,\e,n}^{(K)} &= 
 \left\{ \exists\, j> \tau_{x,h}^n: \, \left|X_j - X_{\tau_{x,h}^n} \right| > \frac{\e}{4} n  \text{ and } \mathcal{T}_{\d,n} \cap \left( \frac{\tau_{x,h}^n}{n g(n)}, \frac{j}{n g(n)} \right) = \emptyset  \right\}
 \\
 &\qquad\qquad
 \cap \left\{ \max_{i \in \Z} L\left( \tau_{x,h}^n, i \right) \leq K^{1+\gamma} g(n)  \right\}. 
\end{align*}
 We claim that
\begin{equation}\label{PAxh}
  \limsup_{n\to\infty} P\left( A_{x,h,\d,\e,n}^{(K)} \right) \leq 2(1-c_0)^{\fl{\frac{\e}{5\d}}}, \qquad \forall\, x \in \d\Z \cap [-K,K], \, h \in \d^\gamma \Z_+ \cap [0,K^{1+\gamma}],
\end{equation}
where $c_0$ is the constant in \eqref{nomergeasm} from Assumption \ref{asm:ND}.
The event $A_{x,h,\d,\e,n}^{(K)}$ involves the walk making a large excursion between consecutive times in $\mathcal{T}_{\d,n}$, and following the proof in \cite{KP25+} we will consider separately the event that this is a large excursion to the right or left. 
That is, if we define
\begin{align*}
 A_{x,h,\d,\e,n}^{(K),+} &= 
 \left\{ \exists\, j> \tau_{x,h}^n: \, X_j > \fl{xn} + \frac{\e}{4} n  \text{ and } \mathcal{T}_{\d,n} \cap \left( \frac{\tau_{x,h}^n}{n g(n)}, \frac{j}{n g(n)} \right) = \emptyset  \right\}
\\
 &\qquad\qquad
 \cap \left\{ \max_{i \in \Z} L\left( \tau_{x,h}^n, i \right) \leq K^{1+\gamma} g(n)  \right\}, 
\end{align*}
and define $A_{x,h,\d,\e,n}^{(K),-}$ similarly with $X_j>\fl{xn}+\frac{\e}{4} n$ replaced by $X_j<\fl{xn}-\frac{\e}{4} n$, then  \eqref{PAxh} will follow from
\begin{equation}\label{PAxh-pm}
  \limsup_{n\to\infty} P\left( A_{x,h,\d,\e,n}^{(K),\pm} \right) \leq (1-c_0)^{\fl{\frac{\e}{5\d}}}, \qquad \forall\, (x,h) \in Q_\d^{(K)}.
\end{equation}
We will describe the proof of \eqref{PAxh-pm} for  $P\left( A_{x,h,\d,\e,n}^{(K),+} \right)$ only as the proof for $P\left( A_{x,h,\d,\e,n}^{(K),-} \right)$ is similar.

The idea behind the bound in \eqref{PAxh-pm} is that on the event $A_{x,h,\d,\e,n}^{(K),+}$ the walk travels at least $\e n/4$ to the right of $X_{\tau_{x,h}^n} = \fl{xn}$ before any other hitting time of the form $\tau_{x',h'}^n$ with $(x',h') \in Q_\d $. But this requires that the curve $\Lambda_{x,h}^n(\cdot)$ does not merge with any of the curves $\Lambda_{x',h'}^n(\cdot)$ above it for at least distance $\e/4$ to the right. This non-merging condition is unlikely because it forces many curves of the form $\Lambda_{k\d,\ell\d^\gamma}^{n}(\cdot)$ and $\Lambda_{k\d,(\ell+1)\d^\gamma}^{n}(\cdot)$ to not meet for distance $\d$ (see Figure \ref{fig:snake} for a visual depiction).
\begin{figure}
  \includegraphics[page=1,scale=0.75]{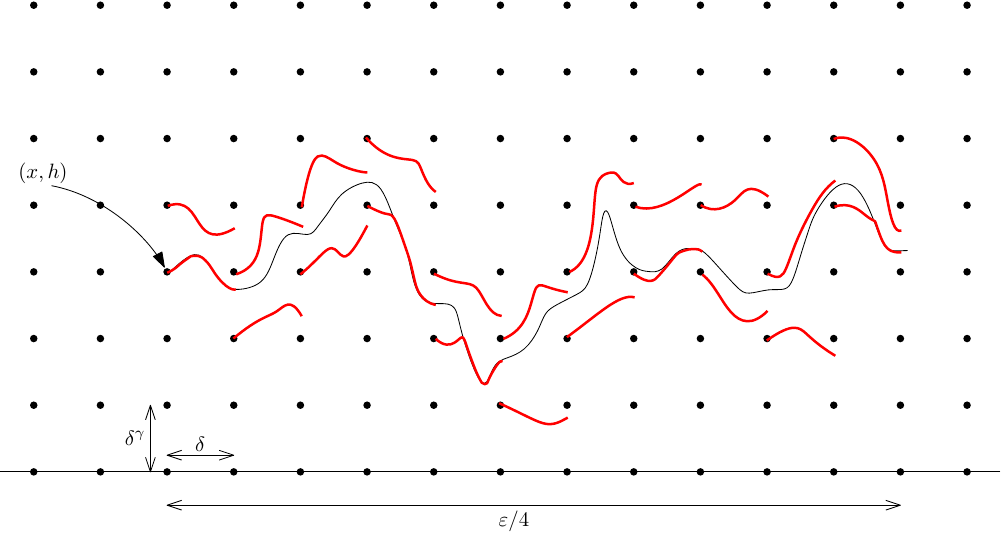}
  \caption{A visual depiction of the event $A^{(K),+}_{x,h,\d,\e,n}$, a long excursion to the right of $\fl{xn}$. The black curve
    in the middle depicts $\Lambda_{x,h}^n(\cdot)$.  The short red
    curves depict pieces of curves
    $\Lambda_{k\d,\ell \d^\gamma}^n(\cdot)$ and
    $\Lambda_{k\d,(\ell+1)\d^\gamma}^n(\cdot)$ which on the event
    $A^{(K),+}_{x,h,\d,\e,n}$ do not meet.
  }\label{fig:snake}
\end{figure}
More precisely,  the event $A_{x,h,\d,\e,n}^{(K),+}$ implies that for all $k\in \Z$  such that $\fl{xn} \leq \fl{k\d n} < \fl{(k+1)\d n} \leq \fl{xn} + \frac{\e n}{4}$,  if
\[
 \frac{\fl{\ell \d^\gamma g(n)}+1}{g(n)} = \Lambda_{k\d, \ell \d^\gamma}^n(k\d) 
 \leq \Lambda_{x,h}^n(k\d) < 
 \Lambda_{k\d, (\ell+1)\d^\gamma}^n(k\d) = \frac{\fl{(\ell+1) \d^\gamma g(n)}+1}{g(n)}, 
\]
then $\Lambda_{k\d,\ell\d^\gamma}^n((k+1)\d) < \Lambda_{k\d,(\ell+1)\d^\gamma}^n((k+1)\d)$, or equivalently
$\mu_{(k\d, \ell \d^\gamma),(k\d,(\ell+1)\d^\gamma)}^{n,+} > \frac{\fl{(k+1)\d n}}{n}$.
Therefore, it follows from the augmented joint RKT, Assumption \ref{asm:Bessel}, and the rescaling property \eqref{Lxh-rescale} that
\begin{align*}
&\limsup_{n\to\infty} P\left( A_{x,h,\d,\e,n}^{(K),+} \right) \\
 &\qquad \leq P\left(\bigcap_{k=\frac{x}{\d}}^{\frac{x}{\d}+ \fl{\frac{\e}{5\d}} -1} \bigcup_{0\leq \ell \leq \frac{K^{1+\gamma}}{\d^\gamma} } \left\{ \ell \d^\gamma \leq \Lambda_{x ,h}(k\d) < (\ell+1)\d^\gamma , \,  \mu_{(k\d, \ell\d^\gamma),(k\d,(\ell+1)\d^\gamma)}^{+} \geq \left(k+\frac{1}{2}\right)\d \right\} \right) \\
 &\qquad = P\left(\bigcap_{k=\frac{x}{\d}}^{\frac{x}{\d}+ \fl{\frac{\e}{5\d}} -1} \bigcup_{0\leq \ell \leq \frac{K^{1+\gamma}}{\d^\gamma} } \left\{ \ell \leq \Lambda_{\frac{x}{\d},\frac{h}{\d^\gamma}}(k) < \ell+1, \,  \mu_{(k, \ell),(k,\ell+1)}^{+} \geq k+\frac{1}{2} \right\} \right).
 \end{align*}
 Finally, using the Markov property from Assumption \ref{asm:JMarkov} and repeating the argument in the proof of equation (4.6) in \cite{KP25+} to bound the last probability above, we get that
 \[
 \limsup_{n\to\infty} P\left( A_{x,h,\d,\e,n}^{(K),+} \right)
 \leq \left\{ \sup_{k\in\Z} \sup_{\ell \in \Z_+} P\left( \mu_{(k,\ell),(k,\ell+1)}^+ \geq k + \frac{1}{2} \right)  \right\}^{\fl{\frac{\e}{5\d}}}
 \leq (1-c_0)^{\fl{\frac{\e}{5\d}}}, 
 \]
where the last inequality follows from Assumption \ref{asm:ND}.  

Having completed the proof of \eqref{PAxh}, we now show how to finish the proof of the proposition. First of all, for any $\e,\d,\d' \in (0,1)$ and $K\geq 1$ we can conclude from \eqref{Xtightdec} and \eqref{PAxh} that 
\begin{align*}
 \limsup_{n\to\infty}\, &P\left(
  \max_{\substack{k,\ell \leq n g(n) \\ |k-\ell|\leq \d' n g(n) }} 
  |X_k - X_\ell| > \e n \right) \\
  &\leq \limsup_{n\to\infty} \left\{ 
  P(R_{K,1,n}^c) 
  + P\left( \text{gap}\left(\mathcal{T}_{\d,n}^{(K)}\right) \leq \d' \right) 
  \right\}
  +  \frac{12 K^{2+\gamma}}{\d^{\gamma+1}} (1-c_0)^{\left\lfloor \frac{\e}{5\d} \right\rfloor}, 
\end{align*}
and so we need only to show that for $\e,\eta>0$ fixed we can choose
the constants $\d,\d'$ and $K$ so that the right
side above is at most $\eta$.  To see this, we first choose $K\geq 1$
(depending only on $\eta$) as in Lemma \ref{lem:RKn}, then choose
$\d\in(0,1)$ sufficiently small so that
$\frac{12 K^{2+\gamma}}{\d^{\gamma+1}} (1-c_0)^{\left\lfloor
    \frac{\e}{5\d} \right\rfloor} < \frac{\eta}{3}$, and finally
choose $\d'\in (0,1)$ (depending on $K$, $\eta$, and $\delta$) as in
Lemma \ref{lem:gap}.
\end{proof}

\subsection{Tightness of the local times.}
The main result of this subsection is the tightness of the rescaled
local time profiles $\mathcal{L}_n$ defined in \eqref{lntx}. 

\begin{proposition}\label{prop:Lntight}
For any fixed $T > 0$ and $\epsilon>0$ there exists $\delta>0$ such that
\[
\limsup_{n \to\infty} P \left(  \sup_{\substack{s,s'\le T,\,|s-s'|<\delta\\x,x'\in\R,\,|x-x'|<\delta}}  \big\vert  {\cal L}_n(s,x) - {\cal L}_n(s', x')  \big\vert  > \epsilon \right) < \epsilon.
\]
\end{proposition}

Before giving the proof of Proposition \ref{prop:Lntight}, we note
that this statement implies the following tightness result for the rescaled local times profiles. 
\begin{corollary}\label{cor:Lntight}
  $(\mathcal{L}_n(t,\cdot))_{t\geq 0}$ is tight as a family of $D([0,\infty);D(\R))$-valued processes.
\end{corollary}

Proposition \ref{prop:Lntight} follows immediately from the
  following two lemmas.
\begin{lemma}\label{lem:Lntight-time}
 For any fixed $T>0$ and $\epsilon>0$ there exists a $\delta>0$ such that 
 \[
 \limsup_{n\to\infty}  P\left( \sup_{\substack{s,s' \leq T,\, |s-s'| < \delta\\x\in \R}} \left| \mathcal{L}_n(s,x) - \mathcal{L}_n(s',x) \right| > \frac{\e}{2} \right) < \frac{\e}{2}. 
 \]
\end{lemma}
\begin{lemma}\label{lem:Lntight-space}
 For any fixed $T>0$ and $\epsilon>0$ there exists a $\delta>0$ such that 
 \[
 \limsup_{n\to\infty}  P\left( \sup_{\substack{x,x' \in \R,\,|x-x'| < \d\\s \leq T}} \left| \mathcal{L}_n(s,x) - \mathcal{L}_n(s,x') \right|  > \frac{\e}{2} \right) < \frac{\e}{2}.
 \]
\end{lemma}

The key to our proofs of Lemmas \ref{lem:Lntight-time} and
\ref{lem:Lntight-space} will be showing how control of the local time
profiles at finitely many stopping times $\tau_{x,h}^n$ (i.e.,
finitely many curves $\Lambda^{n}_{x,h}(\cdot)$)
can be used to control changes in the local time profiles at
deterministic times.
We will need the following result, which we will show in Appendix \ref{app:Bessel} follows from Assumption \ref{asm:Bessel} and similar tail bounds for BESQ processes.
\begin{corollary}\label{cor:Lxh-dev}
 For any fixed $r>0$ there exist constants $C, c'_r>0$ where $C$ depends only on the parameters from Assumption \ref{asm:Bessel} and $c'_r$ depends on these parameters and also $r$, such that 
 \begin{align}\label{Lxh-dev}
  &\sup_{x \in \R, \, h \leq L} 
  P\left( \sup_{|y-x|\leq T} \left| \Lambda_{x,h}(y) - h \right| \geq r L \right)
  \leq C e^{-c'_r L^{1/\gamma}/T}, \qquad \forall\, L,T>0, \\
 \shortintertext{and} 
\label{Lgap}
  &\sup_{x \in \R, \, h \leq L} P\left( \sup_{|y-x|\leq T} ( \Lambda_{x,h+1}(y) - \Lambda_{x,h}(y) )  \geq r L \right) \leq C e^{-c'_r L^{1/\gamma}/T}, \quad \forall\, L \geq \frac{3}{r},\ T>0.
 \end{align}
\end{corollary}
Now we are ready to give the proofs of Lemmas
\ref{lem:Lntight-time} and \ref{lem:Lntight-space}.

\begin{proof}[Proof of Lemma \ref{lem:Lntight-time}]
For convenience of notation, for any $\d,\eta,T>0$ and $n\geq 1$ let 
\[
 J_{\d,\eta,T,n} = \left\{ \sup_{s,s'\leq T,\,|s-s'|<\d} \left| \mathcal{X}_n(s) - \mathcal{X}_n(s') \right| \leq \eta \right\}. 
\]
Then, for any fixed parameters $\d>0$, $\eta\in(0,1)$ and $K\geq 1$ (to be chosen appropriately later) we have
\begin{multline}
 P\left( \sup_{\substack{s,s' \leq T,\, |s-s'| < \delta\\x\in \R}} \left| \mathcal{L}_n(s,x) - \mathcal{L}_n(s',x) \right|  > \frac{\e}{2} \right) \\
 \leq P\left( R_{K,T,n}^c \right) + P\left( J_{\d,\eta,T,n}^c \right)
 + P\left( \sup_{\substack{s,s' \leq T,\, |s-s'| < \delta\\x\in [-K,K]}} \left| \mathcal{L}_n(s,x) - \mathcal{L}_n(s',x) \right|  > \frac{\e}{2}, \, R_{K,T,n}, \, J_{\d,\eta,T,n} \right). \label{Lntimebound}
\end{multline}
To bound the last probability in the line above, suppose that
$s<s'<s+\d$, $s' \leq T$ and $x\in [-K,K]$ are such that
$\mathcal{L}_n(s',x)-\mathcal{L}_n(s,x) > \frac{\e}{2}$. Then, the
event $J_{\d,\eta,T,n}$ implies that when $n$ is large enough with
respect to $\eta$, there is an integer $k$ with
$k\eta < x \leq (k+3)\eta$ and such that
$\mathcal{L}_n(s,y) = \mathcal{L}_n(s',y)$ for all $y\leq k\eta$ and
since $x\in [-K,K]$ we have $|k| \leq K/\eta +3$. For this $s$ and
$k$, let $\ell \geq 0$ be the integer such that
$\tau_{k\eta,\ell \eta^\gamma}^n \leq s ng(n) < \tau_{k\eta,(\ell+1)
  \eta^\gamma}^n$, if it exists (and note that the event $R_{K,T,n}$
implies that $\ell \leq K^{1+\gamma}/\eta^{\gamma}$).  Since the
condition $\mathcal{L}_n(s,k\eta) = \mathcal{L}_n(s',k\eta)$ implies
furthermore that $s'ng(n) < \tau_{k\eta, (\ell+1)\eta^\gamma}^n$, then
we can conclude that
\[
\Lambda_{k\eta,(\ell+1)\eta^\gamma}^n(x) - \Lambda_{k\eta, \ell\eta^\gamma}^n(x) = \mathcal{L}_n\left( \frac{\tau_{k\eta,(\ell+1)\eta^\gamma}^n}{ng(n)},x \right) 
- \mathcal{L}_n \left(\frac{\tau_{k\eta,\ell\eta^\gamma}^n}{ng(n)},x \right) 
\geq \mathcal{L}_n(s',x)
- \mathcal{L}_n(s,x) > \frac{\e}{2}.
\]
If $\ell$ does not exist, it means $\lfloor k\eta n\rfloor$ was not visited before time $s n g(n)$, and since $\mathcal{L}_n(s,k\eta) = \mathcal{L}_n(s',k\eta)$ we have that $\fl{k \eta n}$ has not been visited by time $s' n g(n)$. Hence $s'ng(n) \leq \tau_{k\eta,0}^n$ and we get $\Lambda_{k\eta,0}^n(x) =  \mathcal{L}_n(\frac{\tau_{k\eta,0}^n}{n g(n)}, x) \geq \mathcal{L}_n(s',x)>\frac{\varepsilon}{2}$.
Therefore, it follows from the joint RKTs and the scaling property \eqref{Lxh-rescale} that
\begin{align*}
 &\limsup_{n\to\infty}  P\left( \sup_{\substack{s,s' \leq T,\, |s-s'| < \delta\\x\in [-K,K]}} \left| \mathcal{L}_n(s,x) - \mathcal{L}_n(s',x) \right|  > \frac{\e}{2}, \, R_{K,T,n}, \, J_{\d,\eta,T,n} \right) \\
 &\leq P\left( \max_{\substack{|k|\leq K/\eta+3 \\  \ell\leq K^{1+\gamma}/\eta^\gamma}} \sup_{x \in [k\eta, (k+3)\eta]} ( \Lambda_{k\eta,(\ell+1)\eta^\gamma}(x) - \Lambda_{k\eta,\ell \eta^\gamma}(x) ) \geq \frac{\e}{2}  \right)\\
 &\qquad + P\left(\max_{|k|\leq K/\eta+3}\sup_{x \in [k\eta, (k+3)\eta]} \Lambda_{k\eta,0}(x)\geq \frac{\e}{2}\right) \\
 &\leq \left(\frac{2K}{\eta}+7\right) \left( \frac{K^{1+\gamma}}{\eta^\gamma} + 1 \right) \max_{\substack{|k|\leq K/\eta+3 \\  \ell\leq K^{1+\gamma}/\eta^\gamma}} P\left( \sup_{x \in [k\eta, (k+3)\eta]} ( \Lambda_{k\eta,(\ell+1)\eta^\gamma}(x) - \Lambda_{k\eta,\ell \eta^\gamma}(x) ) \geq \frac{\e}{2}  \right) \\
 &\qquad   +\left(\frac{2K}{\eta}+7\right) \max_{|k|\leq K/\eta+3}P\left(\sup_{x \in [k\eta, (k+3)\eta]} \Lambda_{k\eta,0}(x)\geq \frac{\e}{2}\right) \\
 &= \left(\frac{2K}{\eta}+7\right) \left( \frac{K^{1+\gamma}}{\eta^\gamma} + 1 \right) \max_{\substack{|k|\leq K/\eta+3 \\  \ell\leq K^{1+\gamma}/\eta^\gamma}} P\left( \sup_{x \in [k, k+3]} ( \Lambda_{k,\ell+1}(x) - \Lambda_{k,\ell}(x) ) \geq \frac{\e}{2\eta^\gamma}  \right) \\
 & \qquad  +\left(\frac{2K}{\eta}+7\right) \max_{|k|\leq K/\eta +3 }P\left(\sup_{x \in [k, k+3]} \Lambda_{k,0}(x)\geq \frac{\e}{2\eta^\gamma}\right) \\
 &\leq C \frac{K^{2+\gamma}}{\eta^{1+\gamma}} e^{-c_{\e,K}/\eta},
\end{align*}
where the last inequality follows from Corollary \ref{cor:Lxh-dev} as long as $\frac{\e}{2\eta^\gamma} \geq 3$, and the constant $c_{\e,K}$ in the last line depends on $\e>0$, $K \geq 1$ and the parameters from Assumption \ref{asm:Bessel}, but not on $\eta$.
Recalling \eqref{Lntimebound}, we obtain that 
\begin{align*}
 &\limsup_{n\to\infty} P\left( \sup_{\substack{s, s' \leq T, \, x \in \R\\|s-s'| < \d}} \left| \mathcal{L}_n(s,x) - \mathcal{L}_n(s',x) \right|  > \frac{\e}{2} \right) \\
 &\qquad \leq C \frac{K^{2+\gamma}}{\eta^{1+\gamma}} e^{-c_{\e,K}/\eta} + \limsup_{n\to\infty} \left\{ P\left( R_{K,T,n}^c \right) + P\left( J_{\d,\eta,T,n}^c \right) \right\}.
\end{align*}
Finally, for a fixed $\e,T>0$ we choose the parameters $K,\eta,\d$ as follows. 
\begin{enumerate}[--]
 \item Using Lemma \ref{lem:RKn}, choose $K$ depending on $\e$ and $T$ so that $\limsup_{n\to\infty}P(R_{K,T,n}^c) \leq \frac{\e}{6}$. 
 \item Then choose $\eta\in (0,1)$ depending on $\e$ and $K$ so that $\frac{\e}{2\eta^\gamma} \geq 3$ and $C \frac{K^{2+\gamma}}{\eta^{1+\gamma}} e^{-c_{\e,K}/\eta} < \frac{\e}{6}$.
 \item Finally, using the tightness of $\mathcal{X}_n$ (or more precisely \eqref{Xtight-cond}) we can choose $\d>0$ depending on $\e$, $\eta$, and $T$ so that $\limsup_{n\to\infty} P(J_{\d,\eta,T,n}^c) \leq \frac{\e}{6}$.
\end{enumerate}
\end{proof}

\begin{proof}[Proof of Lemma \ref{lem:Lntight-space}]
Suppose that $s\leq T$ and $x<x'<x+\d$ are such that 
$|\mathcal{L}_n(s,x)-\mathcal{L}_n(s,x')| > \frac{\e}{2}$. We also assume $R_{K,T,n}$ occurs. Since on the event $R_{K,T,n}$ we have that $L(\lfloor sng(n) \rfloor,k)=0$ for all $|k|>Kn$,
it follows that  when $n$ is large enough with respect to $\d$,
there is an integer $k$ with $|k|\leq K/\d+1$  such that 
\begin{equation}\label{Lnspatial1}
 \sup_{|y-k\d|\leq 2\d} |\mathcal{L}_n(s,k\d)-\mathcal{L}_n(s,y)| > \frac{\e}{4} \qquad \text{and}\qquad \mathcal{L}_n(s,k\d)>0. 
\end{equation}
(This can be accomplished for instance by taking $|k|$ maximal with $|k\delta| \leq |x| \wedge |x'|$.)
For this $s\leq T$ and $|k|\leq K/\d + 1$ there is also a non-negative integer $\ell$ with $\tau_{k\d, \ell\d^\gamma}^n \leq s n g(n) < \tau_{k\d, (\ell+1)\d^\gamma}^n$ (and on the event $R_{K,T,n}$ we have that $\ell \leq K^{1+\gamma}/\d^\gamma$), so that
\[
 \Lambda_{k\d, \ell \d^\gamma}^n(y) = \mathcal{L}_n\left( \frac{\tau_{k\d, \ell \d^\gamma}^n}{ng(n)}, y \right) \leq \mathcal{L}_n(s,y)
 \leq \mathcal{L}_n\left( \frac{\tau_{k\d, (\ell+1) \d^\gamma}^n}{ng(n)}, y \right)
 = \Lambda_{k\d, (\ell+1)\d^\gamma}^n(y)
 \qquad \forall\, y \in \R,
\]
where the second inequality is a strict inequality if $y=k\d$; note that letting $y=k\d$ this implies that $\ell \d^\gamma \leq \mathcal{L}_n(s,k\d) \leq (\ell+1)\d^\gamma$. Therefore, \eqref{Lnspatial1}
implies for this $k$ and $\ell$ that 
\begin{equation}\label{Lnspatial2}
 \text{either}
 \quad
 \inf_{|y-k\d|\leq 2\d} \Lambda_{k\d,\ell\d^\gamma}^n(y) < (\ell+1)\d^\gamma - \frac{\e}{4}
 \quad\text{ or }\quad 
 \sup_{|y-k\d|\leq 2\d} \Lambda_{k\d,(\ell+1)\d^\gamma}^n(y) > \ell \d^\gamma + \frac{\e}{4}.
\end{equation}
Moreover, if $\d^\gamma < \frac{\e}{8}$ then we can simplify the bounds in \eqref{Lnspatial2} since $(\ell+1)\d^\gamma - \frac{\e}{4} < \ell \d^\gamma - \frac{\e}{8}$ and $\ell \d^\gamma + \frac{\e}{4} > (\ell+1)\d^\gamma + \frac{\e}{8}$.
Thus, for any $K\geq 1$, $\e,T>0$ and $\d^\gamma < \frac{\e}{8}$ we have from the joint RKTs and the scaling property \eqref{Lxh-rescale} that
\begin{align*}
& \limsup_{n\to\infty} P\left( \sup_{\substack{s \leq T, \, x,x' \in \R\\|x-x'| < \d}} \left| \mathcal{L}_n(s,x) - \mathcal{L}_n(s,x') \right|  > \frac{\e}{2}, \, R_{K,T,n} \right)
\\
&\leq P\left( \max_{\substack{|k|\leq K/\d + 1 \\ \ell \leq K^{1+\gamma}/\d^\gamma + 1}} \sup_{|y-k\d|\leq 2\d} \left|  \Lambda_{k\d,\ell\d^\gamma}(y) - \ell\d^\gamma \right| \geq \frac{\e}{8} \right)  \\
&\leq \left( \frac{2K}{\d}+ 3 \right) \left( \frac{K^{1+\gamma}}{\d^\gamma} + 2 \right) \max_{\substack{|k|\leq K/\d + 1 \\ \ell \leq K^{1+\gamma}/\d^\gamma + 1}} P\left( \sup_{|y-k\d|\leq 2\d}  \left|  \Lambda_{k\d,\ell\d^\gamma}(y) - \ell\d^\gamma \right| \geq \frac{\e}{8} \right)\\
&= \left( \frac{2K}{\d}+ 3 \right) \left( \frac{K^{1+\gamma}}{\d^\gamma} + 2 \right) \max_{\substack{|k|\leq K/\d + 1 \\ \ell \leq K^{1+\gamma}/\d^\gamma + 1}} P\left( \sup_{y\in[k-2,k+2]}  \left|  \Lambda_{k,\ell}(y) - \ell \right| \geq \frac{\e}{8\d^\gamma} \right) \\
&\leq C \frac{K^{2+\gamma}}{\d^{1+\gamma}} e^{-c_{\e,K}/\d},
\end{align*}
where again by Corollary \ref{cor:Lxh-dev} the constant $c_{\e,K}$ in the last line depends on $\e>0$, $K \geq 1$ and the parameters from Assumption \ref{asm:Bessel}, but not on $\d \in(0,1)$.
Finally, we have that 
\[
 \limsup_{n\to\infty} P\left( \sup_{\substack{s \leq T, \, x,x' \in \R\\|x-x'| < \d}} \left| \mathcal{L}_n(s,x) - \mathcal{L}_n(s,x') \right|  > \frac{\e}{2} \right)
 \leq C \frac{K^{2+\gamma}}{\d^{1+\gamma}} e^{-c_{\e,K}/\d} + \limsup_{n\to\infty} P(R_{K,T,n}^c), 
\]
and the proof is completed by first using Lemma \ref{lem:RKn} to choose $K$ large enough so that the limsup on the right is at most $\e/4$ 
and then choosing $\d>0$ sufficiently small so that $\d^\gamma < \e/8$ and the first term on the right is at most $\e/4$.
\end{proof}

\section{Proofs of Theorems \ref{thm:main_X}--\ref{thm:main_JRK} and Corollary \ref{xh}}\label{sec:fdd}

\begin{proof}[Proof of Theorem~\ref{thm:main_X}]
  For a $K\geq 1$ and $\d>0$ enumerate the finitely many points in the grid $Q_\d^{(K)}$ defined in \eqref{TdnK-def} as
  $(x_1,h_1),\ldots ,(x_M,h_M)$ and introduce the following notation for rescaled discrete and continuous inverse local times at each point of the grid: for
    $m\in\{1,\ldots ,M\}$, we set
    $\theta_{m}^n:=(ng(n))^{-1}\tau_{x_{m},h_{m}}^{n}$,
    $n\in\mathbb{N}$, and
    $\theta_{m}:=\int_\R\Lambda_{x_{m},h_{m}}(y)\,dy$.

    For any $n\in\mathbb{N}$, $t \in (0,\infty)$, we define $N_n(t)$
    as ``the random index $m$ such that $\theta_{m}^n$ is the closest to
    $t$ from below''; more rigorously, for $m\in\{1,\ldots,M\}$, we
    have $N_n(t)=m$ when $\theta_{m}^n \leq t$ and for any
    $m'\in\{1,\ldots ,M\}$, $m' \neq m$, we have
    $\theta_{m'}^n\not\in (\theta_{m}^n,t]$. If there are several $m$
    satisfying this condition, we choose the smallest, but when $n$ is
    large enough, the $\theta_{m}^n$ are all different, so this does not
    happen. Similarly, $N(t)$ will be ``the random index $m$ such that
    $\theta_{m}$ is the closest to $t$ from below'': for
    $m\in\{1,\ldots ,M\}$, we have $N(t)=m$ when $\theta_{m} \leq t$ and
    for any $m'\in\{1,\ldots ,M\}$, $m' \neq m$, we have
    $\theta_{m'}\not\in (\theta_{m},t]$ ($m$ is almost surely unique by Lemma \ref{lem_int_diff}).  Note that $N_n(t)$ and $N(t)$
    also depend on $\d$ and $K$, but we will suppress this dependence
    to keep the notation manageable.

\begin{lemma}\label{lem_conv_indexes}
 Fix arbitrary $\delta>0$ and $K\ge 1$. Then for any $j\in\mathbb{N}$ and  $0 < t_1 < \ldots < t_j$,  the vector $(N_n(t_1),\ldots ,N_n(t_j))$ converges in distribution to $(N(t_1),\ldots ,N(t_j))$ as $n\to \infty$.
\end{lemma}

\begin{proof}
 Let $\ell_1,\ldots ,\ell_j\in\{1,\ldots ,M\}$. We have
 \[
P((N_n(t_1),\ldots ,N_n(t_j))=(\ell_1,\ldots ,\ell_j))=P\left(\bigcap_{i=1}^j \left\{\theta_{\ell_i}^n \leq t_i,\forall\, i' \neq \ell_i, \theta_{i'}^n\not\in (\theta_{\ell_i}^n,t_i]\right\}\right),
 \]
 and a similar equality holds for $P((N(t_1),\ldots ,N(t_j))=(\ell_1,\ldots ,\ell_j))$ with $\theta_{m}$ instead of $\theta_{m}^n$. Moreover, by Proposition \ref{prop_int_no_atoms}, for any $m\in\{1,\ldots ,M\}$, $t>0$, we have $P(\theta_{m}=t)=0$, and Lemma \ref{lem_int_diff} implies that for any $m \neq m'$ in $\{1,\ldots ,M\}$, we have $P(\theta_{m} = \theta_{m'})=0$. Therefore, the lemma follows from the Portmanteau Theorem and Lemma \ref{lem:txhn-conv}.
\end{proof}

With the above notation, for fixed $\d>0$, $K\geq 1$ and $n\geq 1$ we can define $\mathcal{X}_{\d,n}^{(K)}(t) = x_{N_n(t)}$. Our next lemma allows us to compare the finite dimensional distributions of $\mathcal{X}_n$ with those of $\mathcal{X}_{\d,n}^{(K)}$. To this end, for any $j\in\mathbb{N}$, we denote by $d_j$ the Prokhorov metric (see for example \cite[Ch.\,3, Section 1]{ekMP}) on the space of probability measures on
$(\R^j,\|\cdot\|_\infty)$. By a standard abuse of notation, we may
substitute random variables for their laws in the Prokhorov metric.

\begin{lemma}\label{lem_Prohorov}
 For any $j\in\mathbb{N}$, $0 < t_1 < \cdots < t_j$, and any $\varepsilon>0$, there exist $\delta>0$ and $K\geq 1$ such that $d_{j}((\mathcal{X}_n(t_i))_{1 \leq i \leq j},(\mathcal{X}_{\delta,n}^{(K)}(t_i))_{1 \leq i \leq j}) \leq \varepsilon$ when $n$ is large enough.
\end{lemma}

\begin{proof}
  Given $\epsilon>0$ we shall use Lemma~\ref{lem:RKn} to choose
  $K\geq 1$ so that $P(R_{K-1,t_j,n}^c) \leq \frac{\varepsilon}{2}$
  for all large $n$. Then we use \eqref{PAxh} to choose $\delta>0$ so
  that
  $\sum_{m=1}^M P(A_{x_{m},h_{m},\delta,\varepsilon,n}^{(K)}) \leq
  \frac{\varepsilon}{2}$ for all sufficiently large $n$ satisfying
  $n>2/\epsilon$.  It is enough to prove that if
  $R_{K-1,t_j,n}^c\cup\bigcup_{m=1}^M
  A_{x_{m},h_{m},\delta,\varepsilon,n}^{(K)}$ does not occur, then for
  all $1 \leq i \leq j$ we have
  $|\mathcal{X}_n(t_i)-\mathcal{X}_{\delta,n}^{(K)}(t_i)| <
  \varepsilon$.

  We will argue that on the event
    $R_{K-1,t_j,n} \cap (A_{x_{m},h_{m},\d,\e,n}^{(K)})^c$, if $\tau_{x_{m},h_{m}}^n \leq ng(n)t_j$, the walk does
    not move more than distance $\e n/4$ between the stopping time
    $\tau_{x_{m},h_{m}}^n$ and any other of the stopping times
    $\tau_{x,h}^n$ with $(x,h) \in Q_\d$. Suppose
  $R_{K-1,t_j,n}\cap\bigcap_{m=1}^M
  (A_{x_{m},h_{m},\delta,\varepsilon,n}^{(K)})^c$ occurs. We fix
  $i\in\{1,\ldots,j\}$ and, in order to lighten the notation,
  temporarily write $(x_{N_n(t_i)},h_{N_n(t_i)})$ as
  $(\tilde x,\tilde h)$. Since $\theta_{N_n(t_i)}^n \leq t_i$, or
  equivalently $\tau_{\tilde x,\tilde h}^{n} \leq ng(n)t_i$, and the event $R_{K-1,t_j,n}$ occurs, we can conclude
  that
  $\max_{k\in\mathbb{Z}}L(\tau_{\tilde x,\tilde h}^n,k)\leq
  K^{1+\gamma}g(n)$, so that, as
  $A_{\tilde{x},\tilde{h},\d,\e,n}^{(K)}$ does not occur, we have
\[
\left|X_{\lfloor t_ing(n)\rfloor}-X_{\tau_{\tilde x,\tilde h}^n}\right|\leq
    \frac{\varepsilon}{4}n\ \ \text{ or }\ \
    \mathcal{T}_{\delta,n}\cap \left(\frac{\tau_{\tilde x,\tilde
          h}^n}{ng(n)},\frac{\lfloor
        t_ing(n)\rfloor}{ng(n)}\right)\neq\emptyset.
\]
The definitions of $R_{K-1,t_j,n}$ and $N_n(t_i)$ imply that
  $\mathcal{T}_{\delta,n}\cap \left(\frac{\tau_{\tilde x,\tilde
        h}^n}{ng(n)},\frac{\lfloor
      t_ing(n)\rfloor}{ng(n)}\right)=\emptyset$, hence
$|X_{\lfloor t_ing(n)\rfloor}-X_{\tau_{\tilde x,\tilde h}^n}|\leq
  \frac{\varepsilon}{4}n$ and
\begin{align*}
 \left|\mathcal{X}_n(t_i)-\mathcal{X}_{\delta,n}^{(K)}(t_i)\right|
 =  \left| \frac{X_{\fl{t_i n g(n)}} }{n} - \frac{X_{\tau_{\tilde{x},\tilde{h}}^n}}{n} +  \frac{\fl{\tilde{x}n} }{n} - \tilde{x} \right|
 \leq \frac{\e}{4} + \frac{1}{n} < \e,
\end{align*}
as claimed.
\end{proof}

We can now finish the proof of Theorem~\ref{thm:main_X}. Let us first prove
the convergence in distribution of the finite-dimensional marginals of
$(\mathcal{X}_n(t))_{t \geq 0}$. Fix $j\in\mathbb{N}$ and
$0 < t_1 < \cdots < t_j$.
For ease of notation, define the following probability measures on $\R^j$ for any $\d>0$, $K\geq 1$, and $n\geq 1$.
\begin{enumerate}[--]
 \item $\mu_n$ is the distribution of $(\mathcal{X}_n(t_i))_{1\leq i\leq j}$.
 \item $\mu_{\d,n}^{(K)}$ is the distribution of  $(\mathcal{X}_{\d,n}^{(K)}(t_i))_{1\leq i\leq j}$.
 \item $\mu_{\d}^{(K)}$ is the distribution of $\left( x_{N(t_i)} \right)_{1\leq i\leq j}$.
\end{enumerate}
For any $\varepsilon>0$, Lemma \ref{lem_Prohorov} allows us to choose
$\delta>0$, $K \geq 1$ so that when $n$ is large enough,
$d_{j}(\mu_n,\mu_{\d,n}^{(K)} ) \leq
\frac{\varepsilon}{4}$. For this fixed $\d>0$ and $K\geq 1$,
Lemma \ref{lem_conv_indexes} implies that
$(\mathcal{X}_{\delta,n}^{(K)}(t_i))_{1 \leq i \leq j}$ converges in
distribution to $(x_{N(t_i)})_{1\leq i\leq j}$ as
$n\to\infty$, and since convergence in distribution is
equivalent to convergence under the Prokhorov metric, when $n$ is large enough we have
  $d_j(\mu_{\d,n}^{(K)}, \mu_\d^{(K)} ) \leq \frac{\e}{4}$.
Therefore, for any $\e>0$ we can choose $\d>0$ and $K\geq 1$ so that when $n,m$ are large enough,
\begin{align*}
 d_j( \mu_n,\mu_m)
 &\leq d_j\left(\mu_n, \mu_{\d, n}^{(K)} \right)
 + d_j\left(\mu_{\d, n}^{(K)}, \mu_{\d}^{(K)} \right)
 + d_j\left(\mu_{\d, m}^{(K)}, \mu_{\d}^{(K)} \right)
 + d_j\left(\mu_m, \mu_{\d, m}^{(K)} \right) \\
 &\leq \frac{\varepsilon}{4}+\frac{\varepsilon}{4}+\frac{\varepsilon}{4}+\frac{\varepsilon}{4}=\varepsilon.
\end{align*}
Consequently, $(\mu_n)_{n\in\mathbb{N}}$ is a Cauchy sequence and, since the space of probability measures on $\R^{j}$ with the Prokhorov metric is complete, the sequence $(\mu_n)_{n\in\mathbb{N}}$ converges under the Prokhorov metric when $n$ tends to $\infty$. This implies that $(\mathcal{X}_n(t_i))_{1 \leq i \leq j}$ converges in distribution as $n\to\infty$.

 By Proposition \ref{prop:Xtight} and Prokhorov's theorem, the sequence
$(\mathcal{X}_n)_{n \geq 1}$ is relatively compact with respect to the
Skorokhod $J_1$ topology on $D([0,+\infty))$, and every subsequential
limit is concentrated on continuous paths. The finite-dimensional
distributions of any subsequential limit must then coincide with the
limits of the finite-dimensional distributions of
$(\mathcal{X}_n)_{n \geq 1}$, which were shown above to exist. Since the
law of a process with continuous paths is determined by its
finite-dimensional distributions, all subsequential limits have the same
law. Therefore, $(\mathcal{X}_n)_{n \geq 1}$ converges in distribution to
a continuous process. This proves Theorem \ref{thm:main_X}.
\end{proof}

\begin{proof}[Proof of Theorem~\ref{thm:main_XL}]
First of all, note that it follows from Theorem \ref{thm:main_X} and Corollary \ref{cor:Lntight} that the sequence of processes
  $(\mathcal{X}_n(t), {\cal L}_n(t,\cdot))_{t\ge 0}$ is relatively
  compact in $D([0,\infty)) \times D([0,\infty); D(\R))$.
We know that ${\cal X}_n\Rightarrow \mathfrak{X}$ in $D([0,\infty))$ and
that the limiting process is a.s.\ continuous. The continuous mapping
theorem and continuity of $\mathfrak{X}$ imply that for every $T>0$ as
$n\to\infty$
\[ \inf_{s\in [0,T]}{\cal X}_n(s)\Longrightarrow \inf_{s\in
    [0,T]}\mathfrak{X}(s)>-\infty\ \ \text{ and } \sup_{s\in[0,T]}{\cal
    X}_n(s)\Longrightarrow \sup_{s\in[0,T]}\mathfrak{X}(s)<\infty.\]
This fact and Proposition~\ref{prop:Lntight} imply that every limit point of
$({\cal L}_n)_{n\ge 1}$ is a.s.\ in $C([0,\infty);C_0(\R))$.
Let $(\mathfrak{X},\mathfrak{L}')$ be a limit point of
$({\cal X}_n,{\cal L}_n)_{n\ge 1}$, so that along a subsequence (not
relabeled)
$({\cal X}_n,{\cal L}_n)\Longrightarrow (\mathfrak{X},\mathfrak{L}')$. Note
that for every $f\in C_B(\R)$, $t\ge 0$, and all
$n\in\mathbb{N}$,
  \begin{equation}
    \label{otf1}
    \frac{1}{ng(n)}\sum_{k=0}^{\fl{ng(n)t}}f\left(\frac{X_k}{n}\right)=\sum_{z\in\Z}f\left(\frac{z}{n}\right)\frac{L(\fl{ng(n)t},z)}{g(n)}\,\frac1n.
  \end{equation}
  The equality \eqref{otf1} can be re-written in the integral form as follows:
\begin{equation}
  \label{otf2}
  \int_0^{t_n}f({\cal X}_n(s))\,ds=\int_{{\cal I}_n(t)} f\left(\frac{\fl{xn}}{n}\right){\cal L}_n(t,x)\,dx,
\end{equation}
where $t_n=\frac{\fl{ng(n)t}+1}{ng(n)}$ and ${\cal I}_n(t)$ is the
support of ${\cal L}_n(t,\cdot)$ --  a bounded interval. By the
Skorokhod representation, without loss of generality, we can assume
that $({\cal X}_n,{\cal L}_n)\to (\mathfrak{X},\mathfrak{L}')$ a.s.\  with
$(\mathfrak{X},\mathfrak{L}')\in C([0,\infty))\times
C([0,\infty);C_0(\R))$ a.s. Next we shall show that we can
let $n\to\infty$ in \eqref{otf2} and conclude that, a.s.,
\begin{equation}
  \label{otf3}
  \int_0^t f(\mathfrak{X}(s))\,ds=\int_{{\cal I}_0(t)} f\left(x\right)\mathfrak{L}'(t,x)\,dx,
\end{equation}
where ${\cal I}_0(t)$ is the support of $\mathfrak{L}'(t,\cdot)$ -- a
bounded interval.  

The convergence of the left-hand side of \eqref{otf2} to the left-hand
side of \eqref{otf3} is immediate, since $f\in C_B(\R)$ and
$t_n\to t$. Note that
${\cal I}_n(t)\subseteq[\inf_{s\in [0,t+1]}{\cal
  X}_n(s),\sup_{s\in[0,t+1]}{\cal X}_n(s)+1]$. Hence there is a bounded interval
${\cal J}(t)\supset \cup_{n=0}^\infty\, {\cal I}_n(t)$ on which for all $n$ the absolute
value of the difference of the integrands in the right-hand sides is bounded above by
\begin{align*}
  &\left|f\left(\frac{\fl{xn}}{n}\right)({\cal L}_n(t,x)-\mathfrak{L}'(t,x))\right|+\left|\mathfrak{L}'(t,x)\left(f\left(\frac{\fl{xn}}{n}\right)-f(x)\right)\right|\\ 
  &\le \sup_{x\in{\cal J}(t)}|f(x)|\cdot\sup_{x\in {\cal J}(t)}|{\cal L}_n(t,x)-\mathfrak{L}'(t,x)|+\sup_{x\in{\cal J}(t)}\mathfrak{L}'(t,x)\cdot \sup_{x\in {\cal J}(t)}\left|f\left(\frac{\fl{xn}}{n}\right)-f(x)\right|.
\end{align*}
Clearly, the integral over ${\cal J}(t)$ of the last term vanishes as
$n\to\infty$. We recall that $\mathfrak{L}'\in C([0,\infty);C_0(\R))$, therefore the convergence of ${\cal L}_n$ to $\mathfrak{L}'$ is locally
  uniform. Hence, the integral of the first term over ${\cal J}(t)$
vanishes as well as $n\to\infty$.

With probability one, the above convergence holds for all rational
$t\in[0,\infty)$ and any countable set of functions $f\in
C_B(\R)$. Since $\mathfrak{L}'$ and $\mathfrak{X}$ are continuous,
a standard approximation of indicator functions by a countable
  set of continuous functions gives \eqref{otf} with $\mathfrak{L}'$
in place of $\mathfrak{L}$. Thus, every limit point $\mathfrak{L}'$
must satisfy \eqref{otf}. This gives us the uniqueness of the limit
point and completes the proof.
\end{proof}

\begin{proof}[Proof of Theorem~\ref{thm:main_JRK}]  
By the Ray-Knight theorem for the walks $(X_n)_{n\ge 0}$ and Theorem~\ref{thm:main_XL}, for each $x\in\R$, almost surely $\mathfrak{L}(t,x)\to\infty$ as $t\to\infty$. Hence, the inverse local times $(\mathfrak{t}_{x_i,h_i})_{1\le i\le N}$, are finite with probability 1\footnote{Alternatively, the finiteness of $(\mathfrak{t}_{x_i,h_i})_{1\le i\le N}$ can be seen from \eqref{eq_txh_finite}.}.
Since the curves $\mathfrak{L}(\mathfrak{t}_{x_i,h_i}, \cdot)$ are continuous, to show that $\mathfrak{X}$ has the claimed joint RKTs it is enough to show that the finite dimensional distributions are the same as those of $\Lambda_{x_i,h_i}(\cdot)$.
That is, we need to show that for any $(x_1,h_1),\ldots,(x_N,h_N) \in \R\times[0,\infty)$ and $y_1,y_2,\ldots,y_m \in \R$,
\begin{equation}\label{JRK-fdd}
 \left( \mathfrak{L}(\mathfrak{t}_{x_i,h_i}, y_j)  \right)_{i\leq N, \, j\leq m}
 \overset{\text{Law}}{=} \left( \Lambda_{x_i,h_i}(y_j) \right)_{i\leq N, \,   j\leq m} .
\end{equation}

For any $(x,h) \in \R\times[0,\infty)$ and $n\geq 1$, let  $\theta_{x,h}^n = \frac{\tau_{x,h}^n}{n g(n)}$.
Then it follows from Theorem \ref{thm:main_XL} and Lemma \ref{lem:txhn-conv}  that along some subsequence $n_k\to \infty$ we have the joint convergence
\[
 (\mathcal{L}_{n_k}, \theta_{x_1,h_1}^{n_k}, \theta_{x_2,h_2}^{n_k}, \ldots, \theta_{x_N,h_N}^{n_k})
 \underset{k\to\infty}{\Longrightarrow} 
 (\mathfrak{L}, \theta_{x_1,h_1}, \theta_{x_2,h_2}, \ldots, \theta_{x_N,h_N}).
\]
Since $t\mapsto \mathfrak{L}(t,\cdot)$
is almost surely continuous, 
then it follows from the continuous mapping theorem that
\[
 \left( \mathcal{L}_{n_k}(\theta_{x_i,h_i}^{n_k},\cdot) \right)_{i\leq N} \underset{k\to\infty}{\Longrightarrow} \left( \mathfrak{L}(\theta_{x_i,h_i},\cdot) \right)_{i\leq N}. 
\]
On the other hand, since
$\mathcal{L}_{n}(\theta_{x_i,h_i}^{n}, \cdot) = \Lambda_{x_i,h_i}^n(\cdot)$ then we must have that 
\begin{equation}\label{JRK-theta}
 \left( \mathfrak{L}(\theta_{x_i,h_i},\cdot) \right)_{i\leq N} \overset{\text{Law}}{=} 
 \left( \Lambda_{x_i,h_i}(\cdot) \right)_{i\leq N}. 
\end{equation}
An immediate consequence of this is that 
$\mathfrak{L}(\theta_{x_i,h_i},x_i) = h_i$ for $i\leq N$, almost surely, and then it follows from the definition of the inverse local times $\mathfrak{t}_{x,h}$ (see the statement of the theorem) that $\theta_{x_i,h_i} \leq \mathfrak{t}_{x_i,h_i}$, almost surely.

We can repeat the above argument with $h_i$ replaced by $h_i+\delta$
for fixed $\d>0$.  Again, taking a (possibly further) subsequence to
get joint convergence,
\[
 (\mathcal{L}_{n_k}, \theta_{x_1,h_1+\d}^{n_k}, \theta_{x_2,h_2+\d}^{n_k}, \ldots, \theta_{x_N,h_N+\d}^{n_k})
 \underset{k\to\infty}{\Longrightarrow} 
 (\mathfrak{L}, \theta_{x_1,h_1+\d}, \theta_{x_2,h_2+\d}, \ldots, \theta_{x_N,h_N+\d}),
\]
we can conclude that the joint distribution of the right hand side is such that
$\mathfrak{L}(\theta_{x_i,h_i+\d},x_i) = h_i + \d$, and thus 
\begin{equation}\label{eq_txh_finite}
 \forall\, i \in \{1,\ldots,N\}, \quad \mathfrak{t}_{x_i,h_i} \leq \theta_{x_i,h_i+\d} < \infty \text{ a.s.}
\end{equation}
(the finiteness of the $\theta_{x_i,h_i+\d}$ comes from Assumption \ref{asm:Bessel}).

The above bounds on $\mathfrak{t}_{x_i,h_i}$ imply that
$ \mathfrak{L}(\theta_{x_i,h_i},\cdot) \leq \mathfrak{L}(\mathfrak{t}_{x_i,h_i},\cdot) \leq \mathfrak{L}(\theta_{x_i,h_i+\d},\cdot)$, a.s.
Therefore, for any $\e>0$ we have
\begin{align*}
 P\left( \max_{i\leq N, \, j\leq m} \left| \mathfrak{L}(\mathfrak{t}_{x_i,h_i}, y_j ) - \mathfrak{L}(\theta_{x_i,h_i}, y_j ) \right| > \e \right)
 &\leq  P\left( \max_{i\leq N, \, j\leq m} ( \mathfrak{L}(\theta_{x_i,h_i+\d}, y_j) - \mathfrak{L}(\theta_{x_i,h_i}, y_j) )  > \e \right) \\
 &=   P\left( \max_{i\leq N, \, j\leq m} ( \Lambda_{x_i,h_i+\d} (y_j) - \Lambda_{x_i,h_i} (y_j) ) > \e \right) \\
 &\leq \sum_{i\leq N, \, j\leq m} \frac{E[ \Lambda_{x_i,h_i+\d}(y_j)] - E[ \Lambda_{x_i,h_i}(y_j)]}{\e}. 
\end{align*}
Since it follows from Lemma \ref{lem:curvemean}
that $E[\Lambda_{x_i,h_i+\d}(y_j)] \to E[\Lambda_{x_i,h_i}(y_j)]$ as $\d\to 0$
then the above upper bound vanishes as $\d \to 0$ for $\e>0$ fixed.
Together with \eqref{JRK-theta} this completes the proof of \eqref{JRK-fdd}.
 
It remains to prove that $\mathfrak{X}$ is the unique process with these joint RKTs. 
Suppose now that $\tilde{\mathfrak{X}}$ is another continuous process with a local time $\tilde{\mathfrak{L}}$ satisfying the conditions in the statement of Theorem \ref{thm:main_JRK}.
That is, if $\tilde{\mathfrak{t}}_{x,h} = \inf\{t\geq 0: \tilde{\mathfrak{L}}(t,x) > h \}$ are the corresponding inverse local times, then we assume that for any $(x_1,h_1),\ldots,(x_N,h_N) \in \R \times [0,\infty)$,
\begin{equation}\label{tildeJRK}
 \big( \tilde{\mathfrak{L}}(\tilde{\mathfrak{t}}_{x_i,h_i}, y ) \big)_{y \in \R, 1\leq i \leq N}
 \overset{\text{Law}}{=} 
 \big(\Lambda_{x_i,h_i}(y) \big)_{y\in \R, 1\leq i \leq N}. 
\end{equation}
We need only to show that $\tilde{\mathfrak{X}}$ and $\mathfrak{X}$ have the same finite dimensional distributions.
Our proof will repeat many of the arguments that were used in proofs of tightness and convergence of finite dimensional distributions for $\mathcal{X}_n(\cdot)$. To this end, we note that the following facts can be easily proved from the joint Ray-Knight theorems and the properties of the family of limiting curves $\Lambda_{x_i,h_i}(\cdot)$.
\begin{enumerate}[--]
 \item For any $\d>0$, for any finite collection of points $(x_i,h_i)$, $i\leq N$, equation \eqref{tildeJRK} and the fact that $\tilde{\mathfrak{t}}_{x,h} = \int_\R \tilde{\mathfrak{L}}(\tilde{\mathfrak{t}}_{x,h}, y) \, dy$ a.s. yield $( \tilde{\mathfrak{t}}_{x_i,h_i} )_{i \leq N} \overset{\text{Law}}{=} ( \int_\R \Lambda_{x_i,h_i}(y) \, dy )_{i \leq N}$. Extending this to the countable set of points in $Q_\d= (\d \mathbb{Z}) \times (\d^\gamma \mathbb{Z}_+)$ by Kolmogorov's extension theorem, we obtain
 \begin{equation}\label{tildetxh-law}
 \left( \tilde{\mathfrak{t}}_{x,h} \right)_{(x,h) \in Q_\d} \overset{\text{Law}}{=} \left( \int_\R \Lambda_{x,h}(y) \, dy \right)_{(x,h) \in Q_\d}.
\end{equation}
\item For any $K\geq 1$, $T<\infty$ let 
\[ 
 \tilde{R}_{K,T} = \left\{ 
 \sup_{t\leq T} |\tilde{\mathfrak{X}} (t)| \leq K  \right\} 
 \cap \left\{ \sup_{x \in \R} \tilde{\mathfrak{L}}(T,x) \leq K^{1+\gamma}
 \right\}.  
\]
Then, for $\e>0$ and $T<\infty$ fixed there exists $K\geq 1$ large enough so that $P( \tilde{R}_{K,T} ) \geq 1- \e$. 
Indeed, $\tilde{\mathfrak{X}}$ is continuous hence bounded on $[0,T]$, and by the Ray-Knight theorems, Assumption \ref{asm:Bessel} and the scaling property \eqref{Lxh-rescale} we get
\begin{align*}
 P\left( \sup_{x\in \R} \tilde{\mathfrak{L}}(T,x) > K^{1+\gamma} \right) 
 &\leq P\left( \tilde{\mathfrak{t}}_{0,K} < T \right) + P\left( \sup_{x\in \R} \tilde{\mathfrak{L}}(\tilde{\mathfrak{t}}_{0,K}, x) > K^{1+\gamma} \right) \\
 &= P\left( \int_\R \Lambda_{0,K}(x) \, dx < T \right) + P\left( \sup_{x\in \R} \Lambda_{0,K}(x) > K^{1+\gamma} \right) \\
 &= P\left( \int_\R \Lambda_{0,1}(x) \, dx < \frac{T}{K^{1+1/\gamma}} \right) + P\left( \sup_{x\in \R} \Lambda_{0,1}(x) > K^{\gamma} \right).
\end{align*}

 \item Let $\tilde{\mathfrak{T}}_\d = \{ \tilde{\mathfrak{t}}_{x,h}: (x,h) \in Q_\delta \}$, and for any $\e,\d>0$, $(x,h) \in Q_\d$ let
 \[
  \tilde{A}_{x,h,\e,\d} = 
  \left\{ 
  \exists\, t > \tilde{\mathfrak{t}}_{x,h}: |\tilde{\mathfrak{X}}(t) - x |
  > \e, \, \text{ and } \tilde{\mathfrak{T}}_\d \cap (\tilde{\mathfrak{t}}_{x,h}, t ) = \emptyset 
 \right\}. 
 \]
 Then, 
 \begin{equation}\label{PtildeA-ub}
  P\left( \tilde{A}_{x,h,\e,\d} \right) 
  \leq 2 (1-c_0)^{\fl{\frac{\e}{\d}}}, \quad \forall\, (x,h) \in Q_\d,
 \end{equation}
where $c_0$ is the constant from Assumption \ref{asm:ND}. The proof of \eqref{PtildeA-ub} is similar to the proof of \eqref{PAxh} but is simpler since we do not need to pass to a limit and therefore can avoid the need to work only with finitely many grid points. 
\end{enumerate}

Now, for any $\d>0$ we can construct a discrete approximation process $\tilde{\mathfrak{X}}_\d$ similar to what was done in the proof of Theorem \ref{thm:main_X}. 
That is, 
\[
 \tilde{\mathfrak{X}}_\d(t) = x \in \d \Z 
 \quad \iff \quad 
 \exists\, h \in \d^\gamma \Z_+ \text{ with } \tilde{\mathfrak{t}}_{x,h} \leq t \text{ and } \tilde{\mathfrak{T}}_\d \cap \left( \tilde{\mathfrak{t}}_{x,h}, t \right) = \emptyset.
\]
To prove that $\tilde{\mathfrak{X}}_\d$ is well-defined, we need the $\tilde{\mathfrak{t}}_{x,h}$, $(x,h)\in Q_\d$ to be all different, which holds almost surely thanks to Lemma \ref{lem_int_diff}, and that there is only a finite number of $\tilde{\mathfrak{t}}_{x,h}$, $(x,h)\in Q_\d$ in any bounded interval. To see the latter, we note that for any $T>0$, if the event $\tilde{R}_{K,T}$ occurs, then a.s. $\tilde{\mathfrak{t}}_{x,h} > T$ for all $(x,h)\in Q_\d$ so that $(x,h) \notin [-K,K] \times [0,K^{1+\gamma}]$, so there is only a finite number of $\tilde{\mathfrak{t}}_{x,h}$, $(x,h)\in Q_\d$ below $T$. We can now choose a sequence $T_n \underset{n \to \infty}{\longrightarrow} \infty$ and a sequence $K_n$ so that $P(\tilde R_{K_n,T_n}) \geq 1-1/2^n$, so almost surely an infinite number of $\tilde R_{K_n,T_n}$ occur, hence almost surely there is only a finite number of $\tilde{\mathfrak{t}}_{x,h}$, $(x,h)\in Q_\d$ in any bounded interval.

We have the following analog of Lemma \ref{lem_Prohorov} for $\tilde{\mathfrak{X}}_\d$.
\begin{lemma}\label{lem_Prohorov2}
 For any $j\in\mathbb{N}$, $0 < t_1 < \cdots < t_j$ and $\varepsilon>0$, when $\d>0$ is small enough,
 \begin{equation}\label{eq_Prohorov2}
   P\left( \max_{i\leq j} \left| \tilde{\mathfrak{X}}(t_i) - \tilde{\mathfrak{X}}_\d(t_i) \right| > \e \right) 
   < \e. 
 \end{equation}
\end{lemma}
\begin{proof}
Note that if the event $\tilde{R}_{K,t_j}$ occurs, then $\tilde{\mathfrak{t}}_{x,h} > t_j$ for all $(x,h)\in Q_\d \backslash Q_\d^{(K)}$ (we recall $Q_\d^{(K)}$ was defined in \eqref{TdnK-def}). Then, since no $t_i$ is an atom of any $\tilde{\mathfrak{t}}_{x,h}$ (Proposition \ref{prop_int_no_atoms}) it follows easily from the definitions and \eqref{PtildeA-ub} that for $K\geq 1$ and $\d \in (0,1)$ we have
\begin{equation}\label{XXd-diff}
 P\left( \max_{i\leq j} \left| \tilde{\mathfrak{X}}(t_i) - \tilde{\mathfrak{X}}_\d(t_i) \right| > \e, \, \tilde R_{K,t_j} \right)
 \leq P\left( \bigcup_{(x,h) \in Q_\d^{(K)} } \tilde{A}_{x,h,\e,\d} \right) 
 \leq \frac{12 K^{2+\gamma}}{\d^{1+\gamma}} (1-c_0)^{\fl{\frac{\e}{\d}}}. 
\end{equation}
The proof is then completed first by choosing $K\geq 1$ large enough so that $P(\tilde R_{K,t_j} ) > 1 -  \frac{\e}{2}$ and then  choosing $\d \in (0,1)$ sufficiently small so that the right side of \eqref{XXd-diff} is less than $\frac{\e}{2}$.
\end{proof}

Finally, we are ready to prove that $\tilde{\mathfrak{X}}$ and $\mathfrak{X}$ must have the same finite dimensional distributions. 
Fix $j\in\mathbb{N}$, $0<t_1<t_2<\cdots<t_j$ and $\e>0$.
For any $\d>0$ we can construct a discrete approximation $\mathfrak{X}_\d(t)$ for $\mathfrak{X}(t)$ using the inverse local times $(\mathfrak{t}_{x,h})_{(x,h) \in Q_\d}$  in the same way as $\tilde{\mathfrak{X}}_\d(t)$ was constructed from $(\tilde{\mathfrak{t}}_{x,h})_{(x,h) \in Q_\d}$, and note that it follows from \eqref{tildetxh-law} that $\tilde{\mathfrak{X}}_\d(\cdot)$ and $\mathfrak{X}_\d(\cdot)$ have the same law.
Also, since the proof of Lemma \ref{lem_Prohorov2} only used the joint RKTs, the identical statement holds for $\mathfrak{X}$. We now choose $\d>0$ so that \eqref{eq_Prohorov2} holds for both $\mathfrak{X}$ and $\tilde{\mathfrak{X}}$.
Therefore, if we again let $d_j(\cdot)$ denote the Prokhorov metric on the space of probability measures on $\R^j$, then 
\begin{align*}
d_{j}\left( (\mathfrak{X}(t_i))_{1 \leq i \leq j},(\tilde{\mathfrak{X}}(t_i))_{1 \leq i \leq j} \right) 
&\leq d_{j}\left((\mathfrak{X}(t_i))_{1 \leq i \leq j},(\mathfrak{X}_\d(t_i))_{1 \leq i \leq j} \right) 
+ d_{j}\left( (\tilde{\mathfrak{X}}(t_i))_{1 \leq i \leq j},(\tilde{\mathfrak{X}}_\d(t_i))_{1 \leq i \leq j} \right) \\
&\leq 2\e.
\end{align*}
Since $\e>0$ was arbitrary, we conclude that $ (\mathfrak{X}(t_i))_{1 \leq i \leq j}$ and $(\tilde{\mathfrak{X}}(t_i))_{1 \leq i \leq j} $ have the same distribution.
\end{proof}

\begin{proof}[Proof of Corollary \ref{xh}]
 Note that $(\mathcal{X}_n, \mathcal{H}_n) = \phi(\mathcal{X}_n, \mathcal{L}_n)$ and $(\mathfrak{X},\mathfrak{H}) = \phi(\mathfrak{X},\mathfrak{L})$, where $\phi$ is given by $\phi(x(\cdot),\ell(\cdot,\cdot)) = (x(\cdot), \ell(\cdot, x(\cdot)))$. By Theorem \ref{thm:main_XL} and the Skorokhod Representation Theorem, it is enough to prove that if a sequence $(x_n,\ell_n)\in D([0,\infty)) \times D([0,\infty); D(\R) )$ converges to $(x,\ell)\in C([0,\infty)) \times C([0,\infty);C_0(\R))$ with $\ell(\cdot,x)$ nonnegative and nondecreasing for all $x\in\R$, then $\phi(x_n,\ell_n)$ converges to $\phi(x,\ell)$ in $D([0,\infty))^2$. This can be proven using the fact that the $\ell(t,\cdot)$ are uniformly equicontinuous for $t \in [0,T]$, for which we use the fact that their supports are contained in that of $\ell(T,\cdot)$, which is bounded.
\end{proof}

\section{Polynomially self-repelling walks}\label{PSRcase}

In this section we shall prove Theorem~\ref{main_psr} by checking the
conditions of Theorems \ref{thm:main_X}--\ref{thm:main_JRK}. The key
result that we need is the augmented joint RKT for PSR walks, and we also need to verify
Assumptions~\ref{asm:Bessel}--\ref{asm:ND} for the family of curves in the joint RKT.

Recall that a polynomially self-repelling (PSR) walk is defined by
\eqref{sirw} with the weight function given by \eqref{w} for some
$\alpha>0$. Equation \eqref{lanxh} with $g(n)=n$ becomes
\begin{equation}\label{Ln-PSR}
 \Lambda_{x,h}^n(y) = n^{-1}L(\tau_{\fl{xn},\fl{hn}}, \fl{yn}).
\end{equation}
It follows from \cite[Theorem 1B]{tGRK} that
$ \Lambda_{x,h}^n\Longrightarrow \Lambda_{x,h}$  as
$n\to\infty$ where $(2\alpha+1)\Lambda_{x,h}\in\rm{BESQ}^1$ with
reflection and absorption properties exactly as specified in
Assumption~\ref{asm:Bessel}\footnote{The use of site local times in place of directed edge local times as in \cite{tGRK} is justified below in Section \ref{sec:PSRaug}.}.  Hence, Assumption~\ref{asm:Bessel} is
satisfied with $g(n)=n$, $\gamma=1$,
$\delta_{a,\pm}=\delta_{r,\pm}=1$, and
$C_{a,\pm}=C_{r,\pm}=(2\alpha+1)^{-1}$.

  In the next two subsections we shall check Assumptions
  \ref{asm:JMarkov} and \ref{asm:ND} and prove the
  augmented joint RKT. 
  These tasks are facilitated by results in
  \cite{Y25+}, where the author proved the joint RKT (for edge local times, see Section \ref{sec:PSRaug}) and also
  independently constructed a candidate limiting process for the rescaled
  PSR walks. We shall use the joint RKT from \cite[Theorem~1.2]{Y25+} and
  strengthen it to the desired augmented version. This will allow
  us to apply Theorems \ref{thm:main_X}--\ref{thm:main_JRK} and
  complete the proof of Theorem \ref{main_psr}.

  The work \cite{Y26flows} constructs the entire family
  of curves $(\Lambda_{x,h}(\cdot))_{x\in\R, h\geq 0}$ driven by a
  single white noise under the name ``two-sided PSR flow'', and the family is given another description in \cite[Definition~1.1]{Y25+}. This family can be described as follows.  Let $\mathcal{W}(\cdot,\cdot)$ be a
  space-time white noise on the upper half-plane
  $\R \times [0,\infty)$.
  \begin{enumerate}[--]
     \item The forward curve $(\Lambda_{x,h}(x+t))_{t\geq 0}$ is the unique strong solution to
   \begin{equation}\label{WNrep}
   \Lambda_{x,h}(x+t) = h + \int_x^{x+t} \frac{2}{(\Lambda_{x,h}(u))^\alpha} \int_0^\infty s^\alpha \ind{s \leq \Lambda_{x,h}(u)} \mathcal{W}(du, ds) + \frac{t}{2\alpha+1}, \quad  t\geq 0,
\end{equation}
which is reflected at 0 for $t\in [0,\max\{0,-x\}]$ and absorbed at 0 for $t > \max\{0,-x\}$.
   \item The backward curve $(\Lambda_{x,h}(x-t))_{t\geq 0}$ is the unique strong solution to
   \begin{equation}\label{WNrep-back}
   \Lambda_{x,h}(x-t) = h - \int_{-x}^{t-x} \frac{2}{(\Lambda_{x,h}(-u))^\alpha} \int_0^\infty s^\alpha \ind{s \leq \Lambda_{x,h}(-u)} \mathcal{W}^*(du, ds) + \frac{t}{2\alpha+1}, \quad  t\geq 0,
   \end{equation}
   which is reflected at 0 for $t \in [0,\max\{0,x\}]$ and absorbed at 0 for $t > \max\{0,x\}$, where $\mathcal{W}^*$ is the image of $\mathcal{W}$ under the map $(x,h) \mapsto (-x,h)$.
 \end{enumerate}

\begin{remark}
  We note that the curves defined in \eqref{WNrep} and
  \eqref{WNrep-back} coincide with the continuum Ray-Knight curves of
  \cite[Definition~1.1]{Y25+}. The Ray-Knight theorems in \cite{Y25+}
  (and also in \cite{tGRK}) are proved for directed edge local times,
  whereas in this paper we work with local times at sites, which are
  roughly double the directed edge local times; this is reflected in
  the factor of $2$ and the halved initial height both in
  \cite[Theorem~1.2]{Y25+} and in \eqref{barLxh} below.
  The curves also differ from those in \cite{tGRK} by an additional
  factor of $(2\alpha+1)^{-1}$. This reflects our choice to use
  the scaling function $g(n)=n$ rather than $g(n) = n/(2\alpha+1)$ as
  in \cite{tGRK}.
\end{remark}

\subsection{Verifying Assumptions \ref{asm:JMarkov} and \ref{asm:ND}}\label{sec:PSRasm}

To check Assumption \ref{asm:JMarkov} we use representation
\eqref{WNrep}. By pathwise uniqueness for \eqref{WNrep} (see
\cite[pp.~3-4]{Y25+}), for every $(x,h)\in\R\times[0,\infty)$, the curve
$\Lambda_{x,h}(x+\cdot)$ coincides a.s.\ with the flow line started at
$(h,x)$ of the PSR flow driven by $\mathcal{W}$ (with $(a,r)$ of
\cite{Y25+} corresponding to our $(h,x)$). The flow lines are
right-continuous in $h$ and satisfy the flow property
\cite[Definition~3.1(R2) and (3.4)]{Y25+}. Each $\Lambda_{x,h}(x+\cdot)$,
being a strong solution, is a.s.\ a measurable functional of
$\mathcal{W}|_{(x,\infty)}$, and, for every $h\ge 0$, by
right-continuity,
$\Lambda_{x,h}(x+\cdot)=\lim_{q\downarrow h}\Lambda_{x,q}(x+\cdot)$
a.s.\ along dyadic $q$. Hence, for every $x\in\R$, there is a map
$\Phi (x,h, \mathcal{W}|_{(x,\infty)})$, jointly measurable in
$(h,\mathcal{W}|_{(x,\infty)})$, such that for every $h\in[0,\infty)$
\begin{equation}\label{def_Phi}
\big(\Lambda_{x,h}(x+t) \big)_{t \geq 0} = \Phi \Big(x,h, \mathcal{W}|_{(x,\infty)} \Big) \qquad \text{a.s.}
 \end{equation}
 It follows from the flow property and the right-continuity in $h$
 that for $x_0 \geq x$ necessarily
\[
  \big( \Lambda_{x,h}(x_0+t) \big)_{t \geq 0} = \Phi \Big( x_0, \Lambda_{x,h}(x_0),\, \mathcal{W}|_{(x_0,\infty)} \Big) \qquad \text{a.s.}
 \]

 Let
\begin{align}
  \overline{\mathcal{G}}_u &:= \text{the completion of }\,\sigma\Big( \mathcal{W}(A):\ A \subset (-\infty,u]\times[0,\infty)\ \text{Borel},\ |A|<\infty \Big),\label{def_G}\\
  \shortintertext{and}
  \mathcal{G} ^ \prime _u &:= \sigma\Big( \mathcal{W}(A):\ A \subset (u, \infty)\times[0,\infty)\ \text{Borel},\ |A|<\infty \Big),\nonumber
\end{align}
where $|A|$ denotes the Lebesgue measure of $A$. It follows from
the definition of white noise that for every $u$ the sigma-algebras $\overline{\mathcal{G}}_u $ and
$ \mathcal{G} ^ \prime _u $ are independent.
We have thus that for $(x_i,h_i)_{i \leq N} $ and $I = \{i_1,\ldots,i_k\}  \subset \{1,2, \cdots N\} $ as in Assumption \ref{asm:JMarkov}, that
\[
 \sigma\Big( \big(\Lambda_{x_i,h_i}(x_i+t)\big)_{0\leq t\leq x_N-x_i},\ i\leq N \Big) \;\subset\; \overline{\mathcal{G}}_{x_N},
 \] and, with $H_{i_\ell} := \Lambda_{x_{i_\ell},h_{i_\ell}}(x_N)$, $1\le \ell \le k$,
 \[
  \mathcal{Z}_I(\cdot) = \Big( \Phi\big({x_N}, H_{i_1},  \mathcal{W}|_{(x_N,\infty)} \big), \ldots, \Phi\big(x_N,H_{i_k}, \mathcal{W}|_{(x_N,\infty)} \big) \Big) \qquad \text{a.s.},
  \]
where $(H_{i_\ell})_{\ell \le k}$ is $\overline{\mathcal{G}}_{x_N}$-measurable and
$\mathcal{W}|_{(x_N,\infty)}$ is independent of $\overline{\mathcal{G}}_{x_N}$.
Since $\Phi(x_N,\cdot,\cdot)$ is jointly measurable, Fubini's theorem gives
$P\big(\mathcal{Z}_I \in \cdot \mid \overline{\mathcal{G}}_{x_N}\big)
 = K_{x_N}\big((H_{i_\ell})_{\ell\le k},\cdot\big)$ a.s., where
$K_{x_0}\big((a_\ell)_{\ell\le k},\cdot\big)$ denotes the law of
$\big(\Phi(x_0,a_\ell,\mathcal{W}|_{(x_0,\infty)})\big)_{\ell \le k}$.
The second claim of Assumption~\ref{asm:JMarkov} follows by the tower
property. The same argument with $x_N$ replaced by $x_N+s$, $s \ge 0$,
shows that $\mathcal{Z}_I$ is a Markov process with transition kernels
$K_{x_N+s}$ (inhomogeneous only through the dependence of the absorption
rule on $\max\{0,-(x_N+s)\}$).
By \eqref{WNrep-back}, the backward curves are forward curves driven by
$-\mathcal{W}^*$, so the paths to the left have the analogous Markovian
structure required in Assumption~\ref{asm:JMarkov}.

\medskip

For checking Assumption \ref{asm:ND} (and since it may be of independent interest) we give a more explicit SDE representation of the joint distribution of the forward curves. By the Markov property, it is enough to give the joint distribution of forward curves starting at the same spatial point. Fix $x\in\R$ and $0\leq h_1 < h_2 < \ldots < h_N$, and let
$S_i(t) = \Lambda_{x,h_i}(x+t)$.
Then $(S_1(t),S_2(t),\ldots,S_N(t))_{t\geq 0}$ solves the system of SDEs
\begin{equation}\label{jointSDE}
  \begin{cases}
    dS_1(t)=  \frac{1}{2\alpha+1} \, dt
 + \frac{2}{\sqrt{2\alpha+1}}
 \sqrt{S_1(t)} \, dW_1(t)& \\[1mm]
 dS_j(t)
 = \frac{1}{2\alpha+1} \, dt
    + \frac{2}{\sqrt{2\alpha+1}}
 (S_j(t))^{-\alpha}\sum\limits_{i=1}^j \sqrt{S_i(t)^{2\alpha+1} - S_{i-1}(t)^{2\alpha+1} } \, dW_i(t), &
 2\leq j\leq N,
  \end{cases}
\end{equation}
where $W_1, W_2, \ldots, W_N$ are independent standard Brownian motions,
$S_0\equiv 0$, $S_i(0)=h_i$ for $i \in \{1,\ldots,N\}$, and the coordinates merge when
they meet, reflect at zero for $t \leq \max\{0,-x\}$ and
get absorbed at zero for $t>\max\{0,-x\}$.

It remains to show that Assumption \ref{asm:ND} is satisfied for the family of curves in \eqref{WNrep} and \eqref{WNrep-back}.
Due to symmetry between the forward and backward curves, it is enough to  verify only the condition on the right merging locations in Assumption \ref{asm:ND}. That is, we need to check that there is a $c_0>0$ such that
\[
  \sup_{k \in \Z, \, \ell \geq 0} P\left( \mu_{(k,\ell),(k,\ell+1)}^+
    \geq k + \frac{1}{2} \right)= \sup_{k \in \{-1,0\}, \, \ell
      \geq 0} P\left( \mu_{(k,\ell),(k,\ell+1)}^+ \geq k + \frac{1}{2}
    \right) \leq 1-c_0.
\]
Note that reflecting the lower curve at zero only decreases the
merge time, thus we need only to consider $k=0$.  Moreover,
since all of the probabilities are less than 1 for any $\ell \geq 0$ (indeed, for all $\ell \geq 0$, there is a positive probability that $\Lambda_{0,\ell+1}$ is absorbed at 0 before time $\frac{1}{3}$, as can be seen via a comparison with a Brownian motion),
it is sufficient to argue that
\[
 \limsup_{\ell\to \infty} P\left(  \mu_{(0,\ell),(0,\ell+1)}^+ \geq  \frac{1}{2} \right) < 1.
\]
The inclusion
$\{\mu_{(0,\ell),(0,\ell+1)}^+ \ge \frac{1}{2}\}\subset
\{\Lambda_{0,\ell}\left( 1/3 \right) < \Lambda_{0,\ell+1}\left( 1/3
\right)\}$ and the following lemma complete the verification of
Assumption \ref{asm:ND}.

\begin{lemma}\label{c-base} For every $T>0$,
\[
 \lim_{\ell\to\infty}  P\left( \Lambda_{0,\ell}(T) < \Lambda_{0,\ell+1}(T) \right)
 = P( Z(T) > 0 \mid Z(0) = 1)= 1-e^{-1/(2T)},
\]
where $Z\in\rm{BESQ}^0$ (see \eqref{besq}) which is absorbed upon
hitting $0$.
\end{lemma}

\begin{proof}
Let
 \[
 S_\ell(t) = \Lambda_{0,\ell}(t),\ \
 \quad Y_\ell(t) = \Lambda_{0,\ell+1}(t) - \Lambda_{0,\ell}(t), \quad \text{and}\ \  R_\ell(t)=\frac{Y_\ell(t)}{S_\ell(t)}.
 \]
With this notation, we need to show that 
\begin{equation}\label{YlTzero}
 \lim_{\ell\to\infty} P\left( Y_\ell(T) = 0 \right) = P\left(Z(T) = 0 \mid Z(0)=1 \right) = e^{-1/(2T)}, \qquad \forall\, T>0.
\end{equation}
It is known that the Laplace transform of $Z \in \text{BESQ}^0$ is  $E[e^{-\lambda Z(T)} \mid Z(0)=1 ] = \exp\{ - \frac{\lambda}{1+2\lambda T} \}$ for $\lambda\geq 0$ (see \cite[Ch.\ XI, \S 1]{ryCMBM}); taking $\lambda\to \infty$ gives the second equality in \eqref{YlTzero}.
 
Hence, it suffices to prove that the limit on the left in
\eqref{YlTzero} also equals $e^{-1/(2T)}$. Denote
$\sigma_0 (S_\ell)= \inf\{t\geq 0: S_\ell(t) = 0\}$. It follows from
\eqref{jointSDE} applied to the pair
$(\Lambda_{0,\ell},\Lambda_{0,\ell+1})$, that there is a standard
Brownian motion $B$ such that on the interval $[0,\sigma_0 (S_\ell))$,
\begin{equation}
  dY_\ell(t)= \frac{2}{\sqrt{2\alpha+1}}
    \sqrt{Y_\ell(t) + 2S_\ell(t) \left( 1 - \left( \tfrac{S_\ell(t)}{S_\ell(t)+Y_\ell(t)} \right)^\alpha \right) } \, dB(t)=2
    \sqrt{Y_\ell(t)}
    \sqrt{f_\alpha(R_\ell(t))} \, dB(t), \label{eq:SDE}
\end{equation}
where the function $f_\alpha\in C^1([0,\infty))$, $\alpha>0$, is given by
\[
f_\alpha(z)
=
\begin{cases}
  \frac{1}{2\alpha+1}\left( 1 + \frac{2}{z} - \frac{2}{z(1+z)^\alpha}  \right),&\text{if }z>0;\\1,&\text{if }z=0.
\end{cases}
\]
We will use \eqref{eq:SDE} to show that $Y_\ell$ is a time change of a BESQ$^0$
process.

\begin{lemma}\label{lem:BESQ-tc}
 With notation as above, 
 let 
 \[
  \Gamma_\ell(t) = \int_0^t f_\alpha( R_\ell(s)) \, ds, \quad \text{for } 0 \leq t < \sigma_0 (S_\ell). 
 \]
 Then, on a possibly larger probability space there is a $Z_\ell \in \emph{BESQ}^0$ with $Z_\ell(0)=1$ such that, almost surely, 
 \[
  Y_\ell(t) = Z_\ell(\Gamma_\ell(t)), \qquad \forall\, t < \sigma_0 (S_\ell).
 \]
\end{lemma}

Postponing the proof of Lemma \ref{lem:BESQ-tc} for now, we show how to finish the proof of Lemma \ref{c-base}. 
First, note that since $f_\alpha$ is strictly decreasing\footnote{Indeed,
  $f'_\alpha(z) = \frac{2}{2\alpha+1}\left( \frac{1+(\alpha+1)z
      -(1+z)^{\alpha+1} }{z^2(1+z)^{\alpha+1} }\right) < 0 $, since
  $1+(\alpha+1)z < (1+z)^{\alpha+1}$ for all $z>0$.} and $\lim_{z\to\infty} f_\alpha(z) = \frac{1}{2\alpha+1}$,
\begin{equation}\label{eq:fbounds}
\tfrac{1}{2\alpha+1}\ \le\ f_\alpha(z)\ \le\ 1,\qquad z\ge 0,
\end{equation}
hence
$\frac{t}{2\alpha+1} \leq \Gamma_\ell(t) \leq t$ for all $t<\sigma_0 (S_\ell)$, and
\begin{align*}
 P\left( Y_\ell(T) = 0,\, \sigma_0 (S_\ell) > T \right) \leq P\left( Z_\ell(\Gamma_\ell(T)) = 0, \sigma_0 (S_\ell) > T \right) \leq P\left( Z_\ell(T) = 0 \right) = e^{-1/(2T)},
\end{align*}
where in the second inequality we used that $Z_\ell$ is absorbed
at zero. 
Since $S_\ell$, before hitting $0$, is a multiple of a BESQ$^1$ process started at $S_\ell(0) = \ell$, then $\lim_{\ell\to \infty} P(\sigma_0 (S_\ell) \leq T) = 0$ and we can conclude that
$\limsup_{\ell\to\infty} P( Y_\ell(T) = 0 ) \leq e^{-1/(2T)}$. 

To obtain a corresponding lower bound, first let
$G_{\ell,T} = \{ \sup_{0 \leq t \leq T} |S_\ell(t) - \ell| \leq
\ell^{3/4} \}$ and note that there exist constants $C,c>0$ independent
of $\ell$ such that $P(G_{\ell,T}^c) \leq C e^{-c \sqrt{\ell}/T}$
and also that $G_{\ell,T} \subset \{\sigma_0 (S_\ell) > T \}$.
We claim that on the event $G_{\ell,T}$ the time change $\Gamma_\ell$ is nearly linear in the sense that 
\begin{equation}\label{tc-linear}
 \lim_{\ell\to\infty} P\left( G_{\ell,T}, \, \Gamma_\ell(T) < T-\e \right) = 0, \quad \forall\, \e,T>0.
\end{equation}
Indeed, by the definition of $f_\alpha$ there exists a $\d>0$ such
that $f_\alpha(z) \geq 1-\frac{\e}{T}$ for all $z\in[0,\d]$.  Therefore,
\begin{align*}
 &P\left( G_{\ell,T}, \, \Gamma_\ell(T) < T-\e \right) 
 \leq P\left( G_{\ell,T}, \, \sup_{0\leq t\leq T} R_\ell(t) > \d \right) \\
 &\leq P\left( T < \sigma_0 (S_\ell), \, \sup_{0\leq t\leq T} Y_\ell(t) > \d ( \ell-\ell^{3/4} )  \right)
 \leq P\left( T < \sigma_0 (S_\ell), \, \sup_{0\leq t\leq \Gamma_\ell(T)} Z_\ell(t) > \d ( \ell-\ell^{3/4} )  \right)\\
 &\leq P\left( \sup_{0\leq t\leq T} Z_\ell(t) > \d (\ell - \ell^{3/4}) \right), 
\end{align*}
and since the distribution of $\sup_{0 \leq t\leq T} Z_\ell(t)$ does not depend on $\ell$ (recall that $Z_\ell \in \text{BESQ}^0$ with $Z_\ell(0)=1$), 
this upper bound vanishes as $\ell \to\infty$. This proves \eqref{tc-linear}. 
Next, note that 
\begin{align*}
 P\left( Y_\ell(T) = 0 \right) 
 &\geq P\left( Z_\ell(\Gamma_\ell(T)) = 0, \, T < \sigma_0 (S_\ell) \right) \geq P\left( Z_\ell(T-\e) = 0,\, T < \sigma_0 (S_\ell), \, \Gamma_\ell(T) \geq T-\e \right) \\
 &\geq P\left( Z_\ell(T-\e) = 0,\, G_{\ell,T} , \, \Gamma_\ell(T) \geq T-\e \right) \\
 &\geq P\left( Z_\ell(T-\e) = 0 \right) - P\left( G_{\ell,T}^c \right) - P\left( G_{\ell,T}, \, \Gamma_\ell(T) < T-\e \right). 
\end{align*}
Since $P\left( Z_\ell(T-\e) = 0 \right) = e^{-1/(2(T-\e))}$ and the last two probabilities in the last line vanish as $\ell \to \infty$, it follows that 
$\liminf_{\ell\to\infty} P\left( Y_\ell(T) = 0 \right) 
\geq e^{-1/(2(T-\e))}$. Taking $\e\to 0$ completes the proof of \eqref{YlTzero}, pending the proof of Lemma \ref{lem:BESQ-tc}. 

\begin{proof}[Proof of Lemma \ref{lem:BESQ-tc}]
 Letting $A_\ell = \lim_{t\nearrow \sigma_0 (S_\ell)} \Gamma_\ell(t) \in (0,\infty]$ we have that $\Gamma_\ell^{-1}:[0,A_\ell) \to [0,\sigma_0 (S_\ell))$ is continuous and strictly increasing, and that $\Gamma_\ell^{-1}(u)$ is a stopping time for any $u < A_\ell$.
For $u<A_\ell$ put
\[
W(u)=\int_0^{\Gamma_\ell^{-1}(u)}\sqrt{f_\alpha(R_{\ell}(s))}\, dB(s).
\]
By the time-change for continuous local martingales and their stochastic integrals \cite[Ch.~V, \S1]{ryCMBM}, $W$ is a continuous local martingale on $[0,A_\ell)$ with quadratic variation
\[
\langle W\rangle_u=\int_0^{\Gamma_\ell^{-1}(u)}f_\alpha(R_{\ell}(s))\,ds=\Gamma_\ell\big(\Gamma_\ell^{-1}(u)\big)=u,
\]
hence a Brownian motion on $[0,A_\ell)$. Enlarging the probability space if necessary, we can extend it to a Brownian motion on $[0,\infty)$ by attaching an independent Brownian motion at time $A_\ell$. 
Moreover, since the drift in \eqref{eq:SDE} vanishes on $[0,\sigma_0 (S_\ell))$, the same time-change  gives
\begin{align}
  Y_{\ell}\big(\Gamma_\ell^{-1}(u)\big)
&=1+2\int_0^{\Gamma_\ell^{-1}(u)}\sqrt{Y_{\ell}(t)}\sqrt{f_\alpha(R_{\ell}(t))}\,dB(t) \nonumber \\
&=1+2\int_0^u \sqrt{Y_\ell\big(\Gamma_\ell^{-1}(v)\big)}\,dW(v),\quad u<A_\ell, \label{Yl-tc}
\end{align}
where the last equality follows from a change of variables $v=\Gamma_\ell(t)$ and the definition of $W$. 
If $Z_\ell$ is the unique strong solution on $[0,\infty)$ of $Z_\ell(u) = 1 + 2 \int_0^u \sqrt{ Z_\ell(v)} \, dW(v)$, then it follows from \eqref{Yl-tc} that $Y_\ell(\Gamma_{\ell}^{-1}(u)) = Z_\ell(u)$ for $u<A_\ell$, or equivalently $Y_\ell(t) = Z_\ell(\Gamma_\ell(t))$ for $t < \sigma_0 (S_\ell)$.
\end{proof}
\end{proof}

\subsection{The augmented joint RKT for PSR walks}\label{sec:PSRaug}
The main result of this section is the following theorem.
\begin{theorem}[Augmented joint RKT for PSR walks]\label{JointRKT}
  Consider a PSR walk with weight function $w$ as in \eqref{w} for
  some $\alpha>0$ and with rescaled local time profiles as in
  \eqref{Ln-PSR}.  For any $N\in\N$ and any choice of
  $(x_i,h_i)\in\R\times\R_+$, $1\le i\le N$, the joint distribution
  (in the natural product topology) of the processes
  $\Lambda^n_{x_i,h_i}(\cdot)$, the rescaled endpoints
  $\mu^{n,*}_{x_i,h_i}$, and merging points
  $\mu^{n,*}_{(x_i,h_i),(x_j,h_j)}$, $1\le i\le N$, $i<j\le N$,
  $*\in\{+,-\}$, converges as $n\to \infty$ to the joint distribution
  of $\Lambda_{x_i,h_i}(\cdot)$, $\mu^*_{x_i,h_i}$, and
  $\mu^*_{(x_i,h_i),(x_j,h_j)}$, $1\le i\le N$, $i<j\le N$,
  $*\in\{+,-\}$, where the joint distribution of the curves
    $(\Lambda_{x_i,h_i}(\cdot))_{i\leq N}$ is given by \eqref{WNrep}
    and \eqref{WNrep-back}.
\end{theorem}

We will prove Theorem~\ref{JointRKT} via
Proposition~\ref{prop:fJRK-JRK}, which reduces it to establishing
three statements: an augmented RKT, a ``forward'' augmented joint RKT,
and a ``backward'' augmented joint RKT (see
Definition~\ref{def:JRK-rl}). We first prove these statements for the
local times of directed edges rather than of sites, because edge local
times have a Markov property that site local times lack. Working with
edge local times also lets us use two known results for PSR walks,
both proved for edge local times: the augmented RKT
\cite[Theorem~1B]{tGRK} and the joint RKT 
\cite[Theorem~1.2]{Y25+}.
We then transfer these results to the rescaled site local
time profiles \eqref{Ln-PSR} via Lemma~\ref{lem:Lxh-barLxh}, which is
stated below and proved in Appendix~\ref{subsec:tech}.

We shall need notation for edge local time processes. Recall from
\eqref{upcr} that ${\mathcal E}^+(n,k)$ is the number of crossings of
the directed edge $(k,k+1)$ by time $n$.  For any $k\in \Z$ and
$m\geq 0$ let
 \begin{equation}\label{taue}
  \overline{\tau}_{k,m}:=
  \inf\{n \geq 0:\,\mathcal{E}^+(n+1,k)>m \}
  ,
 \end{equation}
 be the time of the last visit to $k$ before the $(m+1)$-st crossing
 of the directed edge $(k,k+1)$, so that
 $X_{\overline{\tau}_{k,m}} = k$ and
 $\mathcal{E}^+(\overline{\tau}_{k,m},k) = m$,\footnote{This
   definition differs by one unit of time from
   $\inf\{n \geq 1:\, \mathcal{E}^+(n,k)>m\}$ used in \cite[(1.3)]{Y25+}.
   The two definitions give directed edge local time profiles that
   differ only at the site $k$, where they equal $m$ and $m+1$,
   respectively. Hence the corresponding rescaled curves
   \eqref{barLxh} differ by $2/n$ at a single site, their endpoints
   (as defined below) differ by at most $1/n$, and the Ray-Knight
   theorems proved in \cite{Y25+} hold for either definition;
   cf.\ also \cite[Remark~2.4]{Y25+}. The
   advantage of \eqref{taue} is that the walk is at $k$ at time
   $\overline{\tau}_{k,m}$, which yields the coalescence property of
   the curves \eqref{barLxh}, see Lemma \ref{lem:coalesce}(i).}  and
 for $x\in \R$ and $h\geq 0$ let
 \begin{equation}\label{barLxh}
  \overline{\Lambda}_{x,h}^n(y) := 2n^{-1} \mathcal{E}^+(\overline{\tau}_{\fl{xn}, \fl{hn/2}} , \fl{yn} ), \quad y\in\R.
 \end{equation}
 Note that the multiplicative factor of $2$ in \eqref{barLxh} is needed to compare with the curves $\Lambda_{x,h}^n(\cdot)$ as defined in \eqref{Ln-PSR}, since
 \begin{equation}\label{site-edge}
  \left|L(n,k) - \mathcal{E}^+(n,k) - \mathcal{E}^+(n,k-1)\right| \leq 1.
 \end{equation}
 The results in \cite[Theorem 1B]{tGRK}\footnote{Even though the stopping times in \cite{tGRK} are not the
  same as in \eqref{taue}, the results of \cite{tGRK} can be
  transferred to $\overline{\Lambda}_{x,h}^n$ by separating the parts
  of the curve to the right and to the left of $x$ and using
  \eqref{absorption} and \eqref{merge-Z} to show that
  $\overline{\Lambda}_{x,h}^n$ is close to the processes in
  \cite{tGRK}. We note that the proofs of \eqref{absorption} and
  \eqref{merge-Z} rely only on \cite[(5.7.19), (5.8.3)]{tGRK} and not
  on Theorem 1B.} and
 \cite[Theorem~1.2]{Y25+}
 give augmented RKTs
 and joint RKTs for the curves $\overline{\Lambda}_{x,h}^n(\cdot)$,
 and below we will extend these joint RKTs to also include the joint
 convergence of the endpoints and merge points\footnote{We need to define the left endpoints and merge points slightly differently here: since we are working with right directed edge local times, we can only conclude that $\mathcal{E}^+(n_1,k) = \mathcal{E}^+(n_2,k)$ implies that $\mathcal{E}^+(n_1,\ell) = \mathcal{E}^+(n_2,\ell)$ for all $\ell \leq k$ if $k$ is strictly less than $X_{n_1}$ and $X_{n_2}$.}
 
 \begin{align*}
     \overline{\mu}_{x,h}^{n,+} &= m_{x,0}^+(\overline{\Lambda}_{x,h}^n,\mathbf{0}), 
  \qquad \overline{\mu}_{(x,h),(x',h')}^{n,+} = m_{x,x'}^+(\overline{\Lambda}_{x,h}^n,\overline{\Lambda}_{x',h'}^n),
\\
  \overline{\mu}_{x,h}^{n,-} &= \sup\left\{y< \frac{\fl{xn}}{n} \wedge 0: \, \overline{\Lambda}_{x,h}^n(y) = 0 \right\}, \quad \text{and}
\\
  \overline{\mu}_{(x,h),(x',h')}^{n,-} &= \sup\left\{ y < \frac{\fl{xn}}{n} \wedge \frac{\fl{x'n}}{n} : \, \overline{\Lambda}_{x,h}^n(y) = \overline{\Lambda}_{x',h'}^n(y) \right\} 
 \end{align*}
of the family of curves defined in \eqref{barLxh}. The following technical lemma is proven in Appendix \ref{subsec:tech}.

\begin{lemma}\label{lem:Lxh-barLxh}
  For any $(x,h),(x',h') \in \R\times [0,\infty)$ and any
    $\delta>0$ we have
 \begin{align}
 \lim_{n\to \infty} P\left( \sup_{y \in \R} \left| \Lambda_{x,h}^n(y) - \overline{\Lambda}^n_{x,h}(y) \right| > \d \right) &= 0, \label{barLxh-comp} \\
  \lim_{n\to\infty} 
  P\left( |\mu_{x,h}^{n,\pm} - \overline{\mu}_{x,h}^{n,\pm} | > \d \right) &= 0, \label{barmuxh-comp} \\
  \text{and} \quad\lim_{n\to\infty}
  P\left( |\mu_{(x,h),(x',h')}^{n,+} - \overline{\mu}_{(x,h),(x',h')}^{n,+} | > \d \right) &= 0. \label{barmerge-comp}
\end{align}
\end{lemma}
\begin{remark}
 Note that \eqref{barmerge-comp} only gives a comparison of the right merge locations for the two families of curves, and this is because the \emph{right} directed edge local times have a natural joint Markov property to the right (similar to the directional joint Markov property for the limiting curves in Assumption \ref{asm:JMarkov}). The comparison \eqref{barmerge-comp} will therefore only allow us to prove a \emph{forward} augmented joint RKT (see Definition \ref{def:JRK-rl} for a precise definition). By symmetry, a \emph{backward} augmented joint RKT can be proved by working with left directed edge local times instead, and by Proposition \ref{prop:fJRK-JRK} this will be sufficient for proving the augmented joint RKTs in Theorem \ref{JointRKT}.
\end{remark}

We will now state precisely the limit theorems that we will need for the curves $\overline{\Lambda}_{x,h}^n$.
 \begin{theorem}\label{thm:barLxh-JRK}
  For a PSR walk, let the family of curves $\overline{\Lambda}_{x,h}^n(\cdot)$ be defined as in \eqref{barLxh} and the family of curves $\Lambda_{x,h}(\cdot)$ be given by \eqref{WNrep}-\eqref{WNrep-back}. For $(x,h),(x',h') \in \R\times [0,\infty)$, let $\mu_{x,h}^{\pm} = m_{x,0}^\pm(\Lambda_{x,h},\mathbf{0})$ and $\mu_{(x,h),(x',h')}^{\pm} = m_{x,x'}^\pm(\Lambda_{x,h},\Lambda_{x',h'})$ be the endpoints and merge points of the curves in the latter family. 
  Then, we have the following. 
  \begin{enumerate}[--]
   \item \textbf{\em Augmented RKT (\cite{tGRK}).} For any $(x,h) \in \R\times [0,\infty)$, 
\begin{equation}\label{Yu-aRK}
   \left( \overline{\Lambda}^n_{x,h}(\cdot), \overline{\mu}_{x,h}^{n,-}, \overline{\mu}_{x,h}^{n,+} \right) 
  \Longrightarrow
  \left( \Lambda_{x,h}(\cdot), \mu_{x,h}^-, \mu_{x,h}^+ \right).
  \end{equation}
\item \textbf{\em Joint RKT (\cite[Theorem~1.2]{Y25+}).
}
  For any $(x_1,h_1), \ldots, (x_N,h_N) \in \R \times [0,\infty)$,
  \begin{equation}\label{Yu-JRK}
   \left( \overline{\Lambda}^n_{x_i,h_i}(\cdot) \right)_{i\leq N} 
  \Longrightarrow
  \left( \Lambda_{x_i,h_i}(\cdot) \right)_{i\leq N}.
  \end{equation}
 \item \textbf{\em Forward augmented joint RKT.} For any $(x_1,h_1), \ldots, (x_N,h_N) \in \R \times [0,\infty)$, 
 \begin{equation}\label{Yu-fJRK}
 \begin{split}
  &\left( (\overline{\Lambda}^n_{x_i,h_i}(\cdot))_{i\leq N}, \, ( \overline{\mu}_{x_i,h_i}^{n,+} )_{i\leq N}, \, (\overline{\mu}_{(x_i,h_i),(x_j,h_j)}^{n,+} )_{1\leq i < j \leq N} \right) \\
  &\quad \Longrightarrow
  \left( (\Lambda_{x_i,h_i}(\cdot))_{i\leq N}, \, ( \mu_{x_i,h_i}^{+} )_{i\leq N}, \, ( \mu_{(x_i,h_i),(x_j,h_j)}^{+} )_{1\leq i < j \leq N} \right).
 \end{split}
 \end{equation}
  \end{enumerate}
 \end{theorem}
The only new content of Theorem \ref{thm:barLxh-JRK} is the forward augmented joint RKT \eqref{Yu-fJRK}, which we will prove below. Postponing for the moment the proof of \eqref{Yu-fJRK}, we show how to complete the proof of Theorem \ref{JointRKT}.

\begin{proof}[Proof of Theorem \ref{JointRKT}]
  By the ``converging together'' theorem, \cite[Theorem 3.1]{bCOPM},
  the augmented Ray-Knight theorem follows from \eqref{barLxh-comp},
  \eqref{barmuxh-comp}, and \eqref{Yu-aRK}, and the forward augmented
  joint Ray-Knight theorem follows from
  \eqref{barLxh-comp}--\eqref{barmerge-comp}, and \eqref{Yu-fJRK}. By
  the symmetry of the PSR walk model, a backward augmented joint
  Ray-Knight theorem also holds\footnote{The
      symmetric version of the forward augmented joint Ray-Knight
      theorem is not exactly the backward augmented joint Ray-Knight
      theorem: since the integer part function is not symmetric, we get
      slightly different stopping times. However, we can deduce the
      needed backward theorem from the symmetric version of
      \eqref{Yu-fJRK} with the help of the arguments of Appendix
      \ref{subsec:tech}, \eqref{absorption} and
      \eqref{merge-Z}.}. Moreover, the walk is
    recurrent, which can be seen for example from Corollary 1B of
    \cite{tGRK}, and the limiting curves in the augmented Ray-Knight
    theorem satisfy Assumption \ref{asm:Bessel}. Consequently, an
  application of Proposition \ref{prop:fJRK-JRK}
  yields the augmented joint Ray-Knight
    theorem. Moreover, \eqref{barLxh-comp}, \eqref{Yu-JRK}, and
    \cite[Theorem 3.1]{bCOPM} show that the functional limit of
    $(\Lambda^n_{x_i,h_i}(\cdot))_{i\leq N}$ is given
    by \eqref{WNrep} and \eqref{WNrep-back}. This completes the proof
  of Theorem \ref{JointRKT}.
\end{proof}
 
 We now return to the proof of \eqref{Yu-fJRK} for which we
 will need the following notation.
 \begin{definition}\label{def:sig}
  For a stochastic process $(S(y))_{y\in \R}$, any $x\in \R$ and $\d\geq 0$ let
\[
  \sigma_{x,\delta}(S) := \inf\{y \geq x:\,|S(y)| \le \delta \text{ or }|S(y-)| \le \delta\},
  \quad \text{where } \inf\emptyset :=\infty,
\]
$S(y-)=\lim\limits_{z \to y-}S(z)$ if $y>x$ and $S(x-)=S(x)$. Similarly, for a stochastic process  $(S_i)_{i\in\Z_+}$ and any $\d\geq 0$ let
$\sigma_\delta(S) := \inf\{ i\in\Z_+: |S_i| \leq \delta \}$, where $\inf\emptyset:= \infty$.
 \end{definition}

 \begin{proof}[Proof of \eqref{Yu-fJRK}]
For $\d>0$, using the notation from Definition \ref{def:sig}, we define
\begin{alignat*}{2}
 \overline{\mu}_{x,h}^{n,+,\d} &= \sigma_{x\vee 0,\d}(\overline{\Lambda}_{x,h}^n ),\qquad  \overline{\mu}_{(x,h),(x',h')}^{n,+,\d} &&= \sigma_{x \vee x', \d}(\overline{\Lambda}_{x,h}^n - \overline{\Lambda}_{x',h'}^n); \\
\mu_{x,h}^{+,\d} &= \sigma_{x\vee 0,\d}(\Lambda_{x,h} ),  \qquad  \mu_{(x,h),(x',h')}^{+,\d} &&= \sigma_{x \vee x', \d}(\Lambda_{x,h} - \Lambda_{x',h'}).
\end{alignat*}
Observe that for fixed
$(x,h),(x',h') \in \R\times [0,\infty)$ and $\d>0$, the maps
$S\mapsto \sigma_{x\vee 0,\delta}(S)$ and
$(S_1,S_2)\mapsto \sigma_{x\vee x', \delta} (S_2-S_1)$ are a.s.\
continuous at $\Lambda_{x,h}(\cdot)$ and
$(\Lambda_{x,h}(\cdot), \Lambda_{x',h'}(\cdot))$,
respectively\footnote{
  By \cite[Ch.\,VI, Proposition 2.11]{JS} (applied with the function $e^{-|S|}$ and the level $a=e^{-\d}$), it suffices to check that, a.s., the processes $(\Lambda_{x,h}(y))_{y\geq x\vee 0}$ and $(|\Lambda_{x,h}(y)-\Lambda_{x',h'}(y)|)_{y\geq x\vee x'}$ cross the level $\d$ upon hitting it. \cite[Ch.\,VI, Lemma 2.10]{JS} implies that for any realization of the process, it is true for a.e. $\delta\in(0,1)$, hence for a.e.\ $\delta\in(0,1)$ it is true a.s., which is actually enough for our purposes.
}.
Then, it follows from \eqref{Yu-JRK} and the continuous mapping theorem that
\begin{equation}\label{almost-aJRK}
 \begin{split}
   &\left( (\overline{\Lambda}^n_{x_i,h_i}(\cdot))_{i\leq N}, \, ( \overline{\mu}_{x_i,h_i}^{n,+,\d} )_{i\leq N}, \, (\overline{\mu}_{(x_i,h_i),(x_j,h_j)}^{n,+,\d} )_{1\leq i < j \leq N} \right) \\
  &\quad \Longrightarrow
  \left( (\Lambda_{x_i,h_i}(\cdot))_{i\leq N}, \, ( \mu_{x_i,h_i}^{+,\d} )_{i\leq N}, \, ( \mu_{(x_i,h_i),(x_j,h_j)}^{+,\d} )_{1\leq i < j \leq N} \right).
 \end{split}
\end{equation}
By \cite[Theorem 3.2]{bCOPM} the proof of \eqref{Yu-fJRK} will be complete if we can show that for any $(x,h),(x',h') \in \R \times [0,\infty)$ we have 
\begin{align}
&\lim_{\d\to 0} \mu_{x,h}^{+,\d} = \mu_{x,h}^+
 \quad \text{and}\quad 
 \lim_{\d\to 0} \mu_{(x,h),(x',h')}^{+,\d} = \mu_{(x,h),(x',h')}^+, \quad \text{a.s.}, \label{dnear-cont} \\
 &\lim_{\d\to 0} \limsup_{n\to\infty} P\left( |\overline{\mu}_{x,h}^{n,+,\d} - \overline{\mu}_{x,h}^{n,+} | > \e \right) = 0, \quad \forall\, \e>0, \label{dnear-end}\\
 \text{and}\quad 
 &\lim_{\d\to 0} \limsup_{n\to\infty} P\left( |\overline{\mu}_{(x,h),(x',h')}^{n,+,\d} - \overline{\mu}_{(x,h),(x',h')}^{n,+} | > \e \right) = 0, \quad \forall\, \e>0. \label{dnear-merge}
\end{align}
Since the curves $\Lambda_{x,h}(\cdot)$ are continuous a.s., \eqref{dnear-cont} holds. Thus, it remains only to prove \eqref{dnear-end} and \eqref{dnear-merge}.

  We will prove \eqref{dnear-end} and \eqref{dnear-merge} by working with the
  unscaled directed edge local times. For $k \in \Z$ and
  $0\leq m\leq m'$, let
  $(Z_i,Z'_i) = (\mathcal{E}^+(\overline{\tau}_{k,m},k+i),
  \mathcal{E}^+(\overline{\tau}_{k,m'},k+i))$, $i\ge
  0$, and note that $(Z_0,Z'_0)=(m,m')$ by \eqref{taue}.
  We denote the law of
  $(Z_i,Z'_i)_{i\geq 0}$ by $P_{k,m,m'}$, the law of $(Z_i)_{i\geq 0}$
  by $P_{k,m}$, and set $\Delta_i = Z'_i-Z_i$. We will use the
  following properties.
\begin{enumerate}[--]
\item {\em Markov property and homogeneity.} Under $P_{k,m,m'}$ the
  process $(Z_i,Z'_i)_{i\geq 0}$ is a Markov chain whose transition
  probabilities from time $i-1$ to time $i$ depend on $(k,i)$ only
  through the sign of $k+i$ (see Appendix \ref{sec:Polya}). In
  particular, for $k\geq 0$ the chain $(Z_i)_{i\geq 0}$ is
  homogeneous (the ``right'' chain); for $k<0$,
  $(Z_i)_{0\leq i<|k|}$ is homogeneous with law independent of $k$
  (the ``left'' chain), the single step from $|k|-1$ to $|k|$ is
  exceptional, and, conditionally on $Z_{|k|}=m_0$, the process
  $(Z_{|k|+i})_{i\geq 0}$ has law $P_{0,m_0}$.
\item {\em Monotonicity.} If $m\leq m'$ then
  $\overline{\tau}_{k,m}\leq \overline{\tau}_{k,m'}$, so that $Z_i\leq Z'_i$
  for all $i\geq 0$, $\sigma_0(\Delta)\leq \sigma_0(Z')$, and
  $m\mapsto P_{k,m}(\sigma_0(Z)>t)$ is non-decreasing.
\item {\em Tails of increments.} It follows from
  \cite[Lemma 3.1]{KMMP26} and the representation of the transition
  probabilities of $(Z_i)_{i\geq 0}$ in Appendix \ref{sec:Polya}
  that there are constants $C,c>0$ such that
  \begin{equation}\label{Zstep-tail}
   P_{k,m}\left( |Z_{i+1}-Z_i| \geq y \mid Z_i=j \right)
   \leq C e^{-c\frac{y^2}{y\vee j}},
   \quad \forall\, k\in\Z,\ m \geq 0,\ i\geq 0,\ j\geq 1,\ y>0.
  \end{equation}
\end{enumerate}
By \eqref{barLxh} and the Markov property, \eqref{dnear-end} reduces
to showing that a single process started at $m\leq \fl{\d n}$
reaches $0$ within $\fl{\e n}$ steps, and \eqref{dnear-merge} to
showing that two processes started within $\fl{\d n}$ of each other
merge within $2\fl{\e n}$ steps, in both cases with high
probability. That is, it is enough to prove that for any $\e,\eta>0$
there exists $\d>0$ such that for all sufficiently large $n$
\begin{equation}\label{absorption}
 \sup_{k\in\Z}\ \sup_{m\leq \d n} P_{k,m}\left( \sigma_0(Z) > \e n \right) < \eta,
\end{equation}
\begin{equation}\label{merge-Z}
 \sup_{k\in\Z}\ \sup_{0\leq m\leq m'\leq m+\d n} P_{k,m,m'}\left( \sigma_0(\Delta) > 2\e n \right) < 2\eta.
\end{equation}

We first prove \eqref{absorption}. It suffices to
prove the bound in \eqref{absorption} for each of three ranges of
$k$ with some range-dependent $\d>0$; then \eqref{absorption} holds
with the smallest of the three.
We will use the absorption estimate \cite[(5.7.19)]{tGRK} for the
homogeneous chains: for any $\e',\eta'>0$ there is a
$\bar{\delta}=\bar{\delta}(\e',\eta')>0$ such that for all sufficiently large $n$,
\begin{equation}\label{absorb-hom}
 \sup_{m\leq \bar{\delta} n} P^{\pm}_m\left( \sigma_0(Z) > \e' n \right) < \eta',
\end{equation}
where $P^{+}_m$ and $P^{-}_m$ are the laws of the right and left
chains started at $m$ (see Appendix \ref{sec:Polya}). The estimate
(5.7.19) is proved in \cite{tGRK} for the right chain; its proof
relies only on the asymptotics of \cite[Lemma 2B]{tGRK}, which are the
same for both chains, so \eqref{absorb-hom} holds for both.

Set $\d_+=\bar{\delta}(\e/2,\eta/6)$. For $k\geq 0$ we have $P_{k,m}=P^+_m$,
and \eqref{absorb-hom} gives
$\sup\limits_{k\geq 0}\sup\limits_{m\leq\d_+ n}P_{k,m}(\sigma_0(Z)>\e n)<\eta/6$
for all large $n$.
Next, fix $\e_1\in(0,\e/2)$ and $\d_0\in(0,\d_+)$, to be chosen
below, and let $-\e_1 n\leq k<0$ and $m\leq \d_0 n$. By the Markov
property at time $|k|$ and the inequality $\e n - |k| \geq \e n/2$,
\begin{align}\label{crossing-dec}
 P_{k,m}\left( \sigma_0(Z) > \e n \right)
 \leq P^{-}_m\Big( \max_{0\leq i < |k|} Z_i \geq \frac{\d_+ n}{2} \Big)
 &+ \sup_{j\leq \d_+ n/2} P_{k,m}\Big( Z_{|k|} > \d_+ n \,\Big|\, Z_{|k|-1} = j \Big) \\ \nonumber
 &+ \sup_{m_0\leq \d_+ n} P_{0,m_0}\Big( \sigma_0(Z) > \frac{\e n}{2} \Big).
\end{align}
The third term is less than $\eta/6$ by the case $k\geq 0$. The
second is at most $Ce^{-c\d_+ n/4}<\eta/6$ for large $n$ by
\eqref{Zstep-tail}. For the first, the rescaled left chain converges
to a multiple of a reflected $\mathrm{BESQ}^1$ process,
\cite[(5.8.3)]{tGRK}, so by Proposition \ref{prop:BESQdev} and
Portmanteau's theorem its limsup as $n\to\infty$ is bounded,
uniformly in $m\leq\d_0 n$, by
$C\exp\{-c(\d_+/2-\d_0)^2/(\d_+\e_1)\}<\eta/6$, provided that $\e_1$
and $\d_0$ are fixed small enough. Hence
$\sup\limits_{-\e_1 n\leq k<0}\sup\limits_{m\leq\d_0 n}
P_{k,m}(\sigma_0(Z)>\e n)<\eta/2$ for all large $n$.

Finally, for $k<-\e_1 n$ the process $(Z_i)_{0\leq i\leq \e_1 n}$ has
the law of the first $\fl{\e_1 n}$ steps of the left chain, so
\eqref{absorb-hom} with $(\e',\eta')=(\e_1,\eta/2)$ yields a $\d_->0$
with
$\sup\limits_{k<-\e_1 n}\sup\limits_{m\leq \d_- n}
P_{k,m}(\sigma_0(Z)>\e_1 n)<\eta/2$ when $n$ is large enough. Since $\e_1<\e$,
\eqref{absorption} holds for all $\d\leq\min\{\d_+,\d_0,\d_-\}$.

Returning now to the proof of \eqref{merge-Z}, let $\d'>0$ be
  such that the bound in \eqref{absorption} holds for all
  $\d \leq \d'$. Then \eqref{absorption} allows us to obtain the
estimates needed in \eqref{merge-Z} when $m \leq \d'n/2$. Indeed, in
this case if $\d < \d'/2$ then $m' \leq m+ \d n \leq \d' n$ and since
$\sigma_0(\Delta) \leq \sigma_0(Z')$ it follows that
$P_{k,m,m'}(\sigma_0(\Delta) > 2 \e n) \leq P_{k,m'}(\sigma_0(Z)
  > 2\e n) < \eta$ for $n$ large.  Therefore, to prove
\eqref{merge-Z} we need only to show that for any $\e,\eta>0$ and
$\d'$ as above there is a $\d \in (0,\d'/2)$ such that for all
$n$ large enough
\begin{equation}\label{merge-Z2}
 \sup_{k\in\Z}
 \sup_{\frac{\d'n}{2} < m \leq m' \leq m+\d n} P_{k,m,m'} \left( \sigma_0(\Delta) > 2\e n \right) < 2\eta.
\end{equation}
To prove \eqref{merge-Z2}
we will use that $(Z_i,Z'_i)$ is a joint Markov
process whose transition probabilities can be described by generalized
P\'olya urn processes (see Appendix \ref{sec:Polya}). While this
is a time-inhomogeneous Markov process, there are only three different
types of transition probabilities to consider (corresponding to three
different generalized P\'olya urns) and we will use estimates on the
transition probabilities that hold for all three of these cases.  In
particular, we will use the following facts:
\begin{enumerate}[--]
 \item By Lemma \ref{D-bounds}(i), for any $\d_1>0$ there exists a $z_1 \in \N$ such that
 \begin{equation}\label{Z-smalldrift}
  E_{k,m,m'}[ \Delta_1 - \Delta_0 ] < \d_1, \quad \forall\, k\in \Z, \,  z_1 \leq m \leq m'.
 \end{equation}
\item By Lemma \ref{D-bounds}(ii), there is a $\d_2 \in (0,1)$
 and $c_2>0$ such that
\begin{equation}\label{Zstepbound}
 P_{k,m,m'}\left( \Delta_1 - \Delta_0 \leq - c_2 \sqrt{\Delta_0} \right) > c_2, \quad \forall\, k \in \Z, \, m < m' \leq (1+\d_2) m.
\end{equation}
\end{enumerate}
Since all of these estimates are uniform over $k\in \Z$, for the
remainder of the proof we will suppress the dependence of the
processes on the spatial starting location $k$, writing $P_{m,m'}$
and $P_m$ instead of $P_{k,m,m'}$ and $P_{k,m}$, and we will assume
that $\frac{\d' n}{2} < m \leq m' \leq m+\d n$.

The main tool that we will use in our proof of \eqref{merge-Z2} is that $\sqrt{\Delta_i}$ is a supermartingale as long as both $Z$ and $Z'$ stay sufficiently large and the ratio $\frac{\Delta_i}{Z_i}$ is sufficiently small.
To be more precise, let $\d_2$ and $c_2$ be as in \eqref{Zstepbound} and choose $\d_1 < \frac{c_2^3}{8}$ and  $z_1 \in \N$ so that \eqref{Z-smalldrift} holds.
We will show that if
\[
 \beta = \sigma_0(\Delta) \wedge
 \inf\{i\geq 0: Z_i < z_1, \text{ or } Z'_i > (1+\delta_2)Z_i \},
\]
then $\sqrt{\Delta_{i \wedge \beta }}$ is a supermartingale with respect to $\mathcal{F}^{Z,Z'}_i = \sigma(Z_j, Z'_j, \, j\leq i),\ i\ge 0$.
To this end, note that
a second order Taylor expansion gives that
\[
 \sqrt{y} - \sqrt{x} =  \frac{1}{2\sqrt{x}} (y-x) - \frac{1}{8 z^{3/2}} (y-x)^2,
\]
for some $z$ between $x$ and $y$, and since $z^{-3/2}$ is decreasing in $z$ we have that if $i < \beta$ then \eqref{Z-smalldrift} and \eqref{Zstepbound} imply that
\begin{align*}
&E_{m,m'}\left[ \sqrt{\Delta_{i+1}} \mid \mathcal{F}^{Z,Z'}_i \right] - \sqrt{\Delta_i} \\
&\quad \leq \frac{1}{2\sqrt{\Delta_i}} E_{m,m'}\left[ \Delta_{i+1} - \Delta_i \mid \mathcal{F}^{Z,Z'}_i  \right]
- \frac{1}{8(\Delta_i)^{3/2}} E_{m,m'}\left[ (\Delta_{i+1} - \Delta_i )^2 \ind{\Delta_{i+1} - \Delta_i \leq -c_2\sqrt{\Delta_i}} \mid \mathcal{F}^{Z,Z'}_i \right] \\
&\quad \leq \frac{\d_1}{2\sqrt{\Delta_i}}
- \frac{c_2^2}{8\sqrt{\Delta_i}} P_{m,m'}\left( \Delta_{i+1} - \Delta_i \leq -c_2\sqrt{\Delta_i} \mid \mathcal{F}^{Z,Z'}_i \right)  \leq -\left( \frac{c_2^3}{8} - \frac{\d_1}{2} \right) \frac{1}{\sqrt{\Delta_i}}
\leq -\frac{c_2^3}{16} \frac{1}{\sqrt{\Delta_i}},
\end{align*}
where in the last inequality we used the choice of $\d_1 < \frac{c_2^3}{8}$.
In fact, the above argument shows that not only is $\sqrt{\Delta_{i\wedge \beta }}$ a supermartingale,  but also that $(M_i)_{i\geq 0}$ defined by
\[
 M_0 = \sqrt{\Delta_0}, \quad\text{and}\quad M_i = \sqrt{\Delta_{i\wedge \beta  } }
 + \frac{c_2^3}{16} \sum_{j=0}^{(i\wedge \beta ) - 1} \frac{1}{\sqrt{\Delta_j}}, \quad i\geq 1,
\]
is a supermartingale with respect to $\mathcal{F}^{Z,Z'}_i,\ i\ge 0$.

We will now use these non-negative supermartingales to
show that if $\Delta_0 \leq \d n$ then
$\beta \le \e n$ with high probability.  Indeed,
\begin{align}
 P_{m,m'}(\beta > \e n)
 &\leq P_{m,m'}\left( \sup_{i \leq \e n} \sqrt{\Delta_{i\wedge \beta}} > \sqrt{n} \right)
 + P_{m,m'}\left( \beta > \e n, \, \sup_{i\leq \e n} \sqrt{\Delta_i} \leq \sqrt{n} \right) \nonumber \\
 &\leq \frac{ E_{m,m'}[ \sqrt{\Delta_0} ]}{\sqrt{n} }
 + P_{m,m'} \left( M_{\fl{\e n}} \geq \frac{c_2^3 \fl{\e n} }{16 \sqrt{n}} \right) \nonumber \\
 &\leq \sqrt{\d} + E_{m,m'}\left[ M_{0} \right] \frac{16 \sqrt{n}}{c_2^3 \fl{\e n}  }
 \leq \sqrt{\d}\left( 1 + \frac{32}{c_2^3 \e}\right), \label{beta-large}
\end{align}
where the last inequality holds for $n$ sufficiently large (depending
only on $\e$).  Since this bound vanishes as $\d\to 0$,
we need only to 
obtain bounds on
$ P_{m,m'}(\sigma_0(\Delta) > 2\e n, \, \beta \leq \e n)$. 
To this end, note that
\begin{align}
 P_{m,m'}(\sigma_0(\Delta) > 2\e n, \, \beta \leq \e n)
 &\leq  P_{m,m'}\left( \sigma_0(\Delta) > 2\e n, \, \sigma_{\d' n/2}(Z) < \beta \leq \e n  \right) \label{sD1} \\
 &\quad + P_{m,m'}\left( \beta \leq \sigma_{\d' n/2}(Z) \wedge \e n  < \sigma_0(\Delta)  \right). \label{sD2}
\end{align}
To bound the probability on the right in \eqref{sD1}, note that the event $\{ \sigma_{\d' n/2}(Z) < \beta \leq \e n   \}$ implies that there is a $j\leq \e n$ with $Z_j \leq \frac{\d' n}{2}$ and $Z'_j\leq (1+\d_2) Z_j \leq (1+\d_2) \frac{\d'}{2} n < \d' n$.
Again, using that $\sigma_0(\Delta) \leq \sigma_0(Z')$ it follows from \eqref{absorption} that when $n$ is large enough,
\begin{align*}
 P_{m,m'}\left( \sigma_0(\Delta) > 2\e n, \, \sigma_{\d' n/2}(Z) < \beta \leq \e n  \right)
 \leq P_{m,m'}\left( \sigma_0(Z') > 2\e n, \sigma_{\d' n}(Z') \leq \e n  \right) < \eta.
\end{align*}
 To bound \eqref{sD2}, first note that
\begin{align}
  &P_{m,m'}\left( \beta \leq \sigma_{\d' n/2}(Z) \wedge \e n  < \sigma_0(\Delta)   \right)  \nonumber \\
  &\qquad \leq
 P_{m,m'}\left( \beta < \sigma_{\d'n/4}(Z) \wedge \sigma_0(\Delta) \right)
 + P_{m}\left( \sigma_{\d'n/4}(Z) = \sigma_{\d'n/2}(Z) \leq \e n \right).  \label{sD2a}
\end{align}
For the first probability in \eqref{sD2a}, for $n$ large enough, $\beta < \sigma_{\d' n/4}(Z) \wedge \sigma_0(\Delta)$
implies that we must have $Z_{\beta} \geq \d' n/4 \geq z_1$ and $Z'_{\beta} \geq (1+\d_2)Z_{\beta}$ so that
$\Delta_{\beta} \geq \d_2 Z_{\beta} \geq \frac{\d_2 \d'n}{4}$.
Then it follows from Doob's maximal inequality that
\begin{equation}\label{beta-sZ}
 P_{m,m'}\left( \beta < \sigma_{\d'n/4}(Z) \wedge \sigma_0(\Delta) \right)
 \leq P_{m,m'}\left( \sup_{i\geq 0} \sqrt{\Delta_{i\wedge \beta}} \geq \frac{\sqrt{\d_2 \d' n}}{2} \right)
 \leq
 2\sqrt{\frac{\d}{\d_2 \d'}}.
\end{equation}
For the second probability in \eqref{sD2a}, it follows from \eqref{Zstep-tail} that
\[
 P_{m}\left( \sigma_{\d'n/4}(Z) = \sigma_{\d'n/2}(Z) \leq \e n \right)
 \leq \fl{\e n} \sup_{\ell > \d' n/2} P_{\ell}\left( Z_1 \leq \frac{\d' n}{4} \right)
 \leq \fl{\e n} C e^{-c \d' n/16}.
\]
Combining the above estimates and choosing $\d$ small enough so that the right sides of \eqref{beta-large} and \eqref{beta-sZ} are both less than $\eta/4$, we have that \eqref{merge-Z2} holds for all $n$ sufficiently large.
 \end{proof} 

  \begin{proof}[Proof of Theorem~\ref{main_psr}]
Above, we verified Assumptions~\ref{asm:Bessel}--\ref{asm:ND} and proved the augmented joint RKT. The polynomially self-repelling random walk is recurrent, which can be seen for example from Corollary 1B of \cite{tGRK}. Hence, Theorems~\ref{thm:main_X}--\ref{thm:main_JRK} apply and yield all statements of Theorem~\ref{main_psr}.
\end{proof}

\subsection*{Acknowledgement}
This paper, including all of its novel ideas and arguments, was conceived and written by the authors. AI was used for some of the technical results in Appendices \ref{subsec:tech}, \ref{app:contprop},
and \ref{app:integrals} that should be standard but were not available in the literature in the form we needed. The literature search for these results with large language models (Gemini, Google; Claude, Anthropic) produced their initial proof drafts instead of references. We verified, corrected, and rewrote these proofs. 

\appendix

\section{Results about generalized P\'olya urn model}\label{sec:Polya}

In this section we shall recall relevant facts about generalized
P\'olya urns and their connection to the directed edge local times of
self-interacting random walks \eqref{sirw}. Our construction of
generalized P\'olya urns follows closely that of \cite{tGRK} and
\cite{KMP23}. Then we focus on PSR walks. The main result in
this section is Lemma~\ref{D-bounds} that we used in the proof of
Theorem \ref{JointRKT}.

Given sequences of positive numbers
  $(b(i))_{i\geq 0}$ and $(r(i))_{i\geq 0}$, the corresponding generalized
  P\'olya urn process $((\mathfrak{B}_j, \mathfrak{R}_j))_{j\geq 0}$ 
  is a Markov chain on $\Z_+^2$ started at
  $(\mathfrak{B}_0, \mathfrak{R}_0) = (0,0)$ with transition
  probabilities given by
\begin{align*}
 P\left( (\mathfrak{B}_{j+1}, \mathfrak{R}_{j+1}) = (i+1,m) \mid (\mathfrak{B}_j, \mathfrak{R}_j) = (i,m) \right) &= \frac{b(i)}{b(i)+r(m)},\ \text{and}\\
P\left( (\mathfrak{B}_{j+1}, \mathfrak{R}_{j+1}) = (i,m+1) \mid (\mathfrak{B}_j, \mathfrak{R}_j) = (i,m) \right) &= \frac{r(m)}{b(i)+r(m)},\ \ i,m\ge 0.
\end{align*}
If we consider this as being generated by drawing red/blue balls from an urn, then $\mathfrak{B}_j$ and $\mathfrak{R}_j$ will be the numbers of blue and red balls, respectively, drawn from the urn up to time $j$.

Self-interacting random walks \eqref{sirw} involve
three types of associated generalized P\'olya urn processes that we
now describe.  For $* \in \{-,+,0\}$ we will let
$((\mathfrak{B}^*_j, \mathfrak{R}^*_j))_{j\geq 0}$ be the generalized
P\'olya urn process corresponding to the sequences
$(b^*(i))_{i\geq 0}$ and $(r^*(i))_{i\geq 0}$
\[
\begin{array}{l}
 b^-(i) = w(2i) \\
 r^-(i) = w(2i+1)
\end{array},
\qquad
\begin{array}{l}
 b^+(i) = w(2i+1) \\
 r^+(i) = w(2i)
\end{array},
\quad\text{and}\quad
\begin{array}{l}
 b^0(i) = w(2i) \\
 r^0(i) = w(2i)
\end{array},
\quad \text{for } i\geq 0.
 \]
 The connection of the self-interacting random walk with these urn
 processes is as follows. At any fixed site $\ell\leq -1$ the sequence of
 left/right steps on successive visits to $\ell$ has the same
 distribution as the sequence of blue/red draws from the generalized
 P\'olya urn process
 $((\mathfrak{B}^-_j, \mathfrak{R}^-_j))_{j\geq 0}$. Similarly, the
 process $((\mathfrak{B}^+_j, \mathfrak{R}^+_j))_{j\geq 0}$
 corresponds to the sequence of left/right steps from a site
 $\ell\geq 1$ and $((\mathfrak{B}^0_j, \mathfrak{R}^0_j))_{j \geq 0}$
corresponds to the sequence of left/right steps from $\ell=0$.

When considering any of the urn models described above, we will let $\tau_k^\mathfrak{B^*}$ (or $\tau_k^\mathfrak{R^*}$) denote the number of trials until a blue (or red) ball is selected for the $k$-th time. More explicitly, letting $\tau_0^\mathfrak{B^*} = \tau_0^{\mathfrak{R^*}}=0$ we have for $k\geq 1$ that
\begin{equation}\label{taub}
\tau_k^\mathfrak{B^*} = \inf\{ j> \tau_{k-1}^\mathfrak{B^*}:  \mathfrak{B}^*_j = \mathfrak{B}^*_{j-1} + 1 \}
\  \text{ and } \
\tau_k^\mathfrak{R^*} = \inf\{ j> \tau_{k-1}^{\mathfrak{R}^*}:  \mathfrak{R}^*_j = \mathfrak{R}^*_{j-1} + 1 \}.
\end{equation}
We will also let $\D^*_j$ be the signed difference in the number of red balls and blue balls drawn in the first $j$ steps of the urn process. That is,
\begin{equation}\label{discr}
 \D^*_j = \mathfrak{R}^*_j - \mathfrak{B}^*_j.
\end{equation}

We use generalized P\'olya urn processes to describe the
  Markovian structure of the directed edge local times. For
  details specific to various models we refer to \cite{tTSAW,
    tTSAWGBR, tGRK, KMP23}.

  The rest of this section concerns PSRWs. We will need a number of
  results that were proved in \cite{KMP23} for the generalized
  P\'olya urns associated with PSR walks. However, the results in
  \cite{KMP23} were shown only for the specific choice of the weight
  function $w(j) = (j+1)^{-\alpha}$, and so in \cite{KMMP26} we
  extended these results to more general setting where $w$ has the
  form \eqref{w}.

First, we recall the notation \eqref{upcr} and
  \eqref{ltaukm} and prove the following lemma which was used to justify
  \eqref{E+initial}.
\begin{lemma}\label{lem:PSRbalance}
 There exist constants $C,c>0$ such that
 \[
  P\left( \left| \mathcal{E}^{+}(\tau_{k,m},k) - \frac{m}{2} \right| > y \right) \leq C e^{-c \frac{y^2}{y\vee m}}, \quad \forall\, k \in \Z, \, m\geq 0, \, y > 0.
 \]
\end{lemma}
\begin{proof}
 Since $\mathcal{E}^+(\tau_{k,m},k)$ is the number of steps from $k$ to $k+1$ in the first $m$ steps away from $k$, then it is easy to see that
 $\mathcal{E}^+(\tau_{k,m},k)
  \overset{\text{Law}}{=} \mathfrak{R}_m^* = (\mathfrak{D}_m^* + m)/2$,
where $*=+$ if $k\geq 1$, $*=-$ if $k\leq -1$ and $*=0$ if $k=0$.
Therefore, $P\left( \left| \mathcal{E}^+(\tau_{k,m},k) - m/2 \right| \geq y \right) =
P\left( |\mathfrak{D}_m^*| > 2y \right)$, and the conclusion of the lemma follows from the tail bounds for $\mathfrak{D}^*_m$ in \cite[Lemma 3.1]{KMMP26}.
\end{proof}

Next, we explain how the transition probabilities of the joint
process
$(Z_i,Z'_i) = (\mathcal{E}^+(\overline{\tau}_{k,m},k+i),
\mathcal{E}^+(\overline{\tau}_{k,m'},k+i))$, used in the proof of
Theorem \ref{thm:barLxh-JRK}, relate to generalized urn processes.
In particular, for $\ell > k$ we wish to describe the distribution
of
$\left( \mathcal{E}^+(\overline{\tau}_{k,m},\ell),
  \mathcal{E}^+(\overline{\tau}_{k,m'},\ell) \right)$ conditioned on
$\left( \mathcal{E}^+(\overline{\tau}_{k,m},\ell-1),
  \mathcal{E}^+(\overline{\tau}_{k,m'},\ell-1) \right) = (j,j')$.
The analysis depends on whether $\ell \leq -1$, $\ell=0$, or
$\ell\geq 1$. Consider first one coordinate and $\ell\leq -1$. If
$\mathcal{E}^+(\overline{\tau}_{k,m},\ell-1) = j$, then the walk has made
$j$ steps from $\ell-1$ to $\ell$ by time $\overline{\tau}_{k,m}$. Since
the walk starts at $0>\ell$ and at time $\overline{\tau}_{k,m}$ is at
$k<\ell$, it has made $j+1$ steps to the left from $\ell$, and its
last step from $\ell$ was to the left.
Therefore,
$\mathcal{E}^+(\overline{\tau}_{k,m},\ell)$ has the same distribution as
$\mathfrak{R}^-_{\tau_{j+1}^{\B^{-}}} = \D^-_{\tau_{j+1}^{\B^-}} +
(j+1)$. In particular, the ``left'' chain of
Section~\ref{sec:PSRaug}, which has transitions corresponding to $\ell \leq -1$, coincides with the Markov chain defined in
\cite[(5.1.1), (5.1.2)]{tGRK} (note that the colors of the balls are
interchanged). The cases $\ell=0$ and $\ell\geq 1$ are similar, and
the ``right'' chain, which has transitions corresponding to $\ell \geq 1$, coincides with the Markov chain defined in
\cite[(5.1.1), (5.1.3)]{tGRK}. This analysis carries over to the
joint case as well, and we get the following:
\begin{align*}
 & P\left( \mathcal{E}^+(\overline{\tau}_{k,m},\ell) = j + y, \, \mathcal{E}^+(\overline{\tau}_{k,m'},\ell) = j' + y'
 \mid \mathcal{E}^+(\overline{\tau}_{k,m},\ell-1) = j , \, \mathcal{E}^+(\overline{\tau}_{k,m'},\ell-1) = j' \right) \\
 &\quad
 =
 \begin{cases}
  P\left( \D^-_{\tau_{j+1}^{\B^-}} = y-1, \, \D^-_{\tau_{j'+1}^{\B^-}} = y'-1 \right) & \text{if } \ell \leq -1 \\
  P\left( \D^0_{\tau_{j+1}^{\B^0}} = y-1, \, \D^0_{\tau_{j'+1}^{\B^0}} = y'-1 \right) & \text{if } \ell = 0 \\
  P\left( \D^+_{\tau_{j}^{\B^+}} = y, \, \D^+_{\tau_{j'}^{\B^+}} = y' \right) & \text{if } \ell \geq 1.
 \end{cases}
\end{align*}
In particular, recalling the definitions of $(Z_i,Z'_i)$ and $\Delta_i = Z'_i-Z_i$, we have that
\begin{align*}
 P_{k,m,m'}\left( \Delta_i - \Delta_{i-1} = z \mid (Z_{i-1},Z'_{i-1}) = (j,j') \right)
 =
 \begin{cases}
  P\left( \D^-_{\tau_{j'+1}^{\B^-}} - \D^-_{\tau_{j+1}^{\B^-}} = z \right) & \text{if } k+i \leq -1 \\
  P\left( \D^0_{\tau_{j'+1}^{\B^0}} - \D^0_{\tau_{j+1}^{\B^0}} = z \right) & \text{if } k+i =0 \\
  P\left( \D^+_{\tau_{j'}^{\B^+}} - \D^+_{\tau_{j}^{\B^+}} = z \right) & \text{if } k+i \geq 1.
 \end{cases}
\end{align*}
In all cases, the distribution of $\Delta_i - \Delta_{i-1}$ can be described in the form $\mathfrak{D}^*_{\tau_m^{\B^*}} - \mathfrak{D}^*_{\tau_\ell^{\B^*}}$ for some $* \in \{-,+,0\}$ and $m\geq \ell\geq 0$.
Thus, the following lemma will be useful in controlling the distribution of the process $(\Delta_i)_{i\geq 0}$ in the proof of Theorem \ref{thm:barLxh-JRK}.
Since the results below hold for all three of the urn processes we will omit the $*$ from the notation in the rest of this section.

\begin{lemma}\label{D-bounds}
For an urn model associated with the PSR case the following holds.
\begin{itemize}
\item [(i)] For every $\delta_1>0$ there exists $z_1\in\N$ such that \[m\ge \ell\ge z_1\ \ \Longrightarrow\ \ \left|E\left(\D_{\tau^{\B}_m}-\D_{\tau^{\B}_\ell}\right)\right|<\delta_1.\]
\item [(ii)] There exist $\delta_2,c_2\in(0,1)$ such that
\[
\delta_2\ell\ge m-\ell\ge 1  \ \Longrightarrow\ \ P\left( \D_{\tau^{\B}_m}-\D_{\tau^{\B}_\ell} \le - c_2 \sqrt{m-\ell} \right)>c_2.
\]
\end{itemize}
\end{lemma}

\begin{proof}
  Part (i) follows from \cite[Proposition 3.9]{KMMP26}, so we shall only
  prove (ii). Note that for $c_2\in (0,1)$ the probabilities in
  (ii) are positive for any $\ell < m$, and thus it is sufficient to
  prove (ii) only for all sufficiently large $\ell$. For
  definiteness, we consider urns to the right of $0$. The urns to the
  left of $0$ and at $0$ can be treated in the same way. To lighten
  the notation we shall drop the superscript $\B$.

  For $R,\delta>0$, $\ell\in\N$ let
  \[B_\ell(R,\delta)=\{|\D_j|\le R\sqrt{\ell},\ \forall\,
    j\in[\tau_\ell,\tau_{\fl{\ell(1+\delta)}}]\}\]
    and
    \[
     \delta_2' = \inf_{m\in\mathbb{N}}P\left(\sum_{k=1}^m(G_k-E[G_k]) < 0\right)\text{ with }G_1,G_2,...\text{ i.i.d. }\text{Geom}\left(\frac{1}{2}\right)\footnote{We write $G\sim \text{Geom}(p)$ if
    $P(G=j)=(1-p)^jp$, $j\in\Z_+$.}.
    \]
    Note that, by the Central Limit Theorem, $\delta_2'>0$. Let $\delta_2''=\frac{1}{4}
    \wedge\frac{1}{2}\delta_2'$. It follows from \cite[Lemma 3.1 and
    Lemma 3.5]{KMMP26} that there exist $R>0$ and $\delta_0\in(0,1)$
    such that for all $\delta\in(0,\delta_0)$ and all sufficiently
    large $\ell$, $P(B_\ell(R,\delta))> 1-\delta_2''$. On the event
    $B_\ell(R,\delta)$, we will show that for
    $1\le k\le\fl{\delta \ell}$, the difference
    $\D_{\tau_{\ell+k}}-\D_{\tau_{\ell+k-1}}$ can be stochastically
    bounded above by $G_{\ell,k}-1$, where $G_{\ell,k}$,
    $1\le k\le\fl{\delta \ell}$, are independent geometric random
    variables with parameters approaching 1/2 from below as
    $\ell \to \infty$. Indeed, first note that in our setting a
    ``failure'' on the $j$-th step corresponds to
    $\D_{j+1}-\D_j=1$. Therefore, to upper bound
    $\D_{\tau_{\ell+k}}-\D_{\tau_{\ell+k-1}}$, we need an upper bound
    on
    the probability that $\D_{j+1}-\D_j=1$ for $j \in [\tau_{\ell+k-1}, \tau_{\ell+k})$. To this end, letting $\mathcal{F}_j^{\mathfrak{B},\mathfrak{R}}=\sigma((\mathfrak{B}_i,\mathfrak{R}_i),i
  \leq j)$ we have that if $j \in [\tau_{\ell+k-1}, \tau_{\ell+k})$ then
\[
 P\left( \D_{j+1} - \D_j=1 \,\big\vert\, \mathcal{F}_j^{\B,\mathfrak{R}}  \right) = \frac{ w(2(\ell+k-1+\D_j))}{ w(2(\ell+k-1+\D_j)) + w(2(\ell+k-1)+1) } =: \frac{1}{1+a_j},
 \]
 where we denoted
 $\frac{w(2(\ell+k-1)+1)}{w(2(\ell+k-1+\D_j))}$ by
 $a_j$. First of all, observe that on the event $B_\ell(R,\delta)$,
 for $j\in[\tau_{\ell+k-1},\tau_{\ell+k})$,
 $1\le k\le \fl{\delta \ell}$, and all sufficiently large $\ell$,
 $\left|a_j-1\right|\le \frac{2\alpha R}{\sqrt{\ell}}\le 1$.
Then, under the same conditions, we have that
  \begin{align*}
    P( \D_{j+1} - \D_j = 1 \,\big\vert\, \mathcal{F}_j^{\B,\mathfrak{R}}   ) = \frac{1}{1+a_j}&=\frac12\left(1+\frac12\left(a_j-1\right)\right)^{-1}\\ &\le\frac12\left(1-\frac12\left|a_j-1\right|\right)^{-1}\le \frac12\left(1+\left|a_j-1\right|\right)\le\frac{1}{2} + \frac{\alpha R}{\sqrt{\ell}}.
  \end{align*}
Taking the complement probability as a ``success'' parameter and considering i.i.d.\ random variables $G_{\ell,k}\sim\text{Geom}\left(\frac12-\frac{\alpha R}{\sqrt{\ell}}\right)$, we get for $\d_2 \in (0,\d_0)$, $\ell < m \leq (1+\d_2)\ell$, and $\ell$ sufficiently large that
\begin{align}
 P\left( \D_{\tau_m}-\D_{\tau_\ell} > -c_2 \sqrt{m-\ell} \right)
 &\leq P(B_\ell(R,\d_2)^c) + P\left( \sum_{k=1}^{m-\ell} (G_{\ell,k} - 1) > -c_2 \sqrt{m-\ell} \right) \nonumber\\
 &\leq
 \delta_2'' + P\left( \sum_{k=1}^{m-\ell} ( G_{\ell,k} - E[G_{\ell,k}] ) > -c_2\sqrt{m-\ell} - \frac{4\alpha R(m-\ell)}{\sqrt{\ell}- 2\alpha R} \right) \nonumber\\
 &\leq \delta_2'' + P\left( \sum_{k=1}^{m-\ell} \frac{ G_{\ell,k} - E[G_{\ell,k}] }{\sqrt{m-\ell}}  > - c_2 - \frac{4\alpha R \sqrt{\d_2 \ell} }{\sqrt{\ell}-2 \label{appAlimthm}
   \alpha R} \right).
\end{align}
By the Berry-Esseen theorem, denoting by $\sigma_\ell^2$ the variance of $G_{\ell,1}$, there exists $C_2$ such that when $\ell$ is large enough, the last probability is smaller than
\[
P\left(\mathcal{N}(0,1) > - \frac{1}{\sigma_\ell}\left(c_2 + \frac{4\alpha R \sqrt{\d_2 \ell} }{\sqrt{\ell}-2
   \alpha R}\right) \right)+C_2\frac{E(|G_{\ell,1}-E(G_{\ell,1})|^3)}{\sigma_\ell^{3}\sqrt{m-\ell}}.
\]
Therefore, if $c_2$ and $\d_2$ are sufficiently small, then there
exist $z_2,\ell_0 \in \N$ such that for all $\ell \geq \ell_0$ and
$m-\ell \geq z_2$ the probability in \eqref{appAlimthm} is at most
$5/8$. On the other hand, for each fixed $1\leq m' < z_2$, if
$m=\ell+m'$ then
$\sum_{k=1}^{m-\ell} (G_{\ell,k} - E[G_{\ell,k}])
\underset{\ell\to\infty}{\Longrightarrow}
\sum_{k=1}^{m'}(G_k-E[G_k])$. Hence, when $c_2$ and $\d_2$ are
sufficiently small and $\ell$ is large, the probability in
\eqref{appAlimthm} is smaller than $1-\frac{2}{3}\delta_2'$ when
$m-\ell<z_2$. Hence, there are $c_2>0$ and $\delta_2 \in (0,1)$ such
that
$P\left( \D_{\tau_m}-\D_{\tau_\ell} \leq -c_2 \sqrt{m-\ell} \right)
\geq \min(\frac{1}{8},\frac{1}{6}\delta_2')>c_2$ for $\ell$
sufficiently large and $1\leq m-\ell \leq \d_2 \ell$.
\end{proof}

\section{Augmented JRKT: reductions and technical details}\label{sec:fJRK}

\subsection{Augmented JRKT from directional JRKT} The key assumption in
Theorems \ref{thm:main_X}--\ref{thm:main_JRK} is that the random walk
satisfies an augmented joint RKT.  For certain random walks (e.g.,
self-interacting random walks of the form \eqref{sirw}) it turns out
to be easier to prove directional augmented joint RKTs instead.  To
this end, let $\Lambda_{x,h}^{n,+}(\cdot)$ and
$\Lambda_{x,h}^{n,-}(\cdot)$ denote the rescaled local time profile
$\Lambda_{x,h}^n(\cdot)$ restricted to $[x,\infty)$ and $(-\infty,x]$,
respectively. We will call $\Lambda_{x,h}^{n,+}$ the ``forward''
curves and $\Lambda_{x,h}^{n,-}$ the ``backward'' curves. The topology
used on $D([x,\infty))$ is the one obtained by considering any
$f \in D([x,\infty))$ as $f(x+\cdot)\in D([0,\infty))$, while the
topology used on $D((-\infty,x])$ is the one obtained by considering
any $f \in D((-\infty,x])$ as the right-continuous version of
$f(x-\cdot)$, which is in $D([0,\infty))$.

\begin{definition}\label{def:JRK-rl}
  A recurrent random walk $(X_n)_{n\geq 0}$ satisfies a \textbf{forward augmented joint RKT} if there is a scaling function $g(n) = a_0 n^{\gamma}$ for some $a_0,\gamma>0$ such that for any $N \in \N$ and any finite set of points $(x_1,h_1),\ldots,(x_N,h_N) \in \R\times \R_+$ there are processes $\Lambda_{x_1,h_1}^+(\cdot),\ldots,\Lambda^+_{x_N,h_N}(\cdot)$ with $\Lambda_{x_i,h_i}^+(\cdot)$ defined on $[x_i,\infty)$ and such that the joint distribution of the forward curves $(\Lambda_{x_i,h_i}^{n,+}(\cdot))_{i\leq N}$ together with the right endpoints $(\mu_{x_i,h_i}^{n,+})_{i\leq N}$ and right merge points $(\mu_{(x_i,h_i),(x_j,h_j)}^{n,+})_{1\leq i<j\leq N}$ converges in the natural topology to the joint distribution of the curves $(\Lambda_{x_i,h_i}^+(\cdot))_{i\leq N}$ together with the right  endpoints $(\mu_{x_i,h_i}^+)_{i\leq N}$ and right merge points $(\mu_{(x_i,h_i),(x_j,h_j)}^+)_{1\leq i<j\leq N}$.
  
  Similarly, a random walk satisfies a \textbf{backward augmented joint RKT} if an analogous statement holds for the joint distribution of the backward curves $\Lambda_{x_i,h_i}^{n,-}(\cdot)$, the left endpoints $(\mu_{x_i,h_i}^{n,-})_{i\leq N}$, and left merge points $(\mu_{(x_i,h_i),(x_j,h_j)}^{n,-})_{1\leq i<j\leq N}$.
\end{definition}

The main result of this subsection is that 
an augmented joint Ray-Knight theorem can be deduced from an augmented Ray-Knight theorem (for a single curve) together with forward augmented and backward augmented joint Ray-Knight theorems.

\begin{proposition}\label{prop:fJRK-JRK} Suppose that a recurrent
  random walk satisfies an augmented RKT, a forward augmented joint
  RKT, and a backward augmented joint RKT with a common
  scaling function, and that the limiting curves
  in the augmented RKT satisfy Assumption \ref{asm:Bessel}. Then
  \begin{enumerate}[(i)]
  \item the random walk satisfies an augmented joint RKT with that
     scaling function;
\item the joint distribution of the curves in the joint RKT is
  uniquely determined by the joint distri\-bution of the forward curves.
\end{enumerate}
  \end{proposition}

  \begin{proof}
    The proof builds on that of \cite[Corollary 3.5]{KP25+}. As
explained there, it suffices to show that, along any subsequence
such that the joint distribution of the local time curves,
endpoints, and merge points converges, the limiting distribution of
the curves $(\Lambda_{x_i,h_i})_{i \leq N}$ is unique. Let
$m_1,\ldots,m_N \in \N$, $(y_{i,j})_{i \leq N, j \leq m_i}$, and
$(b_{i,j})_{i \leq N, j \leq m_i}$ be such that $b_{i,j}>0$ and
$b_{i,j} \neq h_i$ whenever $y_{i,j}=x_i$. Since, by Assumption
\ref{asm:Bessel}, the limiting curves in the augmented RKT are
continuous and $P(\Lambda_{x_i,h_i}(y_{i,j}) = b_{i,j})=0$ for all
$i,j$, along the chosen subsequence we have
\[
 P\left(\Lambda_{x_i,h_i}(y_{i,j}) \leq b_{i,j},\ \forall\, i\leq N,\ j \leq m_i\right)
 = \lim_{n \to \infty}P\left(\Lambda_{x_i,h_i}^n(y_{i,j}) \leq b_{i,j},\ \forall\, i\leq N,\ j \leq m_i\right).
\]
Moreover, the arguments of \cite{KP25+}\footnote{These arguments
    use the continuity of the distribution of $\mu_{x,h}^+$,
    which can be seen, e.g., from the formula following Remark 2.1 in
    \cite{HamanaMatsumoto2013}.}  yield an analog of
\cite[(3.7)]{KP25+}: the limit on the right-hand side is determined by
the limiting joint distribution of the forward curves. Since the
probabilities on the left-hand side determine the distribution of
$(\Lambda_{x_i,h_i})_{i\leq N}$, this proves (ii) as well as the
uniqueness needed for (i).
\end{proof}
\subsection{Comparing site and directed edge local time profiles for PSR walks}\label{subsec:tech}

The goal of this subsection is to prove
Lemma~\ref{lem:Lxh-barLxh}. Recall that the limiting curves
  $\Lambda_{x,h}$ are given by \eqref{WNrep}-\eqref{WNrep-back}.
The proof uses a ``hybrid'' family of curves, defined as follows:
\begin{equation}\label{hatLxh}
  \hat{\Lambda}_{x,h}^n(y) = 2 n^{-1} \mathcal{E}^+( \tau_{x,h}^n, \fl{yn} ), \quad y \in \R. 
 \end{equation}
 These curves represent the rescaled directed edge local times but the
 ``snapshot'' is taken at the site inverse local times $\tau_{x,h}^n$
rather than at directed edge inverse local times
 $\overline{\tau}_{x,h}^n:=\overline{\tau}_{\fl{xn},\fl{hn/2}}$ as
 for the curves $\overline{\Lambda}_{x,h}^n$ (see \eqref{barLxh}).
We shall denote by $\hat{\mu}_{x,h}^{n,\pm}$ the right/left
absorption points and by $\hat{\mu}_{(x,h),(x',h')}^{n,+}$ the right
merge points of these ``hybrid'' curves, with one adjustment: we set
$\hat{\mu}_{x,h}^{n,-} = \sup\{y < 0\wedge \fl{xn}/n :\,
\hat{\Lambda}_{x,h}^n(y)=0\}$, so that $\hat{\mu}_{x,h}^{n,-}$ is
determined by the part of the curve to the left of $x$ and does not
equal $x$ simply because $\mathcal{E}^+(\tau_{x,h}^n,\fl{xn})=0$.

Note that, by the Markovian structure of the directed edge local times (Appendix \ref{sec:Polya}), the only effect of using $\tau_{x,h}^n$ in place of $\overline{\tau}_{x,h}^n$ is in the values of the curves at $x$:
$\overline{\Lambda}^n_{x,h}(x) =
2n^{-1}\fl{hn/2}$ is deterministic,
while
$\hat{\Lambda}^n_{x,h}(x) = 2 n^{-1}
\mathcal{E}^+(\tau_{x,h}^n,\fl{xn})$ is random. However, since Lemma
\ref{lem:PSRbalance} implies that
 \begin{equation}\label{E+initial}
  \lim_{n\to\infty} n^{-1}\mathcal{E}^+(\tau_{x,h}^n, \fl{x n}) = h/2 \quad \text{in probability,}
 \end{equation}
 with high probability $\hat{\Lambda}^n_{x,h}(x)$ and
   $\overline{\Lambda}^n_{x,h}(x)$ are close to each other. The main
   part of the proof of Lemma~\ref{lem:Lxh-barLxh} will be to show analogs of
   \eqref{barLxh-comp}--\eqref{barmerge-comp} with
   $\hat{\Lambda}_{x,h}^n(y)$, $\hat{\mu}_{x,h}^{n,\pm}$,
   $\hat{\mu}_{(x,h),(x',h')}^{n,+}$ in place of $\Lambda_{x,h}^n(y)$,
   $\mu_{x,h}^{n,\pm}$, $\mu_{(x,h),(x',h')}^{n,+}$. Once this is
   done, the original statements
   \eqref{barLxh-comp}--\eqref{barmerge-comp} will follow readily. Indeed, by the definitions of the curves and \eqref{site-edge},
 \begin{align*}
  &\left| 
  \Lambda_{x,h}^n(y) - \frac12\left( \hat{\Lambda}_{x,h}^n(y) + \hat{\Lambda}^n_{x,h}\left(y - \frac{1}{n}\right)\right)
  \right| \leq \frac{1}{n}, \\
  &\left| \mu_{x,h}^{n,\pm} - \hat{\mu}_{x,h}^{n,\pm} \right| \leq \frac{1}{n},
  \quad \text{and}\quad
  \left| \mu_{(x,h),(x',h')}^{n,+} - \hat{\mu}_{(x,h),(x',h')}^{n,+} \right| \leq \frac{1}{n}.
 \end{align*}
 Moreover, by the augmented Ray-Knight theorem (\cite[Theorem
 1B]{tGRK}, see also \eqref{Yu-aRK}), we have
 $( \overline{\Lambda}^n_{x,h}(\cdot), \overline{\mu}_{x,h}^{n,-},
 \overline{\mu}_{x,h}^{n,+}) \Longrightarrow( \Lambda_{x,h}(\cdot),
 \mu_{x,h}^-, \mu_{x,h}^+)$, with $\Lambda_{x,h}(\cdot)$ almost surely
 continuous with compact support.  Hence, for any $\d>0$,
 \[
  \lim_{n\to \infty}P\left(\sup_{y \in \R} \left| \overline{\Lambda}^n_{x,h}\left(y-\frac{1}{n}\right) - \overline{\Lambda}^n_{x,h}(y) \right| > \d\right)=0.
 \]

Therefore, all we have to prove is that for any
$(x,h),(x',h') \in \R\times [0,\infty)$ and any $\delta>0$,
 \begin{align}
 \lim_{n\to \infty} P\left( \sup_{y \in \R} \left| \hat{\Lambda}_{x,h}^n(y) - \overline{\Lambda}^n_{x,h}(y) \right| > \d \right) &= 0,  \label{hatbarLxh} \\
  \lim_{n\to\infty} 
  P\left( |\hat{\mu}_{x,h}^{n,\pm} - \overline{\mu}_{x,h}^{n,\pm} | > \d \right) &= 0,  \label{muxhhatbar} \\
  \text{and} \quad\lim_{n\to\infty}
  P\left( |\hat{\mu}_{(x,h),(x',h')}^{n,+} - \overline{\mu}_{(x,h),(x',h')}^{n,+} | > \d \right) &= 0. \label{merge-hatbar}
 \end{align}
 We will need four lemmas, which we state and prove at the end of this subsection. 
 
To establish \eqref{hatbarLxh}--\eqref{merge-hatbar} we will use different comparisons depending on whether (1) $h>0$, (2) $h=0$ and $x>0$, or (3) $h=0$ and $x\leq 0$. 
\begin{enumerate}[\tiny{$\bullet$}]
 \item \textbf{Case $h>0$.} Since $\mathcal{E}^+(\overline{\tau}_{k,m},k) = m$ by \eqref{taue}, it follows that for any $\e \in (0,h)$, and $n$ sufficiently large,
 \begin{align}
  \left\{ 
  \left| \mathcal{E}^+(\tau_{x,h}^n, \fl{xn}) - \frac{hn}{2} \right| \leq \frac{\e n}{4} 
  \right\} 
  &\subset 
  \left\{ \overline{\tau}_{\fl{xn}, \fl{(h-\e)n/2}} < \tau_{x,h}^n < \overline{\tau}_{\fl{xn}, \fl{(h+\e)n/2}} \right\} \nonumber \\
  &\subset \left\{ \overline{\Lambda}^n_{x,h-\e}(\cdot) \leq \hat{\Lambda}^n_{x,h}(\cdot) \leq \overline{\Lambda}^n_{x,h+\e}(\cdot) \label{squeeze} \right\}. 
 \end{align}
 Note also that \eqref{E+initial} implies that these events occur with probability tending to 1 as $n\to\infty$. 
 \item \textbf{Case $h=0$ and $x>0$.} If $0\leq k' < k$, then $\overline{\tau}_{k',0} < \tau_{k,0} \leq \overline{\tau}_{k,0}$. Thus, for any $\e \in (0,x)$, and $n$ sufficiently large we have
 \begin{equation}\label{squeeze0}
  \overline{\Lambda}^n_{x-\e,0}(\cdot) \leq \hat{\Lambda}^n_{x,0}(\cdot) \leq \overline{\Lambda}^n_{x,0}(\cdot). 
 \end{equation}
 \item \textbf{Case $h=0$ and $x\leq 0$.} If $k\leq 0$ then the walk stays in $(-\infty,k]$ between
   times $\tau_{k,0}$ and $\overline{\tau}_{k,0}$, so that
   $\mathcal{E}^+(\overline{\tau}_{k,0},\ell) =
   \mathcal{E}^+(\tau_{k,0},\ell)$ for all $\ell \geq k$, and, since $\mathcal{E}^+(\tau_{k,0},k') = 0$ for
   $k'\leq k$, we have that for
     $x\leq 0$
 \begin{equation}\label{hzero-right}
  \hat{\Lambda}_{x,0}^n(y)
  = \overline{\Lambda}_{x,0}^n(y) \ind{y \geq \fl{xn}/n }.
 \end{equation}
\end{enumerate}

We now use these comparisons to prove \eqref{hatbarLxh}--\eqref{merge-hatbar}; throughout, $n$ is assumed to be sufficiently large.

 \noindent\textbf{Proof of \eqref{hatbarLxh}.}
 If $h>0$, then it follows from \eqref{E+initial}, \eqref{squeeze},  the augmented Ray-Knight theorem \eqref{Yu-aRK} and the joint Ray-Knight theorem given in \eqref{Yu-JRK} that for any $\e \in (0,h)$,
\begin{align*}
 \limsup_{n\to \infty} P\left( \sup_{y \in \R} \left| \hat{\Lambda}_{x,h}^n(y) - \overline{\Lambda}^n_{x,h}(y) \right| > \d \right)
 &\leq P\left( \sup_{y \in \R} \left( \Lambda_{x,h+\e}(y) -  \Lambda_{x,h-\e}(y) \right) \geq \d \right). 
\end{align*}
We claim that the last probability can be made arbitrarily small by taking $\e\to 0$.
To see this, assume further that $\e < \d/4 < h$. By Lemma \ref{lem:coalesce}(ii), for any $\d'>0$ we have
\begin{multline*}
 P\left( \sup_{y \in \R} \left( \Lambda_{x,h+\e}(y) -  \Lambda_{x,h-\e}(y) \right) \geq \d \right)
 \leq P\left( \mu^+_{(x,h-\e),(x,h+\e)} > x+\d' \right) \\
 \qquad + P\left( \mu^-_{(x,h-\e),(x,h+\e)} < x-\d' \right) + P\left( \sup_{y \in [x-\d',x+\d']} \left(  \Lambda_{x,h+\d/4}(y) - \Lambda_{x,h-\d/4}(y) \right) \geq \d \right). 
\end{multline*}
 Given $\e'>0$, since the paths $\Lambda_{x,h\pm \d/4}(\cdot)$ are continuous with a gap of $\d/2$ at $y=x$, we can first choose $\d'>0$ sufficiently small so that the last probability is less than $\e'/3$, and then by Lemma \ref{lem:mu-cont} we can choose $\e$ small enough so that the first two probabilities on the right are both bounded by $\e'/3$.

When $h=0$ and $x>0$ it follows from \eqref{squeeze0} and the augmented/joint Ray-Knight theorems in \eqref{Yu-aRK}-\eqref{Yu-JRK} that for any $\e \in (0,x)$,
\begin{align*}
 \limsup_{n\to\infty}  P\left( \sup_{y \in \R} \left| \hat{\Lambda}_{x,0}^n(y) - \overline{\Lambda}^n_{x,0}(y) \right| > \d \right)
 &\leq P\left( \sup_{y \leq x} \left( \Lambda_{x,0}(y) - \Lambda_{x-\e,0}(y) \right) \geq \d \right) \\
 &\leq P\left( \sup_{0 \leq s \leq \sigma_{x,\e,0} } \Lambda_{x,0}(x-s) \geq \d  \right),
\end{align*}
where $\sigma_{x,\e,0} = \inf\{s\geq \e:\, \Lambda_{x,0}(x-s)= 0\}$ and the last inequality holds by Lemma \ref{lem:coalesce}(ii): the two curves coincide to the left of $x-\sigma_{x,\e,0}$.
Since $(\Lambda_{x,0}(x-s))_{s\in [0,x]}$ is a multiple of a square of a standard Brownian motion, $\sigma_{x,\e,0} \to 0$ as $\e\to 0$, almost surely, and thus this last probability vanishes as $\e \to 0$.

Finally, when $h=0$ and $x\leq 0$, it follows from \eqref{hzero-right} and the augmented Ray-Knight theorem in \eqref{Yu-aRK} that
\begin{align*}
 \limsup_{n\to \infty} P\left( \sup_{y \in \R} \left| \hat{\Lambda}_{x,0}^n(y) - \overline{\Lambda}^n_{x,0}(y) \right| > \d \right)
 &\leq P\left( \sup_{y\leq x} \Lambda_{x,0}(y) \geq \d \right) = 0. 
\end{align*}

 \noindent\textbf{Proof of \eqref{muxhhatbar}.}
If $h>0$ and $\e \in (0,h)$, then on the event \eqref{squeeze}
\[
 |\hat{\mu}_{x,h}^{n,\pm} - \overline{\mu}_{x,h}^{n,\pm} | \leq |\overline{\mu}_{x,h+\e}^{n,\pm} - \overline{\mu}_{x,h-\e}^{n,\pm}|.
\]
Since, by \eqref{E+initial}, the probability of this event tends to $1$, we have
\begin{equation}\label{mxhcompare}
 \limsup_{n\to\infty} P\left( |\hat{\mu}_{x,h}^{n,\pm} - \overline{\mu}_{x,h}^{n,\pm} | > \d \right) \leq \limsup_{n\to\infty} P\left( |\overline{\mu}_{x,h+\e}^{n,\pm} - \overline{\mu}_{x,h-\e}^{n,\pm}| > \d \right). 
\end{equation}
By \eqref{Yu-aRK} and Lemma \ref{lem:mu-cont},
$\overline{\mu}^{n,\pm}_{x,h+\e}
\underset{n\to\infty}{\Longrightarrow} \mu^\pm_{x,h+\e}
\underset{\e\to 0}{\Longrightarrow} \mu^\pm_{x,h}$ and
$\overline{\mu}^{n,\pm}_{x,h-\e}
\underset{n\to\infty}{\Longrightarrow} \mu^\pm_{x,h-\e}
\underset{\e\to 0}{\Longrightarrow} \mu^\pm_{x,h}$, and the
right side of \eqref{mxhcompare} vanishes as $\e\to 0$ by Lemma
  \ref{lem:cip}.

If $h=0$ and $x>0$ then $\hat{\mu}^{n,+}_{x,0} = \overline{\mu}^{n,+}_{x,0} = x$, since both curves vanish at the site $\fl{xn}$.
For the left endpoints note that for any $\e \in (0,x)$
it follows from \eqref{squeeze0} that $\overline{\mu}_{x,0}^{n,-} \leq \hat{\mu}_{x,0}^{n,-} \leq \overline{\mu}_{x-\e,0}^{n,-}$, and therefore
\begin{equation}\label{mx0compare}
 \limsup_{n\to\infty} P\left( |\hat{\mu}_{x,0}^{n,-} - \overline{\mu}_{x,0}^{n,-} | > \d \right) \leq \limsup_{n\to\infty} P\left( \overline{\mu}_{x-\e,0}^{n,-} - \overline{\mu}_{x,0}^{n,-} > \d \right). 
\end{equation}
By \eqref{Yu-aRK} and Lemma \ref{lem:leftend},
$\overline\mu_{x,0}^{n,-} \underset{n\to\infty}{\Longrightarrow}
\mu_{x,0}^-$ and
$\overline\mu_{x-\e,0}^{n,-} \underset{n\to\infty}{\Longrightarrow}
\mu_{x-\e,0}^- \underset{\e \to 0}{\Longrightarrow} \mu_{x,0}^-$,
and the right side of \eqref{mx0compare} vanishes as $\e \to 0$
by Lemma \ref{lem:cip}.

If $h=0$ and $x\leq 0$ then \eqref{hzero-right} implies that $\hat{\mu}_{x,0}^{n,+} = \overline{\mu}_{x,0}^{n,+}$, since the right endpoints are determined by the curves on $[0,\infty) \subset [\fl{xn}/n,\infty)$. For the left endpoints, note that $\hat{\mu}_{x,0}^{n,-} = \fl{xn}/n$ and the augmented Ray-Knight theorem in \eqref{Yu-aRK} implies that $\overline{\mu}_{x,0}^{n,-} \Longrightarrow \mu_{x,0}^-$.
Since $\mu_{x,0}^-=x$, almost surely, it follows that $\lim_{n\to\infty}P\left( |\hat{\mu}_{x,0}^{n,-} - \overline{\mu}_{x,0}^{n,-} | > \d \right) = 0$.
 
 \noindent\textbf{Proof of \eqref{merge-hatbar}.}
  By Lemma \ref{lem:coalesce}(i), each pair of curves stays merged
  from its right merge point on. It follows that
 \begin{align*}
 &\left\{ \hat{\Lambda}^n_{x,h}(y) = \overline{\Lambda}^n_{x,h}(y), \, \forall\, y \geq x+\d \right\}
 \cap 
 \left\{ \hat{\Lambda}^n_{x',h'}(y) = \overline{\Lambda}^n_{x',h'}(y), \, \forall\, y \geq x' + \d \right\}
 \\
 &\qquad  \subset 
 \left\{  | \hat{\mu}_{(x,h),(x',h')}^{n,+} - \overline{\mu}_{(x,h),(x',h')}^{n,+} | \leq \d \right\}, 
\end{align*}
and \eqref{merge-hatbar} will follow once we show that for
any $(x,h) \in \R \times [0,\infty)$,
\begin{equation}\label{Lxhcompare-right}
 \lim_{n\to\infty} P\left( \hat{\Lambda}^n_{x,h}(y) = \overline{\Lambda}^n_{x,h}(y), \, \forall\, y \geq x+\d \right) = 1, \quad \forall\, \d>0.
\end{equation}
\eqref{Lxhcompare-right} is trivial when $h=0$ because $\hat{\Lambda}^n_{x,0}(y) = \overline{\Lambda}^n_{x,0}(y)$ for all $y \geq \fl{xn}/n$ (for $x>0$ this follows as in \eqref{hzero-right}, since the walk stays in $(-\infty,k]$ between $\tau_{k,0}$ and $\overline{\tau}_{k,0}$).
On the other hand, if $h>0$ and $\e \in (0,h)$, then it follows from \eqref{squeeze} and Lemma \ref{lem:coalesce}(i) that for $n$ sufficiently large,
 \begin{align}
  &\left\{ 
  \left| \mathcal{E}^+(\tau_{x,h}^n, \fl{xn}) - \frac{hn}{2} \right| \leq \frac{\e n}{4} 
  \right\} \cap \left\{  \overline{\mu}^{n,+}_{(x,h-\e),(x,h+\e)} \leq x +\d \right\}
  \subset 
  \left\{ \hat{\Lambda}^n_{x,h}(y) = \overline{\Lambda}^n_{x,h}(y), \, \forall\, y \geq x+\d \right\}. \label{squeeze2}
 \end{align}
 Thus, \eqref{E+initial} and the forward augmented joint Ray-Knight
 theorem in \eqref{Yu-fJRK} imply that
 \begin{align*}
  \limsup_{n\to\infty}
  P\left( \hat{\Lambda}^n_{x,h}(y) \neq \overline{\Lambda}_{x,h}^n(y) \text{ for some } y \geq x+\d \right) 
  &\leq \limsup_{n\to\infty} P\left( \overline{\mu}^{n,+}_{(x,h-\e),(x,h+\e)} > x +\d \right) \\
  &\leq P\left( \mu_{(x,h-\e),(x,h+\e)}^+ \geq x+\d \right).
 \end{align*}
 Finally, this upper bound vanishes by taking $\e\to 0$ due to Lemma \ref{lem:mu-cont}. This completes the proof of Lemma~\ref{lem:Lxh-barLxh}
 modulo the four lemmas used above, which we now state and prove.

\begin{lemma}\label{lem:coalesce}
 \begin{enumerate}[(i)]
\item Let $0\leq n_1 \leq n_2$ be integers and $k \in \Z$ be such
  that $X_{n_1} \leq k$ and
  $\mathcal{E}^+(n_1,k) = \mathcal{E}^+(n_2,k)$. Then
  $\mathcal{E}^+(n_1,\ell) = \mathcal{E}^+(n_2,\ell)$ for all
  $\ell \geq k$. In particular, almost surely, for any
  $(x,h),(x',h') \in \R\times[0,\infty)$,
  \[
   \overline{\Lambda}^n_{x,h}(y) = \overline{\Lambda}^n_{x',h'}(y)
   \quad \text{for all } y \geq \overline{\mu}^{n,+}_{(x,h),(x',h')},
  \]
  and a similar statement holds for the curves $\hat{\Lambda}^n$ defined in
  \eqref{hatLxh}.
  \item Almost surely, the curves in \eqref{WNrep}-\eqref{WNrep-back}
    satisfy the non-crossing property \eqref{non-cross} (in
    particular, $\Lambda_{x,h'} \leq \Lambda_{x,h''}$ for
    $0\leq h' \leq h''$) and coalesce: for any
    $(x,h),(x',h') \in \R\times[0,\infty)$, 
    $\Lambda_{x,h}(y) = \Lambda_{x',h'}(y)$ for all
    $y \geq \mu^+_{(x,h),(x',h')}$ and all
    $y \leq \mu^-_{(x,h),(x',h')}$.
 \end{enumerate}
\end{lemma}
\begin{proof}
  (i) Since $\mathcal{E}^+(\cdot,k)$ is non-decreasing, the assumption
 implies that the walk makes no step from $k$ to $k+1$ at any time
 $j \in [n_1,n_2)$. Since $X_{n_1} \leq k$, it follows that
 $X_j \leq k$ for all $j \in [n_1,n_2]$ and, therefore, the walk
 makes no step from $\ell$ to $\ell+1$ with $\ell \geq k$ at any time
 $j \in [n_1,n_2)$. This proves the first claim. For the second claim
 we may assume that $\overline{\mu}^{n,+}_{(x,h),(x',h')} < \infty$.
 Let $n_1 = \overline{\tau}^n_{x,h} \wedge \overline{\tau}^n_{x',h'}$
 and $n_2 = \overline{\tau}^n_{x,h} \vee \overline{\tau}^n_{x',h'}$.
 By \eqref{taue}, $X_{n_1} \in \{\fl{xn},\fl{x'n}\}$, and the two directed edge local time profiles
 agree at some site $k \geq \fl{xn}\vee\fl{x'n} \geq X_{n_1}$, so the
 first claim applies. The statement for the curves $\hat{\Lambda}^n$
 follows in the same way, since $X_{\tau^n_{x,h}} = \fl{xn}$.

 (ii) For the non-crossing property, we notice that since
 $\mathcal{E}^+(\cdot\,,\ell)$ is non-decreasing for every $\ell$, for
 each $n$ the curves $\overline{\Lambda}^n_{x,h}$ and
 $\overline{\Lambda}^n_{x',h'}$ are ordered pointwise (in the same way
 as $\overline{\tau}^n_{x,h}$ and $\overline{\tau}^n_{x',h'}$), and,
 in particular, each pair of discrete curves satisfies
 \eqref{non-cross}. Since the limiting curves are continuous,
 \eqref{non-cross} is preserved under the joint convergence
 \eqref{Yu-JRK}. For the coalescence, let
   $\mu:=\mu^+_{(x,h),(x',h')}$, which is a.s.\ finite and is a
   stopping time for the filtration
   $(\overline{\mathcal{G}}_u)_{u\in\R}$ of \eqref{def_G}. We saw in
   Section~\ref{sec:PSRasm} that for $x_0\ge x\vee x'$,
 conditionally on $\overline{\mathcal{G}}_{x_0}$, the pair
 $(\Lambda_{x,h}(x_0+\cdot),\Lambda_{x',h'}(x_0+\cdot))$ has law
   $K_{x_0}\big((\Lambda_{x,h}(x_0),\Lambda_{x',h'}(x_0)),\cdot\big)$. Applying
   this at the stopping times
   $\mu_n:=2^{-n}(\lfloor 2^n\mu\rfloor+1)\downarrow\mu$, letting
   $n\to\infty$, and using continuity of the curves and
 Proposition~\ref{prop:fdd} (which holds since \eqref{Yu-JRK} holds),
 we conclude that, conditionally on $\overline{\mathcal{G}}_\mu$, the
   finite-dimensional distributions of
   $(\Lambda_{x,h}(\mu+\cdot),\Lambda_{x',h'}(\mu+\cdot))$ are those
   of $K_\mu((a,a),\cdot)$ with
   $a=\Lambda_{x,h}(\mu)=\Lambda_{x',h'}(\mu)$. Since, for every $x_0$ and $a$,
   $K_{x_0}((a,a),\cdot)$ is the law of a pair of identical curves,
   the two curves coincide on $[\mu,\infty)$ almost surely. The
 equality for $y\le\mu^-_{(x,h),(x',h')}$ follows since, by
   \eqref{WNrep-back}, the backward curves are forward curves driven
   by $-\mathcal{W}^*$.
\end{proof}

\begin{remark}
Non-crossing and coalescence properties of the two-sided PSR flow were also
obtained, by different methods, in \cite[Proposition~6.7]{Y26flows} and
\cite[Lemma~3.12]{Y25+}, respectively.
\end{remark}

\begin{lemma}\label{lem:cip}
 Let $\{\xi_{\e,n}\}_{n\geq 1, \e>0}$,  $\{\zeta_{\e,n}\}_{n\geq 1, \e>0}$, $\{\xi_\e\}_{\e>0}$, $\{\zeta_\e\}_{\e>0}$, and $\chi$ be random variables with the following properties:
 \begin{enumerate}[(i)]
  \item $\xi_{\e,n} \leq \zeta_{\e,n}$, almost surely, for all $\e>0$ and $n\geq 1$; 
  \item for any $\e>0$, $\xi_{\e,n} \Longrightarrow \xi_\e$ and $\zeta_{\e,n} \Longrightarrow \zeta_\e$ as $n\to \infty$; 
  \item $\xi_\e \Longrightarrow \chi$ and $\zeta_\e \Longrightarrow \chi$ as $\e \to 0$. 
 \end{enumerate}
Then 
\[
 \lim_{\e \to 0} \limsup_{n\to \infty} P\left( \zeta_{\e,n} - \xi_{\e,n} \geq \delta \right) = 0, \quad \forall\, \d>0.
\]
\end{lemma}

\begin{proof}
Let $f(\e) = \limsup_{n\to\infty} P(\zeta_{\e,n}-\xi_{\e,n} \geq \d)$, and let $\e_k \searrow 0$ be a sequence such that $\lim_{k\to\infty} f(\e_k) = \limsup_{\e\to 0} f(\e)$. 
 For each $k\geq 1$, take a subsequence $n_j\to\infty$ so that $f(\e_k) = \lim_{j\to\infty} P(\zeta_{\e_k, n_j} - \xi_{\e_k,n_j} \geq \d)$ and then take a further subsequence so that we have joint convergence in distribution of $(\xi_{\e_k,n_j}, \zeta_{\e_k,n_j})$ to some pair $(\tilde{\xi}_{\e_k}, \tilde{\zeta}_{\e_k})$. Note that our assumptions imply that $\tilde{\xi}_{\e_k} \overset{\text{Law}}{=} \xi_{\e_k}$, $\tilde{\zeta}_{\e_k} \overset{\text{Law}}{=} \zeta_{\e_k}$, and $\tilde{\xi}_{\e_k} \leq \tilde{\zeta}_{\e_k}$, almost surely. Then, it follows from the Portmanteau Theorem that $f(\e_k) \leq P(\tilde{\zeta}_{\e_k} - \tilde{\xi}_{\e_k} \geq \d )$. We can then take a further subsequence $\e_{k_\ell} \to 0$ so that we have joint convergence in distribution of $(\tilde{\xi}_{\e_{k_\ell}}, \tilde{\zeta}_{\e_{k_\ell}})$ to some pair $(\tilde{\xi},\tilde{\zeta})$. However, since our assumptions imply that we must have $\tilde{\xi} \overset{\text{Law}}{=} \chi$, $\tilde{\zeta} \overset{\text{Law}}{=} \chi$, and $\tilde{\xi} \leq \tilde{\zeta}$, almost surely, it follows that $\tilde{\xi} = \tilde{\zeta}$, almost surely so that in fact $\tilde{\zeta}_{\e_{k_\ell}} - \tilde{\xi}_{\e_{k_\ell}}$ converges to zero in probability, and thus
 \[
  \limsup_{\e\to 0} f(\e)
  = \lim_{k\to\infty} f(\e_k)
  = \lim_{\ell\to\infty} f(\e_{k_\ell}) 
  \leq \lim_{\ell \to \infty} P\left( \tilde{\zeta}_{\e_{k_\ell}} - \tilde{\xi}_{\e_{k_\ell}} \geq \d \right) = 0. 
 \]
\end{proof}

\begin{lemma}\label{lem:mu-cont}
 For the family of curves defined  in \eqref{WNrep}-\eqref{WNrep-back}, 
 for any $x \in \R$ and $h>0$ the following limits hold in probability 
 \[\lim_{\e\to 0} \mu_{(x,h-\e),(x,h+\e)}^\pm = x, 
  \quad\lim_{\e\to 0} \mu_{x,h+\e}^\pm = \mu_{x,h}^\pm, 
  \quad\lim_{\e\to 0} \mu_{x,h-\e}^\pm = \mu_{x,h}^\pm. 
\]
\end{lemma}

\begin{proof}
 For the first claim, note that except for the boundary condition, the evolution of the pair $(\Lambda_{x,h-\e},\Lambda_{x,h+\e})$ does not depend on the spatial location, and reflecting the lower curve at $0$ can only make the curves merge sooner. Hence the merging distance $|\mu^{\pm}_{(x,h-\e),(x,h+\e)}-x|$ is stochastically largest for $x=0$, and for any $\e \in (0,h)$ we have
 \[
  P\left( |\mu_{(x,h-\e),(x,h+\e)}^{\pm} - x | > \d \right) 
  \leq  P\left( |\mu_{(0,h-\e),(0,h+\e)}^{\pm}| > \d \right)
  = P\left( \Lambda_{0,h+\e}(\d) > \Lambda_{0,h-\e}(\d) \right), 
 \]
where the last equality holds by Lemma \ref{lem:coalesce}(ii) together with the fact that, by \eqref{WNrep}-\eqref{WNrep-back}, the pairs $(\Lambda_{0,h-\e}(\cdot),\Lambda_{0,h+\e}(\cdot))$ and $(\Lambda_{0,h-\e}(-\ \cdot),\Lambda_{0,h+\e}(-\ \cdot))$ have the same law.
Then, it follows from Lemma \ref{c-base} and the rescaling property \eqref{Lxh-rescale} that this last probability vanishes as $\e\to 0$.
 
For the second and third claims in the lemma, note that, by Lemma
  \ref{lem:coalesce}(ii), $\mu_{x,h}^{\pm}$ is always between
$\mu_{x,h-\e}^{\pm}$ and $\mu_{x,h+\e}^{\pm}$, and
 \[
  \left\{ \left| \mu^\pm_{(x,h-\e),(x,h+\e)} - x \right| \leq \d \right\}
  \subset
  \left\{ \left| \mu^{\pm}_{x,h+\e} - \mu^{\pm}_{x,h-\e} \right| \leq \d \right\}.
 \]
 Hence, the second and third claims follow from the first.
\end{proof}

\begin{lemma}\label{lem:leftend}
  For the family of curves defined  in \eqref{WNrep}-\eqref{WNrep-back}, if $x>0$, then 
 \[
  \lim_{\e\to 0} P\left( \mu_{x-\e,0}^- = \mu_{x,0}^- \right)  = 1. 
 \]
\end{lemma}
\begin{proof}
 For $\e \in (0,x)$, $\Lambda_{x,0}$ and $\Lambda_{x-\e,0}$ satisfy Assumption \ref{asm:Bessel}, hence $\Lambda_{x,0}(x-\e)>0$ almost surely and $\Lambda_{x-\e,0}(x-\e)=0$, thus by Lemma \ref{lem:coalesce}(ii), we have $0\leq \Lambda_{x-\e,0} \leq \Lambda_{x,0}$ almost surely, and the two curves coincide to the left of any common zero in $(-\infty,x-\e]$. Hence,
 \[
  P\left( \mu_{x-\e,0}^- = \mu_{x,0}^- \right)
  \geq P\left( \Lambda_{x,0}(x-s) = 0, \, \text{ for some } \e \leq s \leq x \right)
  = P(B(s) = 0 \text{ for some } s \in [\e,x]), 
 \]
where $B(\cdot)$ is a standard Brownian motion started at $B(0) = 0$. Taking $\e\to 0$, the lower bound converges to $P(B(s) = 0 \text{ for some } s \in (0,x]) = 1$.
\end{proof}

\section{Heuristic derivation of the driving mechanism}\label{psr_sde}

Below we give a non-rigorous derivation of the driving mechanism
  for the PSR motion in the bulk, i.e., away from the boundary of its
  range. This derivation is in the spirit of
  \cite[pp.\,449--451]{TW98}, where it was carried out for the true
  self-repelling motion assuming that it is a scaling limit of
  TSAW. Heuristics for the generalized TSAW, i.e., when
  $w(j)=e^{-\beta j^{\kappa}},\ \kappa\in(0,1)$, are similar, and we
  shall only give the resulting ``equation'' for the limiting process,
  generalized true self-repelling motion, construction of which still
  remains an open problem.

  Writing the driving mechanisms for these three self-repelling models
  side by side leads one to (speculatively) bridge the three cases by
  letting $\kappa \to 1$ or letting $\kappa\to 0$ and
  $\beta\to \infty$ in such a way that $\beta \kappa\to \alpha$. These
  links and the fact that the above limits, in fact, cannot be
  switched with the limit as $n\to\infty$ have been already pointed
  out and briefly discussed in \cite[Remarks,
  pp.\,24,\ 29]{tTSAWGBR}. Nevertheless, we find the picture
  compelling enough to include here.

  Let $(X_n)_{n\ge 0}$ be a self-interacting random walk with the
  weight function $w$. Let
  $\ell_n(k)={\cal E}^+(n,k)+{\cal E}^-(n,k+1)$ be the number of 
  crossings of the undirected bond $\{k,k+1\}$ up to time $n$ and
  ${\cal F}_n$, $n\ge 0$, be the natural filtration of the walk. 
  Then
  we can easily compute
\begin{align*}
  E\left(X_{n+1}-X_n\mid {\cal F}_n\right)&=\frac{w(\ell_n(X_n))-w(\ell_n(X_n-1))}{w(\ell_n(X_n))+w(\ell_n(X_n-1))};\\
  \Var\left(X_{n+1}-X_n\mid {\cal F}_n\right)&=\frac{4w(\ell_n(X_n))w(\ell_n(X_n-1))}{\left(w(\ell_n(X_n))+w(\ell_n(X_n-1))\right)^2},
\end{align*}
and conclude that the process \[M_0=0,\ \ M_n:=X_n-\sum_{j=0}^{n-1}\frac{w(\ell_j(X_j))-w(\ell_j(X_j-1))}{w(\ell_j(X_j))+w(\ell_j(X_j-1))},\quad n\in\N,\] is a martingale with quadratic variation \[\langle M\rangle_0=0,\ \ \langle M\rangle_n=\sum_{j=0}^{n-1}\frac{4w(\ell_j(X_j))w(\ell_j(X_j-1))}{\left(w(\ell_j(X_j))+w(\ell_j(X_j-1))\right)^2},\quad n\in\N.\]

We shall now focus on the case of polynomially self-repelling walks,
and for simplicity consider the specific weight function
$w(n)=(n+1)^{-\alpha}$ with $\alpha>0$.  If we let
$x^{(n)}(t) = X_{\fl{nt}}/\sqrt{n}$ and
$\ell^{(n)}(t,x) = \ell_{\fl{tn}}(\fl{x\sqrt{n}})/\sqrt{n}$ for
$t\ge 0,\ x\in\R$, then, just as in Theorem \ref{main_psr}, we
should have the joint convergence
$(x^{(n)}(t), \ell^{(n)}(t,\cdot) )_{t\geq 0} \Longrightarrow
(\mathfrak{X}_\alpha(t), \mathfrak{L}_\alpha(t,\cdot) )_{t\geq 0}$.
We will use this to derive a heuristic equation for the dynamics of
$\mathfrak{X}_\alpha$.

Using the basic Taylor approximation
\[\frac{(h+\Delta h)^\alpha-h^\alpha}{(h+\Delta
    h)^\alpha+h^\alpha}=\frac{(1+\Delta h/h)^\alpha-1}{(1+\Delta
    h/h)^\alpha+1}= \frac{\alpha}{2}\Delta h/h+O\left(\left(\Delta
      h/h\right)^2\right),\quad\text{as }\Delta h/h\to 0,\] we get
that in the bulk, i.e., on a subinterval of the range around the
origin where the local time is sufficiently large so that we can drop
the error term as $n\to\infty$,
\begin{align}
  \frac{M_{\fl{nt}}}{\sqrt{n}}&=\frac{X_{\fl{nt}}}{\sqrt{n}}-\frac1{\sqrt{n}}\sum_{j=0}^{\fl{nt}-1}\frac{(\ell_{j}(X_j-1)+1)^\alpha-(\ell_{j}(X_j)+1)^\alpha}{(\ell_{j}(X_j-1)+1)^\alpha+(\ell_{j}(X_j)+1)^\alpha} \nonumber \\ &= \frac{X_{\fl{nt}}}{\sqrt{n}}-\frac{\alpha}{2\sqrt{n}}\sum_{j=0}^{\fl{nt}-1}\left(\frac{\ell_{j}(X_j-1)-\ell_{j}(X_j)}{\ell_{j}(X_j)+1}+O\left(\frac{(\ell_{j}(X_j-1)-\ell_{j}(X_j))^2}{(\ell_{j}(X_j)+1)^2}\right)\right) \nonumber \\ &\approx x^{(n)}(t)+\frac{\alpha}{2}\sum_{j=0}^{\fl{nt}-1}\frac{\left(\ell^{(n)}(\frac{j}{n}, x^{(n)}(\frac{j}{n}))-\ell^{(n)}(\frac{j}{n},x^{(n)}(\frac{j}{n})-\frac1{\sqrt{n}})\right)\sqrt{n}}{\ell^{(n)}(\frac{j}{n}, x^{(n)}(\frac{j}{n}))}\,\frac1{n} \label{drift-approx} \\
   &= x^{(n)}(t) + \frac{\alpha}{2} \int_0^{\fl{nt}/n}  \frac{ \left( \ell^{(n)}(s, x^{(n)}(s)) - \ell^{(n)}(s, x^{(n)}(s)-\frac{1}{\sqrt{n}} ) \right)\sqrt{n} }{ \ell^{(n)}(s, x^{(n)}(s) ) } \, ds . \nonumber
\end{align}
For the quadratic variation, using the fact that
\[\frac{4h^\alpha(h+\Delta h)^\alpha}{(h^\alpha+(h+\Delta
    h)^\alpha)^2}=\frac{4(1+\Delta h/h)^\alpha}{(1+(1+\Delta
    h/h)^\alpha)^2}=1-\frac{\alpha^2}{4}(\Delta h/h)^2+O\left((\Delta
    h/h)^3\right),\] we get from similar computations that
$\frac{\langle M\rangle_{\fl{nt}}}{n}\to t$ as $n\to\infty$. Hence,
$\left(n^{-1/2}M_{\fl{nt}}\right)_{t\ge 0}\Longrightarrow (B(t))_{t\ge
  0}$, where $(B(t))_{t\ge 0}$ is a standard Brownian motion. Putting
these together and recalling that
$(x^{(n)}(t),\ell^{(n)}(t,\cdot))_{t\ge 0}\Longrightarrow
(\mathfrak{X}_{\alpha}(t),\mathfrak{L}_{\alpha}(t,\cdot))_{t\ge 0}$ as
$n\to\infty$, we arrive at the following ``equation'':
\[
 \text{``} \mathfrak{X}_{\alpha}(t) 
 = -\frac{\alpha}{2} \int_0^t \left[(\mathfrak{L}_{\alpha}(s,x))^{-1} \frac{ \partial\mathfrak{L}_{\alpha}(s,x) }{ \partial x} \right]_{x=\mathfrak{X}_\alpha(s)} \, ds + B(t) +  (\text{boundary effects}) \text{''}.
\]
Note that we need to add the ``boundary effects'' since the approximation of the drift in
\eqref{drift-approx} is good only when the local
time at the present location is sufficiently large; thus the equation describes the driving mechanism when the process is
away from its running minimum or maximum.  The quotes around the
equation indicate that it should not be taken at face value, as the
local time process is certainly not differentiable in $x$. But for
$\alpha=1$ it gives a link to the rigorous work \cite{LST21}: compare
the above equation with \cite[(1.2)]{LST21}.

For the generalized TSAW, the heuristic derivation is similar but
  the scaling is different.  We set
$\ell^{(n)}(t,x)=
n^{-\frac{1}{\kappa+2}}\ell_{\fl{tn}}(\fl{xn^{\frac{\kappa+1}{\kappa+2}}})$
and $x^{(n)}(t)=n^{-\frac{\kappa+1}{\kappa+2}}X_{\fl{nt}}$. If we assume
that there is a joint limit
$(x^{(n)}(t), \ell^{(n)}(t,\cdot) )_{t\geq 0} \Longrightarrow
(\mathfrak{X}(t), \mathfrak{L}(t,\cdot))_{t\geq 0}$ then computations similar to the above give us the following
``equation'' in the bulk:
\[
 \text{``} \mathfrak{X}(t) 
 = -\frac{\beta\kappa}{2} \int_0^t\left[(\mathfrak{L}(s,x))^{\kappa-1}\frac{\partial\mathfrak{L}(s,x)}{\partial x} \right]_{x=\mathfrak{X}(s)} \, ds  +  (\text{boundary effects}) \text{''}.
\] 
Note that in contrast to the PSR walk, but similar to the driving mechanism for
  the TSAW in \cite{TW98}, the martingale term vanishes, since
$n^{-\frac{2(\kappa+1)}{\kappa+2}}\langle M\rangle_{\fl{nt}}\sim
tn^{-\frac{\kappa}{\kappa+2}}\to 0$ as $n\to\infty$.

It is tempting to write a common ``equation'' for these models for all
$\beta\in(0,\infty)$ and $\kappa\in(0,1]$ which also includes
$(\kappa,\beta)=(0,\infty)$ by allowing $\kappa\to 0$ and
$\beta\to\infty$ in such a way that $\beta \kappa\to\alpha\ge 0$:
\[ \text{``} \mathfrak{X}(t) = -\frac{\beta\kappa}{2}
    \int_0^t\left[(\mathfrak{L}(s,x))^{\kappa-1}\frac{\partial\mathfrak{L}(s,x)}{\partial
        x} \right]_{x=\mathfrak{X}(s)} \, ds + B(t)
    \mathbbm{1}_{\{\kappa=0\}} + (\text{boundary effects}) \text{''}.
\]
The case $(\kappa,\beta)=(0,\infty)$ corresponds to the PSR case
  by setting $\beta \kappa=\alpha$ in the equation if $\alpha>0$ and
  to the asymptotically free case by setting $\beta \kappa=0$. The
  latter is rigorously understood, \cite{KMP23}: the limiting process
  is a BMPE \eqref{BMPE-def} which behaves as a
  Brownian motion away from its extrema.  We point out one more time
that these derivations are mere heuristics, so if and when rigorous
results become available, the constants in front of the terms in the
right-hand side might be different, \cite[Remarks, p.\,24]{tTSAWGBR},
\cite[p.\,451]{TW98}.

\begin{remark}
  Note that in \cite[(B.11)]{TW98}, which treated the TSAW, the
  coefficient in front of the ``drift'' term is $g$ ($g=\beta$,
  $\kappa=1$) while we get a half of it. This is explained by the fact
  that for the TSAW model, see \cite[(1.9)]{TW98}, there should be
  $g/2$ in place of $g$ in the equation (B.4); with this correction,
  the result of the formal derivation there matches the quoted one
  above.
\end{remark}

\section{Technical results on {\rm BESQ} processes}\label{app:Bessel}

\subsection{Tails of BESQ processes}
The main goal of this subsection is to prove tail
estimates for the supremum and infimum of (powers of) \rm{BESQ}
processes that are needed to establish
  Corollary~\ref{cor:Lxh-dev} in the proof of tightness of
  local time processes (Section \ref{tightness}).
Some of these estimates  can be derived from
known and more precise results on hitting times of Bessel
processes, \cite{ryCMBM}, \cite{gySGBP}, \cite{BL23}.
However, those
results do not cover all the cases we need, in particular lower
  bounds for processes of negative dimension. We
present here a simple argument that covers all cases and gives bounds that are easy to work with and
strong enough for our purposes.

\begin{proposition}\label{prop:BESQdev}
 If $(Y(t))_{t \geq 0}$ is a \rm{BESQ}$^{\dim}$ process which is reflected at $0$ if $\dim >0$ and absorbed at $0$ if $\dim \leq 0$, then there are constants $C,c>0$ depending only on the dimension $\dim \in \R$ such that
 \begin{align}\label{rightexit}
  &P\left( \sup_{t\leq T} Y(t) \geq v \mid Y(0) = u\right) \leq C e^{-c \frac{(v-u)^2}{v T}}, \quad \forall\, 0\leq u < v, \, T>0,\quad\text
 {and} \\
\label{leftexit}
  &P\left( \inf_{t \leq T} Y(t) \leq v \mid Y(0) = u \right)
  \leq C e^{-c \frac{(u-v)^2}{uT}}, \quad \forall\, 0\leq v < u, \, T>0.
 \end{align}
\end{proposition}

\begin{remark}
If $\dim > 0$ and  $Y\in\rm{BESQ}^\dim$ is absorbed at $0$
then the estimates in Proposition \ref{prop:BESQdev} still hold. Indeed, the left tail bounds remain unchanged since the event $\{\inf_{t\leq T} Y(t) \leq v\}$ depends only on the behavior prior to hitting zero, and the probability of the event $\{\sup_{t\leq T} Y(t) \geq v\}$ only decreases if $Y$ is absorbed at zero. 
\end{remark}

\begin{proof}
 We will divide the proof of both the right and left exit probabilities into several cases depending on $\dim$. We begin by proving the right exit probabilities in \eqref{rightexit}. 
 
 \textbf{Case 1:} $\dim =1$. In this case, since $Y(t) = (B(t))^2$, where $B$ is a standard Brownian motion,
 we have\footnote{Of course, we can get better bounds
 using the reflection principle for Brownian motion, but these bounds are good enough for our purposes.}
\begin{align*}
 P\left( \sup_{t\leq T} Y(t) \geq v \mid Y(0) = u \right) 
 &= P\left( \sup_{t\leq T} |B(t)| \geq \sqrt{v} \mid B(0) = \sqrt{u} \right) \\
 &\leq 2 P\left( \sup_{t\leq T} B(t) \geq \sqrt{v} \mid B(0) = \sqrt{u} \right) \\
 &= 2 P\left( \sup_{t\leq T} B(t) \geq \sqrt{v}-\sqrt{u} \mid B(0) = 0 \right) \leq 2 e^{-\frac{(\sqrt{v}-\sqrt{u})^2}{2T}} \leq 2 e^{-\frac{(v-u)^2}{8vT}}.
\end{align*}

\textbf{Case 2:} $\dim <1$.  By the comparison
principle, \cite[Chapter~VI, Theorem~1.1]{IW89}, there is a coupling of $Y\in\rm{BESQ}^\dim$ with
$Z\in\rm{BESQ}^{1}$ 
with $Y(0) = Z(0)$ and such that
$Y(t) \leq Z(t)$ for all $t\geq 0$. From this, together with Case 1 we
have that
\[
 P\left( \sup_{t\leq T} Y(t) \geq v \mid Y(0)=u \right)
 \leq P\left( \sup_{t\leq T} Z(t) \geq v \mid Z(0)=u \right) \leq 2 \exp\left\{ - \frac{(v-u)^2}{8vT} \right\}.
\]

\textbf{Case 3:} $\dim > 1$. 
Suppose that $\dim \in (n-1,n]$ for some integer $n\geq 2$ (that is, $n=\lceil \dim \rceil$) and let $\dim' = \dim - (n-1) \in (0,1]$. Then we can represent 
$Y(t) = \sum_{k=0}^{n-1} Y^{(k)}(t)$ where the $(Y^{(k)})_{0\leq k\leq n-1}$ are all independent, $Y^{(0)}\in\rm{BESQ}^{\dim'}$, and  $Y^{(k)}\in\rm{BESQ}^{1}$ for $1\leq k \leq n-1$. 
Then, using the bounds from Cases 1 and 2 we get that 
\begin{align*}
 P\left( \sup_{t\leq T} Y(t) \geq v \mid Y(0)=u \right)
 &\leq \sum_{k=0}^{n-1} P\left( \sup_{t\leq T} Y^{(k)}(t) \geq \frac{v}{n} \mid Y^{(k)}(0) = \frac{u}{n} \right) \leq 2n \exp\left\{ - \frac{(v-u)^2}{8n v T} \right\}.
\end{align*}
Since $n= \lceil \dim \rceil$ this gives 
 \[
  P\left( \sup_{t\leq T} Y(t) \geq v \mid Y(0)=u \right)
  \leq 2\lceil \dim \rceil e^{-\frac{(v-u)^2}{8\lceil \dim \rceil v T}}. 
 \]
 
Now we give the proof of the left exit probabilities in \eqref{leftexit}, again dividing the proof into several cases. 

\textbf{Case 1:} $\dim=1$. In this case, since $Y(t) = (B(t))^2$,  we have
\begin{align*}
 P\left( \inf_{t\leq T} Y(t) \leq v \mid Y(0) = u \right)
 &= P\left( \inf_{t\leq T} |B(t)| \leq \sqrt{v} \mid B(0) = \sqrt{u} \right) \\
 &= P\left( \inf_{t\leq T} B(t) \leq \sqrt{v} \mid B(0) = \sqrt{u} \right) \\
 &= P\left( \sup_{t\leq T} B(t) \geq \sqrt{u}-\sqrt{v} \mid B(0) = 0 \right) 
 \leq e^{-\frac{(\sqrt{u}-\sqrt{v})^2}{2T} } \leq e^{-\frac{(u-v)^2}{8uT} }.
\end{align*}

\textbf{Case 2:} $\dim > 1$.  
By the Ikeda-Watanabe comparison principle, there is a coupling of $Y\in\rm{BESQ}^\dim$ with $Z\in\rm{BESQ}^1$ with $Y(0) = Z(0)$ and such that $Y(t) \geq Z(t)$ for all $t\geq 0$. Thus, it follows from Case 1 that 
\[
 P\left( \inf_{t\leq T} Y(t) \leq v \mid Y(0) = u \right)
 \leq 
 P\left( \inf_{t\leq T} Z(t) \leq v \mid Z(0) = u \right)
 \leq 
 e^{-\frac{(u-v)^2}{8uT} }.
\]

\textbf{Case 3:} $\dim \in [0,1)$. 
We can let $Z(t) = Y(t)+X(t)$ where $X\in\rm{BESQ}^1$ is independent of $Y$ and $X(0)=0$. Then $Z\in\rm{BESQ}^{(1+\dim)}$ with $Z(0)=u$ and 
\begin{align*}
 &P\left(\inf_{t\leq T} Y(t) \leq v \mid Y(0)=u \right) 
 = P\left( \inf_{t\leq T} (Z(t)-X(t)) \leq v \mid Z(0)=u, \, X(0) = 0 \right) \\
 &\hspace{1in} \leq P\left( \inf_{t\leq T} Z(t) \leq \frac{u+v}{2} \mid Z(0) = u \right) + P\left( \sup_{t\leq T} X(t) \geq \frac{u-v}{2} \mid X(0) = 0 \right) \\
 &\hspace{1in} \leq \exp\left\{ - \frac{\left( u - \frac{u+v}{2} \right)^2}{8 u T} \right\} + 2 \exp\left\{ - \frac{u-v}{4T} \right\}
 = e^{-\frac{(u-v)^2}{32 u T} } + 2 e^{-\frac{u-v}{4T}} \leq 3 e^{-\frac{(u-v)^2}{32 u T}}.
\end{align*}

\textbf{Case 4:} $\dim < 0$. For $x\ge 0$ define
$\sigma_x(Y)=\inf\{t\ge 0: Y(t)\le x\}$ and
$\tau_x(Y)=\inf\{t\ge 0: Y(t)\ge x\}$. Due to the absorbing
  boundary condition, $Y(t)=0$ on the set $\{\sigma_0(Y)\le t\}$.
Fix $0 \leq v < u < w$
(we will choose a particular value of $w$ later) 
and note that
\begin{align}
 P\left( \inf_{t\leq T} Y(t) \leq v \mid Y(0) = u \right) &= P\left( \sigma_v(Y) \leq T \mid Y(0) = u \right) \nonumber \\
 &\leq P\left( \tau_w(Y) \leq T \mid Y(0) = u \right) + P\left( \sigma_v(Y) \leq T \wedge \tau_w(Y) \mid Y(0) = u\right) \nonumber \\
 &\leq 2 e^{-\frac{(w-u)^2}{8wT}  } + P\left( \sigma_v(Y) \leq T \wedge \tau_w(Y) \mid Y(0) = u\right), \label{highlow}
\end{align}
where in the last inequality we used the upper bound for the right
tail estimates derived above. To obtain bounds on the probability in
the last line above, we will construct an exponential martingale
  and use Doob's maximal inequality. For $\theta < 0$, let
\[
 Z_\theta(t) = \exp\left(\theta (Y(t))^{1-\frac{\dim}{2}} - \frac{\theta^2(2-\dim)^2}{2} \int_0^t (Y(s))^{1-\dim} \, ds \right).
\]
The process $(Z_\theta(t))_{t\ge 0}$ is a bounded local
  martingale (relative to the natural filtration of $Y$), so it is a martingale.  Note that
  $Z_\theta(0)=\exp\left(\theta u^{1-\frac{\dim}{2}}\right)$ and on
the event $\sigma_v(Y) \leq  T\wedge \tau_w(Y)$ we have that
\[
 \sup_{t\le T}Z_\theta(t)\ge Z_\theta(\sigma_v(Y)) \ge
 \exp\left(\theta v^{1-\frac{\dim}{2}} - \frac{\theta^2(2-\dim)^2}{2} w^{1-\dim} T\right).
\] 
The above inclusion and Doob's maximal inequality imply that
\begin{multline*}
 P\left( \sigma_v(Y) \leq T \wedge \tau_w(Y) \mid Y(0) = u\right) 
 \leq P\left( \sup_{t \leq T} Z_\theta(t) \geq e^{\theta v^{1-\frac{\dim}{2}} - \frac{\theta^2(2-\dim)^2}{2} w^{1-\dim} T} \mid Y(0)=u \right) \\
 \leq e^{\frac{\theta^2}{2} (2-\dim)^2 w^{1-\dim}T - \theta v^{1-\frac{\dim}{2}} } E\left[Z_\theta(0) \mid Y(0)=u \right]= e^{\frac{\theta^2}{2} (2-\dim)^2 w^{1-\dim}T - \theta \left(  v^{1-\frac{\dim}{2}} - u^{1-\frac{\dim}{2}} \right)  }.
\end{multline*}

Optimizing the above bound by choosing $\theta = \frac{v^{1-\frac{\dim}{2}} - u^{1-\frac{\dim}{2}}}{(2-\dim)^2 w^{1-\dim}T}$ we get the estimate
\begin{align*}
 P\left( \sigma_v(Y) \leq T \wedge \tau_w(Y) \mid Y(0) = u\right) 
 \leq \exp\left\{ - \frac{( u^{1-\frac{\dim}{2}} - v^{1-\frac{\dim}{2}})^2 }{2(2-\dim)^2 w^{1-\dim} T}  \right\}
 \leq \exp\left\{ - \frac{ (u-v)^2}{2(2-\dim)^2 u^\dim w^{1-\dim} T}  \right\},
\end{align*}
where in the last inequality we used that
  $\frac{u^{1-\frac{\dim}{2}} - v^{1-\frac{\dim}{2}} }{u-v}$ is
  increasing in $v$, as $1-\frac{\dim}{2}>1$, so
$( u^{1-\frac{\dim}{2}} - v^{1-\frac{\dim}{2}})^2\geq (u-v)^2
u^{-\dim}$ for all $0\leq v<u$.  Combining this with \eqref{highlow}
we get the bound
\[
  P\left( \inf_{t\leq T} Y(t) \leq v \mid Y(0) = u \right)
 \leq 2 \exp\left\{-\frac{(w-u)^2}{8wT} \right\} + 
 \exp\left\{ - \frac{ (u-v)^2}{2(2-\dim)^2 u^\dim w^{1-\dim} T}  \right\},
\]
which holds for all $w> u$. We let $w = 2u-v$ and finally see that
\begin{align*}
  P\left( \inf_{t\leq T} Y(t) \leq v \mid Y(0) = u \right)
  &\leq 2 \exp\left\{ - \frac{(u-v)^2}{8(2u-v) T} \right\} + \exp\left\{ - \frac{ (u-v)^2}{2(2-\dim)^2 u^\dim (2u-v)^{1-\dim} T}  \right\} \\
  &\leq 2 \exp\left\{ - \frac{(u-v)^2}{16 u T} \right\} + \exp\left\{ - \frac{ (u-v)^2}{2^{2-\dim}(2-\dim)^2 u T}  \right\} \\
  &\leq 3 \exp\left\{ - \frac{ (u-v)^2}{2^{2-\dim}(2-\dim)^2 u T}  \right\}.
\end{align*}
The last inequality follows from the fact that $2^{2-\dim}(2-\dim)^2 > 16$ for all $\dim< 0$.
\end{proof} 

\begin{lemma}\label{cor:BESQpowerdev}
 Suppose that $Y\in\rm{BESQ}^\dim$ which is reflected at $0$ if $\dim >0$ and absorbed at $0$ if $\dim \leq 0$ and that $Z(t) = A (Y(t))^\gamma$ for some $A,\gamma>0$. Then, for any $r>0$ there exist constants $C>0$ depending only on $\dim$ and $c_r>0$ depending only on $\dim$, $A$, $\gamma$, and $r$, such that
 \[
  \sup_{z\leq L} P\left( \sup_{t\leq T} |Z(t) - z| \geq r L \mid Z(0) = z \right) \leq C e^{-c_r L^{\frac{1}{\gamma}}/T}, \quad \forall\, L,T>0.
 \]
\end{lemma}

\begin{proof}
 Assume first that $z\geq rL$. Then, it follows from Proposition \ref{prop:BESQdev} that 
 \begin{align}
  &P\left( \sup_{t\leq T} |Z(t) - z| \geq r L \mid Z(0) = z \right) \nonumber \\
  &\qquad \leq P\left( \sup_{t\leq T} Z(t) \geq z+rL \mid Z(0) = z \right) 
  + P\left( \inf_{t\leq T} Z(t) \leq z-rL \mid Z(0) = z \right) \label{devupdown}  \\
  &\qquad \leq C \exp\left\{ -c  \frac{\left( (z+rL)^{\frac{1}{\gamma}} - z^{\frac{1}{\gamma}}  \right)^2}{A^{\frac{1}{\gamma}} (z+rL)^{\frac{1}{\gamma}} T } \right\}
  + C \exp\left\{ - c \frac{\left( z^{\frac{1}{\gamma}} - (z-rL)^{\frac{1}{\gamma}}\right)^2}{ A^{\frac{1}{\gamma}} z^{\frac{1}{\gamma}} T }  \right\} \nonumber \\
  &\qquad = C \exp\left\{ -c  \frac{\left( (\frac{z}{L}+r)^{\frac{1}{\gamma}} - \left( \frac{z}{L} \right)^{\frac{1}{\gamma}}  \right)^2}{A^{\frac{1}{\gamma}} (\frac{z}{L}+r)^{\frac{1}{\gamma}}  } \frac{L^{\frac{1}{\gamma}}}{T} \right\}
  + C \exp\left\{ - c \frac{\left( \left( \frac{z}{L} \right)^{\frac{1}{\gamma}} - (\frac{z}{L}-r)^{\frac{1}{\gamma}}\right)^2}{ A^{\frac{1}{\gamma}} \left( \frac{z}{L} \right)^{\frac{1}{\gamma}}  } \frac{L^{\frac{1}{\gamma}}}{T}  \right\} \nonumber \\
  &\qquad \leq 2C \exp\left\{ -c \inf_{u\in[0,1]} \frac{\left( (u+r)^{\frac{1}{\gamma}}-u^{\frac{1}{\gamma}}\right)^2}{A^{\frac{1}{\gamma}}(u+r)^{\frac{1}{\gamma}}}   \frac{L^{\frac{1}{\gamma}}}{T}  \right\}, \nonumber 
 \end{align}
where the last inequality holds when $z \leq L$. When $z < rL$ the same argument applies with the exception that the second probability in \eqref{devupdown} is equal to zero.
\end{proof}

\begin{proof}[Proof of Corollary~\ref{cor:Lxh-dev}]
 Recall that by Assumption \ref{asm:Bessel} the curve $\Lambda_{x,h}(\cdot)$ is a gluing together of multiples of BESQ processes of different dimensions raised to the power $\gamma$. Thus, the inequality in \eqref{Lxh-dev} follows almost immediately from Lemma \ref{cor:BESQpowerdev}, but we need to account for the gluing together of different processes at the origin when the interval $(x-T,x+T)$ contains the origin. To this end, suppose that $x \in (-T,0)$ (the case $x \in (0,T)$ is similar) and $h\leq L$ and note that
 \begin{align*}
  &P\left( \sup_{y \in [x,x+T]} |\Lambda_{x,h}(y)-h| \geq rL \right) \\
  &\qquad \leq  P\left( \sup_{y \in [x,0]} |\Lambda_{x,h}(y)-h| \geq \frac{rL}{2} \right) + 
  \sup_{h': |h'-h|\leq \frac{rL}{2}}  P\left( \sup_{y \in [0,x+T]} |\Lambda_{0,h'}(y)-h'| \geq \frac{rL}{2} \right).
 \end{align*}
For the second probability on the right, note that $|h'-h|\leq \frac{rL}{2}$ and $h\leq L$ implies that $h' \leq L + \frac{rL}{2}$ and re-write $\frac{rL}{2} = \frac{r}{2+r}(L+\frac{rL}{2})$. Then it follows from Lemma \ref{cor:BESQpowerdev} that 
\begin{align*}
 P\left( \sup_{y \in [x,x+T]} |\Lambda_{x,h}(y)-h| \geq rL \right)
 &\leq C \exp\left\{ - c_{r/2} \frac{L^{\frac{1}{\gamma}}}{|x|} \right\}
 + C \exp\left\{ -c_{r/(2+r)} \frac{(L+\frac{rL}{2} )^{\frac{1}{\gamma}} }{T-|x|} \right\} \\
 &\leq  C \exp\left\{ - c_{r/2} \frac{L^{\frac{1}{\gamma}}}{T} \right\}
 + C \exp\left\{ -c_{r/(2+r)} \left( \frac{2+r}{2} \right)^{\frac{1}{\gamma}} \frac{L^{\frac{1}{\gamma}} }{T} \right\}. 
\end{align*}
This completes the proof of \eqref{Lxh-dev}.

Next we will show how \eqref{Lgap} follows from \eqref{Lxh-dev}. We will prove this bound differently depending on whether $h \leq \frac{rL}{3}$ or $h > \frac{rL}{3}$. Recall that in \eqref{Lgap} we assume that $\frac{rL}{3}\geq 1$. 

\noindent\textbf{Case I: $h \leq \frac{rL}{3}$.} 
Note that in this case we have $h+1 \leq \frac{2rL}{3}$.  Since
$\Lambda_{x,h}(y) \geq 0$, in this case we have
\begin{align*}
 P\left( \sup_{|y-x| \leq T} ( \Lambda_{x,h+1}(y) - \Lambda_{x,h}(y) ) \geq rL \right)
 &\leq P\left( \sup_{|y-x| \leq T} \Lambda_{x,h+1}(y) \geq rL \right) \\
 &\leq P\left( \sup_{|y-x| \leq T} \left|\Lambda_{x,h+1}(y) - (h+1) \right| \geq \frac{rL}{3} \right) \\
 &\leq C \exp\left\{ -c'_{1/2} \frac{\left( \frac{2rL}{3} \right)^{\frac{1}{\gamma}} }{T} \right\}
 = C \exp\left\{ -c'_{1/2} \left( \frac{2r}{3} \right)^{\frac{1}{\gamma}} \frac{L^{\frac{1}{\gamma}}}{T} \right\}. 
\end{align*}

\noindent\textbf{Case II: $h > \frac{rL}{3}$.} In this case we first note that 
\begin{align*}
 &P\left( \sup_{|y-x| \leq T} ( \Lambda_{x,h+1}(y) - \Lambda_{x,h}(y) ) \geq rL \right) \\
 &\qquad \leq P\left( \sup_{|y-x|\leq T} \Lambda_{x,h+1}(y) \geq h + \frac{2rL}{3} \right) + P\left( \inf_{|y-x| \leq T} \Lambda_{x,h}(y) \leq h - \frac{rL}{3} \right) \\
 &\qquad \leq P\left( \sup_{|y-x|\leq T} \left| \Lambda_{x,h+1}(y) - (h+1) \right| \geq \frac{rL}{3} \right) + P\left( \sup_{|y-x| \leq T} \left| \Lambda_{x,h}(y) - h \right| \geq \frac{rL}{3} \right)
\end{align*}
To apply \eqref{Lxh-dev} to the first probability above, note that if $h \leq L$ and $\frac{Lr}{3}\geq 1$ then $h+1 \leq L + \frac{Lr}{3}$, and rewriting $\frac{rL}{3} = \frac{r}{3+r} (L+\frac{rL}{3})$ we have 
\begin{align*}
 P\left( \sup_{|y-x| \leq T} ( \Lambda_{x,h+1}(y) - \Lambda_{x,h}(y) ) \geq rL \right)
 &\leq C\exp\left\{ -c_{r/(3+r)}' \frac{(L+\frac{rL}{3})^{\frac{1}{\gamma}} }{T} \right\} + C\exp\left\{ -c_{r/3}' \frac{L^{\frac{1}{\gamma}}}{T} \right\}.
\end{align*}
\end{proof}

\subsection{Continuity in law properties}\label{app:contprop}

The goal of this subsection is to obtain properties of continuity in law for the random variables $\Lambda_{x,h}(y)$ with respect to $x$, $h$ and $y$. Throughout this subsection we assume that the joint Ray--Knight
theorem of Definition \ref{def:JRK} holds.  That is, for every finite family
$(x_1,h_1),\dots,(x_N,h_N)\in\R\times\R_{+}$ the
limiting curves are obtained as the weak limit
\[
   \bigl(\Lambda^{n}_{x_i,h_i}\bigr)_{i\le N}\ \Longrightarrow\ \bigl(\Lambda_{x_i,h_i}\bigr)_{i\le N}
   \qquad\text{as } n\to\infty
\]
of the rescaled local-time curves.  Two consequences of this
convergence are used below, both inherited from the prelimiting local
time curves. First, the joint laws of finite families of curves
$\Lambda_{x,h}$ are consistent. Second, these curves satisfy the
non-crossing property \eqref{non-cross}. Note that we only need the
joint Ray-Knight theorem, the stronger augmented version is not used. Moreover, the results in this section also hold if we replace the joint Ray--Knight
theorem of Definition \ref{def:JRK} by the ``edge'' joint Ray--Knight theorem \eqref{Yu-JRK}.
We will also assume that the limiting curves satisfy Assumption
\ref{asm:Bessel}.  The positive multiplicative constants $C_{r,\pm}$
and $C_{a,\pm}$ of Assumption~\ref{asm:Bessel} play no role in what
follows. We therefore set $C_{r,\pm}=C_{a,\pm}=1$, so that each part
of $\Lambda_{x,h}(\cdot)$ is simply a power $Y(\cdot)^{\gamma}$ of the
corresponding BESQ process.

\begin{remark}\label{rem:reduction}
  By Assumption~\ref{asm:Bessel} each curve is a concatenation of
  powers of BESQ processes. There is a symmetry between backward and
  forward curves. Hence, in the proofs below we treat only the case
  when the curve crosses the origin from left to right. All other
  cases are either symmetric or simpler.
\end{remark}

The next lemma collects some standard facts about BESQ processes. We refer the reader to \cite[Ch.~XI]{ryCMBM}, \cite{gySGBP}, \cite[Ch.~IV and VI]{IW89}.
\begin{lemma}\label{lem:besq}
Fix $\nu\in\R$ and $\gamma>0$ and let $\{Y^{(u)}\}_{u\ge0}$ be a family
of \textup{BESQ} processes of dimension $\nu$ with $Y^{(u)}(0)=u$, all driven by
the same Brownian motion and subject to the same boundary condition at
$0$.\footnote{For $0<\nu<2$ the boundary is reached, and the same condition, reflecting or
absorbing at $0$, is imposed for all $u\ge0$. For $\nu\ge2$ the boundary is inaccessible
from $u>0$ (and reflecting at $u=0$). For $\nu\le0$ the boundary is absorbing.} Then
\begin{enumerate}
\item[\textup{(i)}] for every $s\ge0$, $0 \leq u \leq u'$, almost surely $Y^{(u)}(s) \leq Y^{(u')}(s)$, and if $(u_n)_{n\in\mathbb{Z}^+}$ is such that $u_n \to u$ as $n \to \infty$, then $Y^{(u_n)}(s)\to Y^{(u)}(s)$ almost surely as $n \to \infty$;
\item[\textup{(ii)}] the map $(u,s)\mapsto E[(Y^{(u)}(s))^{\gamma}]$ is nondecreasing in
$u$ and continuous on $[0,\infty)\times[0,\infty)$.
\end{enumerate}
\end{lemma}

Lemma \ref{lem:besq} has the following consequence for the limiting curves from the Ray-Knight theorems.

\begin{lemma}\label{lem:curvemean}
The map $(x_0,h,y)\mapsto E[\Lambda_{x_0,h}(y)]$ is nondecreasing in $h\in[0,\infty)$ and is
continuous on $\R\times[0,\infty)\times\R$.
\end{lemma}

\begin{proof}
  Monotonicity in $h$ follows from the non-crossing property.  For
  continuity we use Assumption~\ref{asm:Bessel} and
  Lemma~\ref{lem:besq}. By Remark~\ref{rem:reduction} we treat only
  the forward curve with $x_0<0$. For $u,s \geq 0$, write
\[
   m_r(u,s):=E[(Y^{(u)}(s))^{\gamma}],\qquad
   m_a(u,s):=E[(\widetilde Y^{(u)}(s))^{\gamma}],
\]
where $Y^{(u)}$ is a reflecting $\mathrm{BESQ}^{\delta_{r,+}}$ process with $Y^{(u)}(0)=u$ and
$\widetilde Y^{(u)}$ an absorbing $\mathrm{BESQ}^{\delta_{a,+}}$ process with $\widetilde Y^{(u)}(0)=u$. By
Lemma~\ref{lem:besq}(ii), applied to the dimensions $\delta_{r,+}$ and
$\delta_{a,+}$, each of $m_r$ and $m_a$ is nondecreasing in $u$ and
jointly continuous in $(u,s)$ on $[0,\infty)\times[0,\infty)$.

\emph{Reflection at $0$: $x_0\le y\le 0$.} By Assumption \ref{asm:Bessel}, the law of $(\Lambda_{x_0,h}(x_0+s))_{0\le s\le|x_0|}$
does not depend on $x_0$, so
$E[\Lambda_{x_0,h}(y)]=m_r(h^{1/\gamma},\,y-x_0)$. Composing the
jointly continuous $m_r$ with the continuous maps
$h\mapsto h^{1/\gamma}$ and $(x_0,y)\mapsto y-x_0$ shows that
$E[\Lambda_{x_0,h}(y)]$ is jointly continuous in $(x_0,h,y)$ on this
piece.

\emph{Absorption at $0$: $y\ge0$.} By the Markov property of $\Lambda_{x_0,h}$, the curve to the right of $0$ is a power of a $\mathrm{BESQ}^{\delta_{a,+}}$ process
started at the random value $\Lambda_{x_0,h}(0)$ and absorbed at $0$. Conditioning on
$\Lambda_{x_0,h}(0)$ gives
\[
   E[\Lambda_{x_0,h}(y)]=E\Bigl[m_a\bigl((\Lambda_{x_0,h}(0))^{1/\gamma},\,y\bigr)\Bigr].
\]
By Assumption \ref{asm:Bessel},
$\Lambda_{x_0,h}(0)\overset{\mathrm{Law}}{=}(Y^{(h^{1/\gamma})}(|x_0|))^{\gamma}$, and its law converges
weakly as $(x_0',h')\to(x_0,h)$, by Lemma~\ref{lem:besq}(i) and the
continuity of $Y$ in time. To pass to the limit in the expectation above, we notice that by the Skorohod Representation Theorem we can assume the convergence of $\Lambda_{x_0',h'}(0)$ to $\Lambda_{x_0,h}(0)$ is almost sure, so the continuity of $m_a$ proven in Lemma~\ref{lem:besq}(ii) implies $m_a((\Lambda_{x_0',h'}(0))^{1/\gamma},y')$ converges almost surely to $m_a((\Lambda_{x_0,h}(0))^{1/\gamma},y)$ as $(x_0',h',y') \to (x_0,h,y)$. Moreover, Proposition \ref{prop:BESQdev} yields the $m_a((\Lambda_{x_0',h'}(0))^{1/\gamma},y')$ for $x_0-1 \leq x_0' \leq 0$, $0 \leq h' \leq h+1$, $0 \leq y' \leq y+1$ are bounded in $L^{(\gamma+1)/\gamma}$, hence uniformly integrable. This implies $E[m_a((\Lambda_{x_0',h'}(0))^{1/\gamma},y')]$ converges to $E[m_a((\Lambda_{x_0,h}(0))^{1/\gamma},y)]$ as $(x_0',h',y') \to (x_0,h,y)$, so
jointly continuous in $(x_0,h,y)$.
\end{proof}

We now prove that the finite-dimensional distributions of the curves
depend continuously on the base points and the evaluation points. This
is all that the rescaling identity \eqref{Lxh-rescale} requires:
\eqref{Lxh-rescale} is an equality of laws on $C(\R)^{k}$, and
a law on $C(\R)^{k}$ is determined by its finite-dimensional
distributions, so it suffices to upgrade the rational case to
arbitrary $c$ at the level of finite-dimensional distributions. The
ingredients are the non-crossing property, the almost sure continuity
of a single curve, and the continuity of the mean
(Lemma~\ref{lem:curvemean}), used via Markov
inequality.

\begin{proposition}\label{prop:fdd}
Suppose that the joint Ray--Knight theorem or \eqref{Yu-JRK} holds, and that the families of limiting curves satisfy Assumption~\textup{\ref{asm:Bessel}}.
Then for every $k,m\in\mathbb{N}$ the map
\[
   \bigl((x_1,h_1),\dots,(x_k,h_k);\,y_1,\dots,y_m\bigr)\ \longmapsto\
   \mathrm{Law}\Bigl(\bigl(\Lambda_{x_i,h_i}(y_j)\bigr)_{i\le k,\,j\le m}\Bigr)
\]
is continuous from $(\R\times\R_{+})^{k}\times\R^{m}$ to the space of
Borel probability measures on $\R^{km}$ with the topology of weak convergence.
\end{proposition}

\begin{proof}
Weak convergence on $\R^{km}$ is metrizable, so it suffices to fix sequences
$(x_i^{(n)},h_i^{(n)})\to(x_i,h_i)$ for $i\le k$ and $y_j^{(n)}\to y_j$ for $j\le m$, and to
show
\[
   \bigl(\Lambda_{x_i^{(n)},h_i^{(n)}}(y_j^{(n)})\bigr)_{i\le k,\,j\le m}
   \ \Longrightarrow\
   \bigl(\Lambda_{x_i,h_i}(y_j)\bigr)_{i\le k,\,j\le m}
   \quad \text{as }n\to\infty.
\]

\emph{Single probability space.} Fix a sequence $\eta_\ell\downarrow0$
in $(0,1]$.  Define a countable set $J\subset \R\times\R_+$ by
\[
   J:=\left\{(x_i,h_i),(x_i^{(n)},h_i^{(n)}), (x_i^{(n)},\,\max\{0,h_i-\eta_\ell\}),
   (x_i^{(n)},\,h_i+\eta_\ell)\right\}_{1\le i\le k;\, \ell,n\ge 1}.
\]
By the joint Ray--Knight theorem, or by \eqref{Yu-JRK}, the joint distributions of finite
subfamilies of $\{\Lambda_{x,h}\}_{(x,h)\in J}$ are consistent, so
Kolmogorov's extension theorem allows us to realize
$(\Lambda_{x,h})_{(x,h)\in J}$ on a single probability space. On it
the non-crossing property holds for all pairs from $J$ at once.

\emph{Reduction to one coordinate.} Convergence in probability of a
random vector in $\R^{km}$ follows from convergence in
probability of each coordinate, so it suffices to prove, for each
fixed $i\le k$ and $j\le m$,
\[
   \Lambda_{x_i^{(n)},h_i^{(n)}}(y_j^{(n)})\ \xrightarrow[n\to\infty]{P}\ \Lambda_{x_i,h_i}(y_j).
\]
Fix such $i,j$ and abbreviate $(x,h):=(x_i,h_i)$, $y:=y_j$,
$(x',h'):=(x_i^{(n)},h_i^{(n)})$, $y':=y_j^{(n)}$; all curves below
are indexed by elements of $J$. Write
\[
   \Lambda_{x',h'}(y')-\Lambda_{x,h}(y)
   =\underbrace{\bigl(\Lambda_{x',h'}(y')-\Lambda_{x,h}(y')\bigr)}_{(\mathrm I)}
   +\underbrace{\bigl(\Lambda_{x,h}(y')-\Lambda_{x,h}(y)\bigr)}_{(\mathrm{II})}.
\]

Term $(\mathrm{II})$ goes to $0$ almost surely, by continuity of a
single curve $\Lambda_{x,h}$ (Assumption~\ref{asm:Bessel}), hence it is enough to prove that $(\mathrm{I})$ converges in probability to $0$.

\emph{Case $h>0$.} Fix $\ell$ large enough to have $h-\eta_\ell>0$ and write $\eta:=\eta_\ell$,
$b_-:=h-\eta$, $b_+:=h+\eta$; these are deterministic,
$b_+-b_- = 2\eta$, and the indices of curves $\Lambda_{x',b_\pm}$ belong to
$J$. Since $\Lambda_{x,h}$ is almost surely continuous with
$\Lambda_{x,h}(x)=h$ and $x'\to x$, we have $\Lambda_{x,h}(x')\to h$
almost surely, so
\[
   A^{(n)}:=\bigl\{\,|\Lambda_{x,h}(x')-h| < \eta\,\bigr\}\quad\text{satisfies}\quad
   P\bigl((A^{(n)})^{c}\bigr)\xrightarrow[n\to\infty]{}0 .
\]
Also $|h'-h| < \eta$ for large $n$. On $A^{(n)}\cap\{|h'-h| < \eta\}$
the three numbers $h$, $h'$, and $\Lambda_{x,h}(x')$ all lie in
$(b_-,b_+)$. Hence, at $x'$, both $\Lambda_{x',h'}$ and $\Lambda_{x,h}$
are sandwiched between $\Lambda_{x',b_-}$ and $\Lambda_{x',b_+}$. By
the non-crossing property, $\Lambda_{x',h'}(y')$ and $\Lambda_{x,h}(y')$
lie in $[\Lambda_{x',b_-}(y'),\Lambda_{x',b_+}(y')]$ for all $y'$.
Thus on
$A^{(n)}\cap\{|h'-h|\le\eta\}$,
\[
   |(\mathrm I)|=\bigl|\Lambda_{x',h'}(y')-\Lambda_{x,h}(y')\bigr|
   \ \le\ \Lambda_{x',b_+}(y')-\Lambda_{x',b_-}(y')\ =:\ G^{(n)}(y') .
\]
Note that $G^{(n)}(y')\ge 0$ and  $b_\pm$ are
deterministic, so by Markov inequality, for any $\e>0$,
\[
   P\bigl(G^{(n)}(y')>\varepsilon\bigr)\ \le\ \frac1\varepsilon\,E\bigl[G^{(n)}(y')\bigr]
   \ =\ \frac1\varepsilon\Bigl(E[\Lambda_{x',b_+}(y')]-E[\Lambda_{x',b_-}(y')]\Bigr).
\]
As $n\to\infty$, $x'\to x$ and $y'\to y$ while $b_\pm$ stay fixed, so by the joint continuity part of Lemma~\ref{lem:curvemean},
\[
  \limsup_{n\to\infty}P\bigl(G^{(n)}(y')>\varepsilon\bigr) \ \le\
  \frac1\varepsilon\Bigl(E[\Lambda_{x,b_+}(y)]-E[\Lambda_{x,b_-}(y)]\Bigr).
\]
We can now conclude for $h>0$. For every $\varepsilon>0$ and every $\ell$,
\[
   P\bigl(|(\mathrm I)|>\varepsilon\bigr)\ \le\ P\bigl((A^{(n)})^{c}\bigr)+
  \ind{|h'-h|\geq \eta}
   +P\bigl(G^{(n)}(y')>\varepsilon\bigr).
\]
Letting first $n\to\infty$, then $\ell\to\infty$ so that
$\eta_\ell\downarrow0$ and $b_\pm\to h$, we see that the right-hand side
tends to $0$ by continuity in $h$ (Lemma~\ref{lem:curvemean}). Hence
$(\mathrm I)\to0$ in probability, which is enough.

\emph{Case $h=0$.}
We notice that
\[
 (I) = \Lambda_{x',h'}(y')-\Lambda_{x',0}(y')+\Lambda_{x',0}(y')-\Lambda_{x,0}(y').
\]
By the non-crossing property, $\Lambda_{x',h'}(y')-\Lambda_{x',0}(y') \geq 0$, hence for any $\e>0$, the Markov inequality and Lemma \ref{lem:curvemean} yield
\[
 P\left(\Lambda_{x',h'}(y')-\Lambda_{x',0}(y')>\frac{\e}{2}\right) \leq \frac{2}{\e}\left(E[\Lambda_{x',h'}(y')]-E[\Lambda_{x',0}(y')]\right) \underset{n\to \infty}{\longrightarrow} 0.
\]
When $n$ is large enough, $x'$ has the sign of $x$; we consider such $n$. If $x=x'$ we have $\Lambda_{x',0}(y')=\Lambda_{x,0}(y')$. If $x \neq x'$ and $x'$ is closer to $0$ than $x$, then almost surely $\Lambda_{x,0}(x')>0=\Lambda_{x',0}(x')$, hence by the non-crossing property, $\Lambda_{x,0}(y')\geq\Lambda_{x',0}(y')$. Conversely, if $x$ is closer to $0$ than $x'$, then $\Lambda_{x,0}(y')\leq\Lambda_{x',0}(y')$ almost surely. In all cases, for any $\e>0$, the Markov inequality and Lemma \ref{lem:curvemean} imply
\[
 P\left(|\Lambda_{x,0}(y')-\Lambda_{x',0}(y')|>\frac{\e}{2}\right) \leq \frac{2}{\e}\left|E[\Lambda_{x,0}(y')]-E[\Lambda_{x',0}(y')]\right| \underset{n\to \infty}{\longrightarrow} 0.
\]
Consequently, $(\mathrm I)$ converges in probability to $0$, which is enough.
\end{proof}

\subsection{Integrals of powers of BESQ processes}\label{app:integrals}

In the proof of our main results, we use the fact that
distributions of $\int_\R \Lambda_{x,h}(y) \, dy$
have no atoms, and that if $(x,h) \neq (x',h')$, almost surely $\int_\R\Lambda_{x,h}(y)\,dy \neq \int_\R\Lambda_{x',h'}(y)\,dy$. We will prove these results in this subsection. We note that our proof only uses that the curves $\Lambda_{x,h}(\cdot)$ satisfy Assumption \ref{asm:Bessel} and have the non-crossing property \eqref{non-cross}. Since each $\Lambda_{x,h}(\cdot)$ is a concatenation of
\rm{BESQ} processes raised to the power $\gamma$, we will use the following proposition.

\begin{proposition}\label{prop_no_atom_Bessel}
  Let $Y=(Y(t))_{t\geq 0}$ be a \rm{BESQ}$^\dim$ process with
  $\dim < 2$ that is absorbed upon hitting $0$.  Then for any
  $\gamma>0$ and $x>0$, we have
 \[
  P\left( \int_0^\infty (Y(t))^{\gamma} \, dt = c \, \Bigl| \, Y(0) = x \right) = 0, \quad \forall\, c\geq 0.
 \]
\end{proposition}

\begin{proof}
  Throughout the proof $\gamma>0$ will be fixed, and for convenience
  of notation we will let
  $I^{(x)}=\int_0^\infty (Y(t))^{\gamma} \, dt $ conditioned on
  $Y(0) = x>0$.  If $Y(0) = x$, then
  $\tilde{Y}(t) = x^{-1}Y(xt)\in \rm{BESQ}^\dim$ with $\tilde{Y}(0) = 1$ and 
 \begin{equation}\label{Ix-scaling}
  I^{(x)} = \int_0^\infty (Y(t))^{\gamma} \, dt 
  = \int_0^\infty \left( x \tilde{Y}(x^{-1}t) \right)^{\gamma} \, dt
  = x^{1+\gamma} \int_0^\infty (\tilde{Y}(u))^{\gamma} \, du 
  \overset{\text{Law}}{=} x^{1+\gamma} I^{(1)}. 
 \end{equation}
 Thus,
 it will be enough to show that $I^{(1)}$ has a continuous
 distribution, and so for the remainder of the proof we will assume
 that $Y(0) = 1$.
 
 Assume that the distribution of $I^{(1)}$ has atoms, i.e., there is a
 non-empty countable set $\mathcal{A} = (a_k)_k$ such that
 $P(I^{(1)} = a) > 0$ if and only if $a \in \mathcal{A}$.  Now, for
 any $s>0$ since $(Y(s+u))_{u\geq 0}\in\rm{BESQ}^\dim$ started at
 $Y(s)$ and absorbed at $0$, then it follows from \eqref{Ix-scaling} (recall also that
 $Y(u)=0$ for all $u$ starting with the absorption time) that
\[
  I^{(1)} = \int_0^s (Y(t))^{\gamma} \, dt + \int_s^\infty
  (Y(t))^{\gamma} \, dt \overset{\text{Law}}{=} \int_0^s (Y(t))^{\gamma} \, dt
  + (Y(s))^{1+\gamma} \tilde{I}^{(1)},
\]
where $\tilde{I}^{(1)}$ has the same distribution as $I^{(1)}$ and is independent of $(Y(t))_{t\leq s}$. 
Fixing some $a_k \in \mathcal{A}$, we choose $\delta>0$ so that for any $0 \leq s \leq \delta$, $P(Y(s)=0) \leq \frac{1}{4}P(I^{(1)}=a_k)$.
For $0 < s \leq \delta$, conditioning on $\mathcal{F}_s = \sigma(Y(t), \, t\leq s)$, it follows from \eqref{Ix-scaling} that
\begin{align*}
 \frac{3}{4}P(I^{(1)}=a_k) \leq  P(I^{(1)}=a_k,\, Y(s)>0)
 &= P\left( \int_0^s (Y(t))^{\gamma} \, dt + (Y(s))^{1+\gamma} \tilde{I}^{(1)} = a_k, \, Y(s)>0 \right) \\
 &= E\left[ P\left( \tilde{I}^{(1)} = \frac{a_k - \int_0^s (Y(t))^{\gamma} \, dt}{(Y(s))^{1+\gamma}} \, \Bigl| \, \mathcal{F}_s \right) \ind{Y(s)>0} \right] \\
 &\leq E\left[ \sum_j P( \tilde{I}^{(1)} = a_j ) \ind{ a_k - \int_0^s (Y(t))^{\gamma} \, dt= a_j (Y(s))^{1+\gamma}} \right] \\
 &= \sum_j P(I^{(1)} = a_j) P\left( \int_0^s (Y(t))^{\gamma} \, dt + (Y(s))^{1+\gamma} a_j = a_k \right). 
\end{align*}
To control the last probability in the line above, if $j=k$ then since $\int_0^s (Y(t))^{\gamma}\, dt > 0$ it follows that 
$P\left( \int_0^s (Y(t))^{\gamma} \, dt + (Y(s))^{1+\gamma} a_k = a_k \right) \leq P( Y(s) < 1 )$.
Therefore, 
\begin{multline}\label{Ix-ineq}
 P( I^{(1)}=a_k )\left(\frac{3}{4}-P(Y(s) < 1)\right)
 \\\leq \sum_{j\neq k} P(I^{(1)} = a_j) P\left( \int_0^s (Y(t))^{\gamma} \, dt + (Y(s))^{1+\gamma} a_j = a_k \right). 
\end{multline}
Since almost surely $\lim_{s\to 0^+} \left( \int_0^s (Y(t))^{\gamma} \, dt + (Y(s))^{1+\gamma} a_j \right)  = a_j$ (recall that $Y(0) = 1$), it follows that all of the terms in the sum on the right vanish as $s\to 0^+$ and thus
the dominated convergence theorem implies that the right side of
\eqref{Ix-ineq} vanishes as $s\to 0^+$. On the other hand, it can be
shown that $\lim_{s\to 0^+} P(Y(s) < 1) = 1/2$. Indeed, for all
$s \leq \inf\{t \geq 0 : |Y(t)-1| \geq 1/2\}$, we have
$\sqrt{Y(s)}-1=W(s)+\frac12(\mathfrak{d}-1)\int_0^s 1/\sqrt{Y(t)}\,dt$
where $W$ is a Brownian motion. Hence
$W(s)+b_-s\leq \sqrt{Y(s)}-1 \leq W(s)+b_+s$ for some
$b_-,b_+ \in \mathbb{R}$. Thus, \eqref{Ix-ineq} implies that
$P(I^{(1)} = a_k)\leq 0$, contradicting our assumption that the
distribution of $I^{(1)}$ has a set of atoms $\mathcal{A} = \{a_k\}$.
\end{proof}

\begin{proposition}\label{prop_int_no_atoms}
 Under Assumption \ref{asm:Bessel}, for any $x\in\R$, $h \geq 0$ with $(x,h) \neq (0,0)$, $\int_\R\Lambda_{x,h}(y)dy$ has no atoms.
\end{proposition}

\begin{proof}
 Without loss of generality, we assume $x \leq 0$. If $x=0$, $\Lambda_{x,h}(0)=h>0$, while if $x<0$, $\Lambda_{x,h}(0)$ is atomless so almost surely $\Lambda_{x,h}(0)>0$. Moreover, for any $a\in\R$, \[P\left(\int_\R\Lambda_{x,h}(y)dy=a\right)=E\left[P\left(\int_0^{\infty}\Lambda_{x,h}(y)dy=a-\int_{-\infty}^{0}\Lambda_{x,h}(y)dy\,\Big|\,(\Lambda_{x,h}(y))_{y \leq 0}\right)\right]=0\] by Proposition \ref{prop_no_atom_Bessel}.
\end{proof}

\begin{lemma}\label{lem_int_diff}
 If the family of curves in the joint RKT satisfies Assumption \ref{asm:Bessel}, then for any $x,x'\in\R$, $h,h' \geq 0$ with $(x,h) \neq (x',h')$, we have $P(\int_\R\Lambda_{x,h}(y)\,dy = \int_\R\Lambda_{x',h'}(y)\,dy)=0$.
\end{lemma}

\begin{proof}
  By the non-crossing property and continuity of the curves, it is
  enough to prove there almost surely exists $y\in\R$ so that
  $\Lambda_{x,h}(y)<\Lambda_{x',h'}(y)$ or
  $\Lambda_{x,h}(y)>\Lambda_{x',h'}(y)$. If $x=x'$, then
  $\Lambda_{x,h}(x)=h \neq h'=\Lambda_{x',h'}(x)$, which is enough. We
  now assume $x<x'$ (without loss of generality). If $x' \leq 0$ or
  $h' \neq 0$, the distribution of the curves implies
  $\Lambda_{x,h}(x')$ has no atom at $h'$, which is enough. If $x'>0$
  and $h'=0$, let us assume $\Lambda_{x,h}(x')=h'=0$ (otherwise we are
  done). The time of absorption at 0 of
  $(\Lambda_{x,h}(y))_{y \geq \max(0,x)}$ is either $\max(0,x)$ (if
  $\Lambda_{x,h}(\max(0,x))=0$) or is atomless (see for
  example
  just after Remark 2.1 in \cite{HamanaMatsumoto2013}), thus has no
  atom at $x'$, so almost surely there is a
  $\tilde x \in [\max(0,x),x')$ such that $\Lambda_{x,h}(y) = 0$ for
  all $y \in [\tilde{x}, x']$. But then, since
  $(\Lambda_{x',h'}(x'-s))_{0 \leq s \leq x'}$ is (a multiple of a
  power of) a reflected process, almost surely there exists
  $y\in[\tilde x,x']$ such that
  $\Lambda_{x',h'}(y)>0=\Lambda_{x,h}(y)$.
\end{proof}

\end{document}